\documentclass[10pt, a4paper]{article}

\usepackage{mathtools}
\usepackage{amsmath, amsthm, amssymb}
\usepackage{lmodern}
\usepackage[utf8]{inputenc}
\usepackage[T1]{fontenc}
\usepackage[margin=3cm]{geometry}
\usepackage{enumitem}
\usepackage{comment}
\usepackage[notcite, notref, final ]{showkeys}

\usepackage[normalem]{ulem}
\usepackage{cases}
\usepackage{hyperref}
\usepackage{bbm}
\usepackage{float}
\usepackage{tikz}

\usepackage[backend=bibtex, sorting=nyt, style=numeric, url=false, isbn=false, eprint=false, maxbibnames=4, giveninits=true]{biblatex}
\renewbibmacro{in:}{\ifentrytype{article}{}{\printtext{\bibstring{in}\intitlepunct}}}
\DeclareFieldFormat
[article,inbook,incollection,inproceedings,patent,thesis,unpublished]
{title}{#1\isdot}
\DeclareFieldFormat[article]{volume}{\mkbibbold{#1}}
\DeclareFieldFormat[article]{number}{\mkbibbold{#1}}

\newenvironment{system}[1]
{
\left\{\begin{aligned}{#1}} 
{ \end{aligned}\right.}

\newcommand\dd{\mathrm{d}}

\newtheorem{thm}{Theorem}[section]
\newtheorem{theorem}[thm]{Theorem}
\newtheorem{lem}[thm]{Lemma}
\newtheorem{lemma}[thm]{Lemma}
\newtheorem{prop}[thm]{Proposition}
\newtheorem{proposition}[thm]{Proposition}

\theoremstyle{definition}
\newtheorem{defi}{\textbf{Definition}}

\newtheorem{definition}[defi]{\textbf{Definition}}

\theoremstyle{remark}
\newtheorem{rem}[thm]{Remark}
\newtheorem{remark}[thm]{Remark}
\newtheorem{assumption}{\textbf{Assumption}}

\newcommand\parmatrix[1]{\begin{pmatrix}#1\end{pmatrix}}

\newcommand\ep{\varepsilon}
\newcommand\R{\mathbb{R}}

\newcounter{stepnum}
\newcommand\step{\typeout{Use of step outside a stepping environment ! }}
\newenvironment{stepping}{
	\setcounter{stepnum}{1}
	\renewcommand\step{{\bf Step \thestepnum:~}\stepcounter{stepnum}}
}{\renewcommand\step{\alert{Use of step outside a stepping environment !}}}

\allowdisplaybreaks

\newcommand{\writefoot}[1]{
    \renewcommand{\thefootnote}{}
    \footnotetext{\hspace{-16.5pt}\scriptsize#1}
    \renewcommand{\thefootnote}{\arabic{footnote}}
}
\numberwithin{equation}{section}

\begin{document}

\writefoot{\small \textbf{AMS subject classifications (2020).} Primary: 35K40; Secondary: 35K57, 35K58, 92D25. \smallskip}
\writefoot{\small \textbf{Keywords.} 
Reaction-diffusion system, spreading speed, anisotropic propagation, stability analysis, singular limit, Lyapunov function.\smallskip
}
\writefoot{\small \textbf{Acknowledgements:} 
The initial part of the research was conducted when QG was a JSPS International Research Fellow (FY2017; Graduate School of Mathematical Sciences, The University of Tokyo and MIMS, Meiji University). The research was partially supported by JSPS KAKENHI 16H02151 and 21H00995. QG was partially supported by a PEPS-JCJC grant from CNRS (2019) and by ANR grant “Indyana” Number ANR-21-CE40-0008.
}

\begin{center}
    \begin{minipage}{0.9\textwidth}
	\centering
    \LARGE{\bf Front propagation in hybrid reaction-diffusion epidemic models with spatial heterogeneity. \\ Part I: Spreading speed and asymptotic behavior}\bigskip
    \end{minipage}

    \Large
    Quentin Griette\medskip \\
    \normalsize
    {\it Université Le Havre Normandie, Normandie Univ., \\
	LMAH UR 3821, 76600 Le Havre, France. }\\ 
    {\tt quentin.griette@univ-lehavre.fr}
    \bigskip

    \Large
    Hiroshi Matano \medskip \\ 
    \normalsize
    {\it Meiji Institute for Advanced Study of Mathematical Sciences, Meiji University, \\
    4-21-1 Nakano, Tokyo 164-8525, Japan 
    }\\
    {\tt matano@meiji.ac.jp} \\ 
	\bigskip

	\today
\end{center}

\begin{abstract}
    We consider a two-species reaction-diffusion system in one space dimension that is derived from an epidemiological model in a spatially periodic environment with two types of pathogens: the wild type and the mutant. 
    The system is of a hybrid nature, partly cooperative and partly competitive, but neither of these entirely. 
As a result, the comparison principle does not hold 
    {for the whole system}. We study  spreading properties of solution fronts when the infection is localized initially. 
{We show that there is a well-defined spreading speed} 
{both in the right and left directions} 
{and that it can be computed from the linearized equation at the} 
{leading} 
{edge of the propagation front.} 
    Next we study the case where the coefficients are 
spatially homogeneous and show that, when spreading 
occurs, every 
solution to the Cauchy problem converges to 
the unique positive stationary solution as $t\to\infty$. 
Finally we consider the case of rapidly oscillating coefficients, that is, when the spatial period of the coefficients, denoted by $\ep$, is very small. 
We show that there exists a unique positive stationary solution, and that every positive solution to the Cauchy problem converges to this stationary solution as $t\to\infty$. We then discuss the homogenization limit as $\ep\to 0$. 
\end{abstract}

\section{Introduction}

In this paper we consider the following reaction-diffusion system:
\begin{equation}\label{eq:main-sys}
	\begin{system}
		\relax &u_t=\big(\sigma(x) u_{x}\big)_x+\big(r_u(x)-\kappa_u(x)(u+v)\big)u+\mu_v(x)v-\mu_u(x)u, & t>0, x\in\R, \\
		\relax &v_t=\big(\sigma(x) v_{x}\big)_x+\big(r_v(x)-\kappa_v(x)(u+v)\big)v+\mu_u(x)u-\mu_v(x)v, &t>0, x\in\R.
	\end{system}
\end{equation}
Here $u(t, x)$, $v(t,x) $ stand for the density of a population of individuals living in  a periodically heterogeneous environment. We assume that  the growth rates $r_u(x) $ and $ r_v(x)$ are $L$-periodic functions, and that the coefficients $\kappa_u(x)$ and $\kappa_v(x)$ which represent the intensity of the competition are $L$-periodic and positive. Finally, the coefficients $\mu_u(x)>0$, $\mu_v(x)>0$ {(also $L$-periodic)} denote the mutation rates between the two populations, {which creates an effect of}  cooperative coupling in the region where both $u$ and $v$ are small. 
We consider system \eqref{eq:main-sys} under the initial condition
\begin{equation}\label{main-initial}
u(0,x)=u_0(x),\ \ v(0,x)=v_0(x),\quad\ x\in\R ,
\end{equation}
where $u_0, v_0$ are bounded nonnegative functions on $\R $ whose supports are typically localized in some region. 
Our goal is to study the speed of propagation of fronts that appear in \eqref{eq:main-sys}. 

The above problem is motivated partly by the study of SIS epidemiological models describing the propagation of pathogens that are subject to mutations, such as the following:
\begin{equation} \label{eq:SIS}
	\left\{\begin{aligned}\relax
		&\partial_t S_t = \partial_{x} (\sigma(x)\partial_x S) - \big(\beta_1(x) I_1+\beta_2(x) I_2\big)S + \gamma_1(x) I_1 + \gamma_2(x) I_2, && t>0, x\in\R  \\
		&\partial_t I_1 = \partial_x(\sigma(x) \partial_{x}I_1) + \beta_1(x) S I_1 -\gamma_1(x) I_1 + \mu_2(x) I_2 - \mu_1(x) I_1, && t>0, x\in\R   \\ 
		&\partial_t I_2 = \partial_x(\sigma(x) \partial_{x} I_2)+\beta_2(x) S I_2 - \gamma_2(x) I_2+\mu_1(x) I_1 - \mu_2(x) I_2,  && t>0, x\in\R  . 
	\end{aligned}\right.
\end{equation}
Here the $\beta_1,\beta_2$ denote the infection rates, $\gamma_1,\gamma_2$ the recovery rates, and $\mu_1,\mu_2$ stand for the mutation rates between the pathogens.

It is not difficult to show that the quantity $N(t, x):=S(t,x)+I_1(t,x)+I_2(t,x)$ satisfies a pure diffusion equation
$\partial_t N=\partial_x\left(\sigma(x)\partial_x N\right)$. Therefore, if we assume
that $ N(0, x)$ is constant in $x$, 
then $N(t, x)$ remains constant in $t$ and $x$. 
Thus we obtain that  $u=I_1$, $v=I_2$ satisfy \eqref{eq:main-sys} with 
$r_{*}(x) := N\beta_{i}(x)-\gamma_{i}(x)$ and $\kappa_{i}(x):=\beta_{i}(x)$, where $i=1,2$ and  $*=u,v$ .  
Hence the propagation dynamics of \eqref{eq:SIS} is equivalent to that of \eqref{eq:main-sys}.

System \eqref{eq:SIS} describes the propagation of a genetically unstable pathogen in a population of hosts which exhibits heterogeneity in space. This heterogeneity represents the spatially heterogeneous environment that affects the behavior of individuals depending on where they are. Spatial heterogeneity in the use of antibiotics, fungicides or insecticides affects the transmission of pathogens and pests and is explored as a way to minimize the risk of emergence of drug resistance \cite{Deb-Len-Gan-09}.    
Beaumont et al \cite{Bea-Bur-Duc-Zon-12} study a related model of propagation of salmonella in an industrial hen house. In their study the heterogeneity comes from the alignment of cages separated by free space that allow farmers to take care of the animals. Griette et al \cite{Griette-Alfaro-Raoul-Gandon} studied the propagation properties of a closely related model in the context of the evolution of drug resistance.   

The propagation speed of the solutions of reaction-diffusion equations is often linked to special solutions called \textit{traveling wave solutions}, that are  particular solutions that propagate at a prescribed speed. There exists a large literature around traveling wave solutions and the propagation dynamics of solutions to reaction-diffusion equations and systems, see \cite{Kol-Pet-Pis-1937, Fis-1937, Aro-Wei-75,Aro-Wei-78,Wei-82, Lui-89, Vol-Vol-Vol-94, Lia-Zha-07, Lia-Zha-10} among others. 
When the coefficients depend periodically on the spatial variable such as \eqref{eq:main-sys}, the traveling waves are {sometimes} called \textit{pulsating traveling waves}, see \cite{Shi-Kaw-Ter-86, Xin-00, Ber-Ham-02, Wei-02} among others. {It can be shown that traveling waves also exist for our system but we do not discuss it in the present paper. We shall make a detailed study of traveling waves of \eqref{eq:main-sys} in our forthcoming paper \cite{GM}.}

Our system has a rather intriguing character in the sense that it is cooperative when $(u, v)$ is small while the competitive nature becomes dominant when $(u, v)$ is large. Therefore the standard comparison principle does not apply to the entire system. Such a system has been studied by Wang \cite{Wan-11}, Wang and Castillo-Chavez \cite{Wang-Castillo--Chavez-2012}, Griette and Raoul \cite{Gri-Rao-16}, Girardin \cite{Gir-18, Gir-18-MMMAS}, and Morris, B\"orger and Crooks \cite{Mor-Bor-Cro-19}, when the coefficients are homogeneous in space. 
However, in 
our case, the coefficients are spatially periodic. 
As far as scalar equations are concerned, there 
is a large literature on equations with periodic coefficients, notably \cite{Shi-Kaw-Ter-86, Xin-00, Ber-Ham-02, Wei-02, Lia-Zha-10}. 
As for systems, 
Alfaro and Griette \cite{Alf-Gri-18} constructed a traveling wave for a related system that travels at the expected minimal speed. Apart from this last result, to the best of our knowledge, little  is known for systems of hybrid nature with spatially periodic coefficients.

In this paper we study propagation properties of solutions of \eqref{eq:main-sys}. We first discuss under what conditions propagation occurs by using certain principal eigenvalues. and investigate the spreading speed of solution fronts that start from compactly supported initial data. 

Next we consider the special case where the coefficients are spatially homogeneous. This case has been treated in \cite{Wan-11, Wang-Castillo--Chavez-2012, Gri-Rao-16, Gir-18, Gir-18-MMMAS, Mor-Bor-Cro-19}, 
but the behavior of the solution behind the spreading fronts has not been completely understood. 
We show that, under the assumption that $(u,v)=(0,0)$ is unstable, every positive solution to the Cauchy problem converges to the unique positive stationary solution $(u^*,v^*)$ as $t\to+\infty$ locally uniformly on $\R$. 

Finally we study the system with rapidly oscillating coefficients and discuss their homogenization limit as the spatial period $\ep$ tends to $0$.  Among other things we show that every positive solution to the Cauchy problem converges to the unique positive stationary solution as $t\to+\infty$. Note that, when $\ep$ is not small, such convergence is generally not known. We also prove that the above positive stationary solution and the spreading speed for $\ep>0$ converge to those of the homogenized system as $\ep\to 0$.

Our paper is organized as follows. In section \ref{s:main} we first recall key mathematical notions such as principal eigenvalues of various kinds, left and right spreading speeds, and so on. Then we present our main results including a formula for the spreading speeds (Theorem \ref{thm:main-lindet}), the hair-trigger effect (Theorem~\ref{thm:hairtrigger}), 
global asymptotic stability of stationary solution for the homogeneous 
problem (Theorem \ref{thm:ltb}), and the homogenization limit (Theorem \ref{thm:rapidosc}).

In section \ref{s:proof} we give the proof of those results. This section is further subdivided as follows: in section \ref{ss:principal} we establish some results on the principal eigenvalues of the linearized system; in section \ref{ss:propagation-proof} we prove our statement on the propagation dynamics of solutions of the Cauchy problem; in section \ref{ss:global stability-proof} we prove the global asymptotic stability of the positive equilibrium for the homogeneous problem; in section \ref{ss:homogenization-proof} we prove our statement on the homogenization formula for the speed and the global stability of the positive equilibrium in the case of rapidly oscillating coefficients.

\section{Main results}
\label{s:main}
Throughout this article we make the following assumption on the coefficients of \eqref{eq:main-sys}.
\begin{assumption}[Cooperative-competitive system]\label{as:coop-comp}
	We let $\sigma(x)>0$,   $\kappa_u(x)>0$, $\kappa_v(x)>0$, $\mu_v(x)> 0$, $\mu_u(x)> 0$, be $L$-periodic positive continuous functions and  $r_u(x)$, $r_v(x)$ be $L$-periodic continuous functions of arbitrary sign.  We assume moreover that $\sigma\in C^1(\R )$.
\end{assumption}

Before presenting our main results, let us recall that system \eqref{eq:main-sys} has a cooperative nature for small solutions. To see this, we rewrite \eqref{eq:main-sys} as:
\[
	\begin{system}
		\relax &u_t=\big(\sigma(x) u_{x}\big)_x+\big(r_u(x)-\mu_u(x)-\kappa_u(x)u\big)u+\big(\mu_v(x)-\kappa_u(x)u\big)v,\\
		\relax &v_t=\big(\sigma(x) v_{x}\big)_x+\big(r_v(x)-\mu_v(x)-\kappa_v(x)v\big)v+\big(\mu_u(x)-\kappa_v(x)v\big)u.
	\end{system}
\]
Therefore, if $f^u(x,u,v),\,f^v(x,u.v)$ denote the nonlinearities of the above system, then 
\[
\partial_v f^u(x,u,v)=\mu_v(x)-\kappa_u(x)u, \quad\ \partial_u f^v(x,u,v)=\mu_u(x)-\kappa_v(x)v.
\]
Consequently, we have $\partial_v f^u \geq 0$, $\partial_u f^v \geq 0$ so long as $u, v$ satisfy
\begin{equation}\label{cooperative-zone}
\kappa_u(x)u\leq \mu_v(x),\quad\ \kappa_v(x)v\leq \mu_u(x).
\end{equation}
We call the range of $(u,v)$ satisfying \eqref{cooperative-zone} {\it the cooperative zone} of system \eqref{eq:main-sys}. The cooperative zone becomes larger if the mutation rates $\mu_u, \mu_v$ increase, while it shrinks if the competition rates $\kappa_u, \kappa_v$ increase. It is important that the cooperative zone is always non-empty. Note that, for large values of $(u,v)$ for which neither of the inequalities in \eqref{cooperative-zone} holds, we have $\partial_v f^u < 0$, $\partial_u f^v < 0$, hence \eqref{eq:main-sys} exhibits a competitive nature in this range of $(u,v)$.

\subsection{Principal eigenvalues of the linearized system}\label{ss:eigenvalue}

First we introduce different notions of principal eigenvalues that we use in our results. Even in the scalar case, multiple notions of principal eigenvalues turn out to be useful in the analysis of spreading properties; we refer to Berestycki and Rossi \cite{Ber-Ros-15} and Nadin \cite{Nad-09} for an overview of these notions.

The linearized system associated with \eqref{eq:main-sys} is the following. 
\begin{equation}\label{eq:linearized}
	\begin{system}
		\relax &u_t=\big(\sigma(x) u_{x}\big)_x+r_u(x)u+\mu_v(x)v-\mu_u(x)u, & t>0, x\in\R, \\
		\relax &v_t=\big(\sigma(x) v_{x}\big)_x+r_v(x)v+\mu_u(x)u-\mu_v(x)v, &t>0, x\in\R.
	\end{system}
\end{equation}
Note that this is a cooperative system. 
We first define the notions of periodic principal, $\lambda$-periodic principal and Dirichlet principal eigenelements as follows. 
\begin{definition}[Periodic principal eigenpair]\label{def:periodic-principal-eigenpair}
	By a \textit{periodic principal eigenpair} associated with \eqref{eq:linearized} we mean any pair $\big(\lambda_1^{per}, (\varphi(x), \psi(x))\big)$ where $\lambda_1^{per}\in\R $, $\varphi(x)$ and $\psi(x) $ are positive $L$-periodic smooth functions that satisfy 
	\begin{equation}\label{eq:periodic-principal-eigen}
		\begin{system}
			\relax &  
			L^1[\varphi, \psi](x) :=\big(\sigma(x) \varphi_{x}\big)_x+r_u(x)\varphi+\mu_v(x)\psi-\mu_u(x)\varphi=\lambda_1^{per} \varphi, \\
			\relax & 
			L^2[\varphi, \psi](x) :=
			\big(\sigma(x) \psi_{x}\big)_x+r_v(x)\psi+\mu_u(x)\varphi-\mu_v(x)\psi=\lambda_1^{per}\psi .
		\end{system}
	\end{equation}
	We call $\lambda_1^{per}$  the \textit{principal eigenvalue} and $(\varphi, \psi)$  a \textit{principal eigenvector}.
\end{definition}
It follows from the Krein-Rutman Theorem that $\lambda_1^{per}$ is, indeed, unique, and that $(\varphi, \psi)$ is unique up to multiplication by a positive scalar. Heuristically, $\lambda_1^{per}$ corresponds to the rate of growth of a small population, given that the initial data is $L$-periodic.
\medskip

We continue with the notion of $\lambda$-periodic principal eigenpair.
\begin{definition}[$\lambda$-periodic principal eigenpair]\label{def:k(lambda)}
	For $\lambda>0$, by a \textit{$\lambda$-periodic principal eigenpair} associated with \eqref{eq:linearized} we mean any pair $\big(k(\lambda), (\varphi(x), \psi(x))\big)$ where $k(\lambda)\in\R $, $\varphi(x)$ and $\psi(x) $ are positive $L$-periodic smooth functions that satisfy 
	\begin{equation}\label{eq:lambda-periodic-principal-eigen}
		\begin{system}
			\relax & L^1_\lambda[\varphi, \psi](x) := e^{\lambda x} L^1[e^{-\lambda x}\varphi, e^{-\lambda x}\psi](x) =k(\lambda) \varphi, \\
			\relax & L^2_\lambda[\varphi, \psi](x) := e^{\lambda x} L^2[e^{-\lambda x}\varphi, e^{-\lambda x}\psi](x)= k(\lambda)\psi ,
		\end{system}
	\end{equation}
	or, equivalently, 
	\begin{equation}\label{eq:lambda-periodic-principal-eigen2}\tag{\ref{eq:lambda-periodic-principal-eigen}$'$}
		\begin{system}
			\relax & \big(\sigma(x) \varphi_{x}\big)_x-2\lambda \sigma(x) \varphi_x+\big(\lambda^2\sigma(x)-\lambda\sigma_x(x)+r_u(x)\big)\varphi+\mu_v(x)\psi-\mu_u(x)\varphi=k(\lambda) \varphi, \\
			\relax & \big(\sigma(x) \psi_{x}\big)_x-2\lambda \sigma(x)\psi_x+\big(\lambda^2\sigma(x)-\lambda \sigma_x(x)+r_v(x)\big)\psi+\mu_u(x)\varphi-\mu_v(x)\psi = k(\lambda)\psi .
		\end{system}
	\end{equation}
	We call $k(\lambda)$ the \textit{$\lambda$-periodic principal eigenvalue} and $(\varphi, \psi)$ a \textit{$\lambda$-periodic principal eigenvector}.
\end{definition}
Again, it follows from the Krein-Rutman Theorem that $k(\lambda)$ is unique and that $(\varphi, \psi)$ is unique up to multiplication by a positive scalar. 
We use the notation $k(\lambda)$ to emphasize that this eigenvalue should be considered as a function of the parameter $\lambda$. Note that $\lambda_1^{per} = k(0)$.

The $\lambda$-periodic principal eigenpair plays an important role in the analysis of front behaviors at the leading edge for the following reasons: At the leading edge, where the solution is very small, system \eqref{eq:main-sys} is well approximated by the linearized system \eqref{eq:linearized}, and the function pair 
\[
u(t,x):=\alpha e^{-\lambda(x-ct)}\varphi(x)>0, \quad  v(t,x):=\alpha e^{-\lambda(x-ct)}\psi(x)>0, 
\]
where $\alpha$ is a positive constant, satisfies \eqref{eq:linearized} if and only $(\varphi,\psi)$ is  a $\lambda$-periodic principal eigenvector and $c=k(\lambda)/\lambda$.
\medskip

Lastly, we define our notion of Dirichlet principal eigenvalue. 
\begin{definition}[Dirichlet principal eigenpair]\label{def:Dirichlet-principal-eigenpair}
	Let $R>0$ be given. By a \textit{Dirichlet principal eigenpair on $(-R, R)$} associated with \eqref{eq:linearized} we mean any pair $\big(\lambda_1^R, (\varphi(x), \psi(x))\big)$ where $\lambda_1^R\in\R $, $\varphi(x)$ and $\psi(x) $ are positive  smooth functions on $[-R, R]$ that satisfy
	\begin{subequations}\label{eq:Dirichlet-principal-eigen}
	\begin{equation}
		\begin{system}
			\relax &\big(\sigma(x) \varphi_{x}\big)_x+r_u(x)\varphi+\mu_v(x)\psi-\mu_u(x)\varphi=\lambda_1^{R} \varphi, \\
			\relax &\big(\sigma(x) \psi_{x}\big)_x+r_v(x)\psi+\mu_u(x)\varphi-\mu_v(x)\psi=\lambda_1^{R}\psi ,
		\end{system}
	\end{equation}
	and
		\begin{equation}
			\varphi(-R)= \psi(-R)=0 \text{ and } \varphi(R)= \psi(R)=0.
		\end{equation}
	\end{subequations}
	We call $\lambda_1^{R}$ the \textit{principal eigenvalue} and $(\varphi, \psi)$ a \textit{principal eigenvector}.
\end{definition}
As before, the Krein-Rutman theorem ensures that $\lambda_1^R$ is unique and $(\varphi, \psi)$ is unique up to multiplication by a positive scalar. Heuristically, $\lambda_1^R$ corresponds to the rate of growth of a small population that vanishes at $x=-R$ and $x=R$. \medskip

We are now in a position to state our results on the properties of these different notions of principal eigenvalue and their relations. First we establish properties of the $\lambda$-periodic principal eigenvalue and the map $\lambda\mapsto k(\lambda)$.

\begin{prop}[Properties of $k(\lambda)$]\label{prop:k(lambda)}
	Let Assumption \ref{as:coop-comp} hold true.  Then:
	\begin{enumerate}[label={\rm(\roman*)}]
		\item \label{item:eigenpairk(lambda)}
		    For each $\lambda\in\R$, there exists a $\lambda$-periodic principal eigenpair $\big(k(\lambda),(\varphi, \psi)\big)$ with $\varphi(x)> 0$ and $\psi(x)>0$ for all $x\in\R $, which solves \eqref{eq:lambda-periodic-principal-eigen}, and $(\varphi, \psi)$ is unique up to the multiplication by a positive scalar.
		\item \label{item:minimaxk(lambda)}
			The following characterization of $k(\lambda)$ is valid:
			\begin{equation}\label{eq:minimax-k(lambda)}
				k(\lambda)=
				\underset{(\varphi, \psi)\in C^2_{per}(\R)^2}{\underset{{{\varphi}>{0}, \psi>0}}{\min}}\underset{x\in \R}{\sup}\, {\max}\left(\frac{L^1_\lambda[\varphi, \psi](x)}{\varphi(x)},\frac{L^2_\lambda[\varphi, \psi](x)}{\psi(x)}\right), 
			\end{equation}
			where $L^1_\lambda[\varphi, \psi](x)$, $L^2_\lambda[\varphi, \psi](x)$ are as defined in \eqref{eq:lambda-periodic-principal-eigen}. In addition,  the right-hand side has a unique minimizer up to multiplication by a positive scalar, which coincides with the principal eigenvector of the problem \eqref{eq:lambda-periodic-principal-eigen}.			
		\item \label{item:k(lambda)concave}
			The function $\lambda\mapsto k(\lambda) $ is analytic and strictly convex. Furthermore, 
			the following inequalites hold:
			\begin{equation}\label{k-quadratic}
				\sigma_{\min} \lambda^2 + r_{\min}\leq k(\lambda)\leq \sigma_{\max} \lambda^2+r_{\max}
				\quad \hbox{for all}\ \  \lambda\in \R, 
			\end{equation}
			where $\sigma_{\min}:=\min_{x\in\R}\sigma(x)$, $\sigma_{\max}:=\max_{x\in\R}\sigma(x)$ and
			\[
				r_{\min}:= \min\big(\min_{x\in\R} r_u(x), \min_{x\in\R} r_v(x)\big), \quad
				r_{\max}:= \max\big(\max_{x\in\R} r_u(x), \max_{x\in\R} r_v(x)\big).
				\]
	\end{enumerate}
\end{prop}

Next we recall some classical properties of the principal eigenvalue for the Dirichlet problem.
\begin{prop}[On the  Dirichlet principal eigenvalue for cooperative systems]\label{prop:gen-princ-eig}
	Let Assumption \ref{as:coop-comp} hold true. 	
	Then: for any $R\in (0, +\infty)$, there exists a principal eigenvector $(\varphi, \psi)$ associated with $\lambda_1^R$, which is unique up to the multiplication by a positive scalar. Moreover, the mapping $R\mapsto \lambda_1^R$ is {strictly increasing}.
\end{prop}
The following theorem states that the minimum of the function $k(\lambda)$ is exactly given by the supremum of all $\lambda_1^R$ for $R>0$.

\begin{theorem}[Comparison between Dirichlet and $\lambda$-periodic principal eigenvalues]\label{thm:comp-dir-per}
	Let Assumption \ref{as:coop-comp} hold true. 
	Then $\lambda_1^R<k(\lambda)$ for all $R>0$ and $\lambda\in\R $. Furthermore,
			\begin{equation}\label{lambdaR-k}
				\lim_{R\to+\infty}\lambda_1^R=\min_{\lambda\in\R}k(\lambda).
			\end{equation}
\end{theorem}

From \eqref{lambdaR-k} we see that
\[
\lim_{R\to+\infty}\lambda_1^R=\min_{\lambda\in\R}k(\lambda)
\leq k(0)=\lambda_1^{per},
\]
but the equality does not necessarily hold in general. In the case of scalar equations, it is known that $k(-\lambda)=k(\lambda)$, which is a consequence of the Fredholm alternative since the operator $L_{-\lambda}[\varphi]:=e^{-\lambda x}L[e^{\lambda x}\varphi]$ is the adjoint operator of $L_{\lambda}[\varphi]:=e^{\lambda x}L[e^{-\lambda x}\varphi]$, provided that $L$ is self-adjoint (see also \cite[, Prop. 2.14]{Nad-09}). 
In such a case, the equality $\min_\lambda k(\lambda)=k(0)$ always holds since $k(\lambda)$ is even
and convex. As we see below, the same result holds for our system under additional symmetry assumptions.

\begin{proposition}\label{prop:symm-case}
	Suppose that Assumption \ref{as:coop-comp} holds true, and assume further that either: 
	\begin{enumerate}[label={\rm(\roman*)}]
	    \item \label{item:symmmatrix}
		$\mu_u(x)=\mu_v(x)$ for all $x\in\R $, or
	    \item \label{item:symmreflexion}
		all coefficients are even: $\sigma(x)=\sigma(-x)$, $r_u(x)=r_u(-x)$, $r_v(x)=r_v(-x)$, $\mu_u(x)=\mu_u(-x)$ and $r_v(x)=r_u(-x)$, for all $x\in\R $.
	\end{enumerate}
	Then the function $\lambda\mapsto k(\lambda)$ is even, \textit{i.e.} $k(\lambda)=k(-\lambda)$ for all $\lambda\in\R $. Consequently, we have
			\begin{equation}\label{k-even}
			    \lim_{R\to+\infty} \lambda_1^R= \min_{\lambda\in\R } k(\lambda) =k(0)
			    =\lambda_1^{per}.
			\end{equation}
\end{proposition}

\subsection{Propagation dynamics}\label{ss:propagation}

Before presenting our main results in this section, we remark that nonnegative solutions of \eqref{eq:main-sys} are all bounded as $t\to+\infty$. To state this basic estimate, we introduce the following notation:
\begin{equation}\label{rmax-kappamin}
r_{\max}:=\sup_{x\in\R}\max\big(r_u(x), r_v(x)\big),\ \ 
\kappa_{\min}:=\inf_{x\in\R} \min\big(\kappa_u(x), \kappa_v(x)\big),\ \ 
\overline{K}:=\frac{r_{\max}}{\kappa_{\min}}.
\end{equation}

\begin{prop}[Basic boundedness estimate]\label{prop:uniform-bound}
Let $(u(t,x),v(t,x))$ be a solution of \eqref{eq:main-sys} with nonnegative bounded initial data $(u_0(x),v_0(x))$. Then 
$u(t,x)\geq 0$, $v(t,x)\geq 0$ for all $t\geq0, x\in\R$, and
\begin{equation}\label{u+v<max(K,u0+v0)}
u(t,x)+v(t,x)\leq \max\big(\overline{K},\, \sup_{x\in\R}(u_0(x)+v_0(x))\big) \ \ \hbox{for all}\ \ t\geq0, x\in\R,
\end{equation}  
\begin{equation}\label{uniform-bound}
\limsup_{t\to+\infty}\sup_{x\in\R}\big(u(t,x)+v(t,x)\big)\leq \overline{K}.
\end{equation}
In particular, if $u_0(x)+v_0(x)\leq\overline{K}\,(x\in\R)$, then $u(t,x)+v(t,x)\leq\overline{K}\,(t\geq 0,\,x\in\R)$.
\end{prop}

As we shall see, the above proposition follows by a rather simple comparison argument. Note that uniform boundedness guarantees that any nonnegative solution of \eqref{eq:main-sys} exists globally for $t\geq 0$.

Now we discuss the propagation dynamics of the solutions of \eqref{eq:main-sys}. 
We first focus on solutions with front-like initial data, then we consider solutions with compactly supported initial data. Since the propagation speed may differ depending on whether the front faces toward the right or toward the left, we distinguish the right and left spreading speeds.

\begin{definition}\label{def:front-like}
The pair of bounded nonnegative functions $(u_0,v_0)$ on $\R $ that appears in \eqref{main-initial} is called {\it right front-like} if there exist real numbers $K_1<K_2$ such that
\[
\inf_{x\leq K_1}\min(u_0(x), v_0(x))>0, \quad u_0(x)=v_0(x)= 0 \ \ \hbox{for all} \ \ x\geq K_2.
\]
It is called {\it left front-like} if there exist real numbers $K_1<K_2$ such that
\[
u_0(x)=v_0(x)= 0 \ \ \hbox{for all} \ \ x\leq K_1, \quad \inf_{x \geq K_2}\min(u_0(x), v_0(x))>0. 
\]
\end{definition}

\begin{thm}[Spreading speeds for front-like initial data]\label{thm:main-lindet}
	Let Assumption \ref{as:coop-comp} hold true and assume that $\lambda_1^{per}>0$. Then there exist real numbers $c^*_{R}$, $c^*_{L}$ and a positive number $\eta>0$ such that for any solution $(u,v)$ of \eqref{eq:main-sys}--\eqref{main-initial} whose initial data $(u_0,v_0)$ is right front-like, it holds that
\begin{equation}\label{right-spreading}
\begin{cases}
\displaystyle\ \liminf_{t\to\infty} \left[\inf_{x\leq ct}\min(u(t, x), v(t,x))\right] \geq \eta,
\ \ &\text{for all}\ \  c<c^*_{R},\\[9pt]
\displaystyle\ \limsup_{t\to\infty} \left[\sup_{x\geq ct}\max(u(t, x), v(t,x))\right] = 0, \ \ &\text{for all}\ \ c>c^*_{R},
\end{cases}
\end{equation}
while for any solution $(u,v)$ of \eqref{eq:main-sys}--\eqref{main-initial} whose initial data $(u_0,v_0)$ is left front-like, it holds that
\begin{equation}\label{left-spreading}
\begin{cases}
\displaystyle\ \liminf_{t\to\infty} \left[\inf_{x\geq -ct}\min(u(t, x), v(t,x))\right] \geq \eta,
\ \ &\text{for all}\ \  c<c^*_{L},\\[9pt]
\displaystyle\ \limsup_{t\to\infty} \left[\sup_{x\leq -ct}\max(u(t, x), v(t,x))\right] = 0, \ \ &\text{for all}\ \ c>c^*_{L}.
\end{cases}
\end{equation}
Furthermore, we have the following formula:
\begin{equation}\label{eq:speed}
		c^*_{R}=\inf_{\lambda>0} \frac{k(\lambda)}{\lambda}=\min_{\lambda>0} \frac{k(\lambda)}{\lambda}
		\quad\ \ 
		c^*_{L}=\inf_{\lambda<0} \frac{k(\lambda)}{-\lambda} =\min_{\lambda<0} \frac{k(\lambda)}{-\lambda}
		= \min_{\lambda>0} \dfrac{k(-\lambda)}{\lambda},
\end{equation}
where $k(\lambda)$ is the $\lambda$-principal periodic eigenvalue  defined in Definition~\ref{def:k(lambda)}.
\end{thm}

\begin{definition}[Right- and left spreading speed]\label{def:spreading-speed}
The above quantities $c^*_{R}$ and $c^*_{L}$ are called the \textit{right spreading speed} and \textit{left spreading speed} of solutions to \eqref{eq:main-sys}, respectively.
\end{definition}

Note that the constant $\eta$ that appears in \eqref{right-spreading}, \eqref{left-spreading} does not depend on the choice of the initial data. The formula \eqref{eq:speed} is well-known  for scalar KPP type equations \cite{Xin-00, Wei-02, Ber-Ham-02, Ber-Ham-Roq-05-II, Nad-09}. 
Since we are assuming $\lambda_1^{per}\big(=k(0)\big)>0$, the values of $c^*_{R}$ and $c^*_{L}$ are well-defined and finite.
{Numerical simulations} 
{show} 
{that the} 
{propagating} 
{front of a solution starting from a front-like initial data converges} 
{quickly to a typical coherent shape that travels at a constant speed, which strongly suggests that these are the profiles of traveling waves (see Figure \ref{fig:TW}). }

\begin{figure}[H]
    \centering
    \includegraphics[bb = 0 0 118 131]{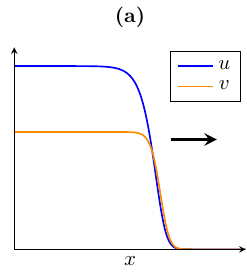} \hspace{1cm}
    \includegraphics[bb = 0 0 118 131]{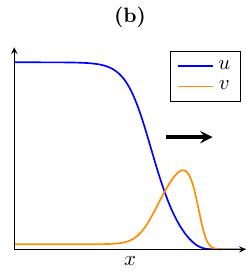}\hspace{1cm}
    \includegraphics[bb = 0 0 118 131]{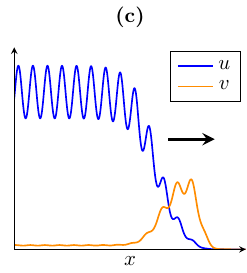}
    \caption{Profiles of propagating fronts of \eqref{eq:main-sys} for different parameter values. {\bf (a)} Spatially homogeneous coefficients with large mutation rates $\mu_u$, $\mu_v$. In this case, the cooperative zone of system \eqref{eq:main-sys} is rather large, and the propagating front lies entirely in this zone. As a result, both $u$ and $v$ have monotone profiles, just as in the case of scalar equations. {\bf (b)} Spatially homogeneous coefficients with small mutation rates $\mu_u$, $\mu_v$. In this case, a large part of the profile of the propagating front lies outside the cooperative zone, and a hump appears on $v$. {\bf (c)} Spatially periodic case. 
The coefficients are the same as in (b), except $r_u(x)$ and $r_v(x)$, which have a cosine-like periodic perturbation.}
    \label{fig:TW}
\end{figure}
By virtue of the inequalities \eqref{k-quadratic}, the following estimates of the spreading speeds hold:

\begin{prop}\label{prop:c*-estimate}
Let $\sigma_{\min}$, $\sigma_{\max}$, $r_{\min}$, $r_{\max}$ be the constants that appear in \eqref{k-quadratic}. Then 
\[
c^*_{R}\leq 2\sqrt{\sigma_{\max} r_{\max}}, \quad c^*_{L}\leq 2\sqrt{\sigma_{\max} r_{\max}}.
\]
Furthermore, if $r_{\min}> 0$, then 
\[
c^*_{R}\geq 2\sqrt{\sigma_{\min} r_{\min}}, \quad c^*_{L}\geq 2\sqrt{\sigma_{\min} r_{\min}}.
\]
\end{prop}

Let us explain the meaning of the formula \eqref{eq:speed} from a different point of view. 
As we mentioned in section~\ref{ss:eigenvalue}, the pair of positive functions of the form
\begin{equation}\label{uv-pair}
u(t,x):=\alpha e^{-\lambda(x-ct)}\varphi(x), \quad  v(t,x):=\alpha e^{-\lambda(x-ct)}\psi(x) ,
\end{equation}
satisfies the linear system \eqref{eq:linearized} if and only if $(\varphi,\psi)$ is a $\lambda$-periodic principal eigenvector of \eqref{eq:lambda-periodic-principal-eigen} and $c=k(\lambda)/\lambda$. Therefore, $c^*_{R}$ can be characterized as follows:
\begin{equation}\label{c*-right}
c^*_{R}=\min \left\{c\in\R \,:\, \hbox{the pair}\ (u,v) \ \hbox{in \eqref{uv-pair} satisfies  \eqref{eq:linearized} for some}\ \ \lambda>0\right\}. 
\end{equation}
As regards the left spreading speed, we consider a pair of functions the form
\begin{equation}\label{uv-pair2}
u(t,x):=\alpha e^{-\lambda(x+ct)}\varphi(x), \quad  v(t,x):=\alpha e^{-\lambda(x+ct)}\psi(x) 
\end{equation}
with $\lambda<0$, since we deal with a front that faces the negative direction of $x$-axis. Then we have
\begin{equation}\label{c*-left}
c^*_{L}=\min \left\{c\in\R \,:\, \hbox{the pair}\ (u,v) \ \hbox{in \eqref{uv-pair2} satisfies \eqref{eq:linearized} for some}\ \ \lambda<0\right\}. 
\end{equation}

Incidentally, combining \eqref{eq:speed} and Proposition~\ref{prop:symm-case}, we obtain the following proposition:

\begin{prop}\label{prop:right-left-speeds}
Let the assumption (i) or (ii) of Proposition~\ref{prop:symm-case} hold. Then $c^*_{R}=c^*_{L}$. In particular, if all the coefficients are spatially homogeneous, then $c^*_{R}=c^*_{L}$. 
\end{prop}

The above result is an immediate consequence of the fact that $k(-\lambda)=k(\lambda)$ which holds under the assumption (i) or (ii). However, without such assumptions, we may have $c^*_{R}\ne c^*_{L}$. 
We shall show such an example in our forthcoming paper \cite{GM}. 
Note that, in the case of spatially periodic scalar KPP type equations, it is known that $k(-\lambda)=k(\lambda)$, as we mentioned earlier, hence we always have $c^*_{R}=c^*_{L}$.

We next consider solutions with compactly supported initial data and discuss the so-called ``hair-trigger effect''. This concept was introduced by Aronson and Weinberger \cite{Aro-Wei-75}, who showed that any solution of the monostable equation $u_t=\Delta u+f(u)$ on $\R ^N$ with nonnegative nontrivial initial data $u_0$ with compact support converges to a positive stationary solution as $t\to+\infty$, no matter how small $u_0$ is. 
However, in the case of system \eqref{eq:main-sys}, we do not know if the condition $\lambda_1^{per}>0$ is sufficient to guarantee the hair-trigger effect, even if we weaken its statement to allow the solution to become simply uniformly positive instead of convergent to a positive equilibrium. 

As we state below, a proper criterion for obtaining a hair-trigger effect is the sign of Dirichlet principal eigenvalues on large domains, which guarantees that both $c^*_{L}$ and $c^*_{R}$ are positive.

\begin{thm}[Hair-trigger effect]\label{thm:hairtrigger}
	Let Assumption \ref{as:coop-comp} hold true. 
Then the following three conditions are equivalent: 
\[
{\rm (a)}\ \lambda_1^R>0 \ \ \hbox{for some} \ R>0,\quad\ 
{\rm (b)}\ \min_{\lambda\in\R }k(\lambda)>0, \quad\ 
{\rm (c)}\ c^*_{R}>0,\ c^*_{L}>0.
\]
If any of these conditions holds, there exists a number $\eta>0$ depending only on the coefficients of system \eqref{eq:main-sys} such that for any nonnegative bounded initial data $(u_0, v_0)$ satisfying $(u_0(x), v_0(x))\not\equiv (0, 0)$, the solution $\big(u(t, x), v(t, x)\big)$ of \eqref{eq:main-sys}--\eqref{main-initial} has the following property: 
	\begin{equation}\label{eq:hairtrigger-below}
		\liminf_{t\to+\infty}u(t, x)\geq \eta \ \ \text{and}\ \  \liminf_{t\to+\infty} v(t, x)\geq \eta \ \ \text{for all}\ \ x\in\R .
	\end{equation}
Furthermore, if, in addition, $u_0$ and $v_0$ are compactly supported, then the right front and the left front of $(u,v)$ propagate at the speed $c^*_{R}$ and $c^*_{L}$, respectively. More precisely,
\begin{subnumcases}{\label{right-spreading2}}
\label{right2a}
\displaystyle\ \liminf_{t\to\infty} \left[\inf_{0\leq x\leq ct}\min(u(t, x), v(t,x))\right] \geq \eta,
\ \ \ \hbox{for all}\ \ 0<c<c^*_{R},\\
\label{right2b}
\displaystyle\ \limsup_{t\to\infty} \left[\sup_{x\geq ct}\max(u(t, x), v(t,x))\right] = 0, \ \ \ \ \hbox{for all}\ \ c>c^*_{R}.
\end{subnumcases}
\begin{subnumcases}{\label{left-spreading2}}
\label{left2a}
\displaystyle\ \liminf_{t\to\infty} \left[\inf_{-ct\leq x\leq 0}\min(u(t, x), v(t,x))\right] \geq \eta,
\ \ \ \hbox{for all}\ \ 0<c<c^*_{L},\\
\label{left2b}
\displaystyle\ \limsup_{t\to\infty} \left[\sup_{x\leq -ct}\max(u(t, x), v(t,x))\right] = 0,\ \ \ \ \hbox{for all}\ \ c>c^*_{L}.
\end{subnumcases}
The assertions \eqref{right2a} and \eqref{left2a} hold for any nonnegative nontrivial solution of \eqref{eq:main-sys}.
\end{thm}

The above theorem shows that the propagation speeds of solutions with compactly supported initial data are the same as those of solutions with front-like initial data. Therefore, the notions of right and left spreading speeds, $c^*_{R}$ and $c^*_{L}$, have a rather universal nature.

Note that if the coefficients of \eqref{eq:main-sys} satisfy the symmetry conditions stated in (i) or (ii) of Proposition~\ref{prop:symm-case}, then by \eqref{k-even}, the above conditions (a), (b), (c) are all equivalent to $\lambda_1^{per}>0$.


\subsection{Global asymptotic stability of the positive equilibrium}
\label{ss:global stability}

Next we turn to the asymptotic behavior of the solutions to the Cauchy problem \eqref{eq:main-sys} in the case where the coefficients are independent of $x$. More precisely, we consider the homogeneous problem
\begin{equation}\label{eq:syst-hom-rd}
	\begin{system}
		\relax &u_t-\sigma u_{xx}=\big(r_u-\kappa_u(u+v)\big)u+\mu_vv-\mu_uu, \\
		\relax &v_t-\sigma v_{xx}=\big(r_v-\kappa_v(u+v)\big)v+\mu_uu-\mu_vv,
	\end{system}
	\quad t>0,\ x\in\R ,
\end{equation}
where $\sigma>0$, $r_u\in\R$, $r_v\in\R$, $\kappa_u>0$, $\kappa_v>0$, $\mu_u>0$, $\mu_v>0$. The linearization of 
\eqref{eq:syst-hom-rd} around $(u, v)=(0,0)$ is given in the following form, which is a spatially homogeneous version of \eqref{eq:linearized}:
\begin{equation}\label{eq:syst-hom-rd-linearized}
	\begin{system}
		\relax &u_t-\sigma u_{xx}=(r_u-\mu_u)u + \mu_vv, \\
		\relax &v_t-\sigma v_{xx}=(r_v-\mu_v)v + \mu_uu,
	\end{system}
	\quad t>0,\ x\in\R .
\end{equation}
As we have seen before, this is a cooperative system, and since the nonlinearity of \eqref{eq:syst-hom-rd} is sublinear, any nonnegative solutions of \eqref{eq:syst-hom-rd} is a subsolution of the cooperative system \eqref{eq:syst-hom-rd-linearized}. Consequently, if we denote by $\big(u,v\big)$ and $\big(\bar u,\bar v\big)$ the solutions of \eqref{eq:syst-hom-rd} and \eqref{eq:syst-hom-rd-linearized}, respectively, then we have
\[
\begin{split}
& \big(u(0,x),v(0,x)\big)\leq \big(\bar u(0,x),\bar v(0,x)\big)\ \ \hbox{for}\ \ x\in\R \\
& \hspace{80pt}
\Rightarrow\ 
\big(u(t,x),v(t,x)\big)\leq \big(\bar u(t,x),\bar v(t,x)\big)\ \ \hbox{for}\ \ t\geq 0,\, x\in\R .
\end{split}
\]

The coefficient matrix of the right-hand side of \eqref{eq:syst-hom-rd-linearized} is given by 
	\begin{equation}\label{matrixA}
		A:=\begin{pmatrix} r_u-\mu_u & \mu_v \\ \mu_u & r_v-\mu_v 
		\end{pmatrix}.
	\end{equation}
	Since the off-diagonal entries of $A$ are positive, we easily see that $A$ has real eigenvalues. 
	Define 
	\begin{equation} \label{eq:lambdaA}
		\lambda_A:=\max\{\lambda\in\R\, |\,\lambda \text{ is an eigenvalue of A}\}.
	\end{equation}
	By the Perron-Frobenius theory, the eigenvector 
	$(\varphi_A^u, \varphi_A^v)^T$ corresponding to $\lambda_A$ is positive. 
	
The sign of $\lambda_A$ plays a key role in the analysis of the corresponding ODE system 
		\begin{equation}\label{eq:syst-ode}
			\begin{system}
				\relax &u_t=(r_u-\kappa_u(u+v))u+\mu_vv-\mu_uu, \\
				\relax &v_t=(r_v-\kappa_v(u+v))v+\mu_uu-\mu_vv.
			\end{system}
		\end{equation}
By definition, the trivial equilibrium point $(0,0)$ is linearly unstable if $\lambda_A>0$ and linearly stable if $\lambda_A<0$. 
Incidentally, if $r_u=r_v$, then $\lambda_A=r_u=r_v$. 

We remark that the value of $\lambda_A$ also plays
an important role in the propagation dynamics of \eqref{eq:syst-hom-rd}. 
To see this, note first that the principal eigenvector $\big(\varphi(x),\psi(x)\big)$ of \eqref{eq:lambda-periodic-principal-eigen} is a constant function. This is because, for any real number $\alpha$, $\big(\varphi(x+\alpha),\psi(x+\alpha)\big)$ is again a principal eigenvector since the coefficients are spatially homogeneous, hence by the uniqueness of the principal eigenvector (up to multiplication of a constant), we have  $\big(\varphi(x+\alpha),\psi(x+\alpha)\big)=\big(\varphi(x),\psi(x)\big)$ for any $a\in\R $, which implies that $\big(\varphi(x),\psi(x)\big)$ is independent of $x$. 
Consequently, the $\lambda$-periodic eigenproblem \eqref{eq:lambda-periodic-principal-eigen} is given in the following simpler form:
\begin{equation}\label{eq:lambda-periodic-eigen-homo}
\begin{system}
	\relax & (\lambda^2\sigma+r_u)\varphi+\mu_v\psi-\mu_u\varphi=k(\lambda) \varphi, \\
	\relax & (\lambda^2\sigma+r_v)\psi+\mu_u\varphi-\mu_v\psi=k(\lambda) \psi.
\end{system}
\end{equation}
It follows that
\begin{equation}\label{k(lambda)-homo}
k(\lambda)=\sigma \lambda^2+\lambda_A.
\end{equation}
In particular, we have
\begin{equation}\label{lambda_A=lambda-per}
\lambda_A=k(0)=\lambda_1^{per}.
\end{equation}
If $\lambda_A>0$, then by \eqref{k(lambda)-homo} and \eqref{eq:speed}, 
\begin{equation}\label{c*-homo}
c^*_{R}=c^*_{L}=2\sqrt{\sigma \lambda_A}>0, 
\end{equation}
hence the hair-trigger effect holds by virtue of Theorem~\ref{thm:hairtrigger}.

Let us come back to the ODE system \eqref{eq:syst-ode} and discuss its dynamics. 
Throughout this section we assume the following, which is a restatement of Assumption~\ref{as:coop-comp} in the spatial homogeneous setting:

	\begin{assumption}\label{as:cond-instab-0}
	The coefficients of \eqref{eq:syst-hom-rd} satisfy Assumption \ref{as:coop-comp}, that is, $\sigma$, $\kappa_u$, $\kappa_v$, $\mu_v$, $\mu_u$ are positive constants and $r_u, r_v$ are constants of an arbitrary sign. 
	\end{assumption}

It can be seen that the condition $\lambda_A>0$ is always satisfied when $r_u>0$ and $r_v>0$, and always fails when $r_u<0$ and $r_v<0$. The situation when $r_u$ and $r_v$ do not have the same sign is more intricate. In this case, there may exist a threshold depending on the values of $\mu_u$, $\mu_v$, such that $(0, 0)$ is stable for small values of $\mu_u$, $\mu_v$, and unstable for larger values. We discuss this threshold later in remark \ref{rem:stab-0-ode}.

The following proposition classifies the long-time behavior of all nonnegative solutions of \eqref{eq:syst-ode} in terms of the sign of $\lambda_A$. 
Note that elements of the proof of this proposition can be found in the work of  Cantrell, Cosner and Yu \cite{Can-Cos-Yu-18}, who proved the global asymptotic stability of the positive equilibrium for a similar system in a bounded domain.

\begin{prop}[Asymptotic behavior of the ODE system]\label{prop:longtime-ode}
	Let Assumption~\ref{as:cond-instab-0} hold, and let
	$(u(t), v(t))$ be the solution of \eqref{eq:syst-ode} starting from a nonnegative nontrivial initial condition $(u_0, v_0)$.
	\begin{enumerate}[label={(\roman*)}]
		\item If $\lambda_A>0$, there is a unique positive equilibrium $(u^*, v^*)$ for \eqref{eq:syst-ode}, and $(u(t), v(t))$ converges to $(u^*, v^*)$ as $t\to+\infty$.
		\item If $\lambda_A\leq 0$, then $(u(t), v(t))$ converges to $(0,0)$ as $t\to+\infty$.
	\end{enumerate}
\end{prop}

As we shall see, the statement (ii) of the above proposition follows easily from the fact that solutions of \eqref{eq:syst-ode} are subsolutions of the linearized system (the ODE version of \eqref{eq:syst-hom-rd-linearized}), which is a cooperative system. On the other hand, the proof of the statement (i) is highly nontrivial, because the system \eqref{eq:syst-ode} is neither entirely cooperative nor entirely competitive. 
To prove the convergence $(u(t),v(t))\to(u^*, v^*)$, we will use two different methods separately depending on the parameter values, one based on a Lyapunov function, and the other based on the so-called ``ultimate cooperative'' property; see section~\ref{ss:global stability-proof} for details.

The following theorem states that, under the assumption $\lambda_A>0$, solutions to the Cauchy problem associated with \eqref{eq:syst-hom-rd} converge in long time to the stationary solution $(u^*, v^*)$. 
The proof of this theorem is based on a Liouville type result on entire solutions of \eqref{eq:syst-hom-rd} (Theorem~\ref{thm:entire-sol-hom}).

\begin{thm}[Asymptotic behavior of the homogeneous RD problem]\label{thm:ltb}
Let Assumption~\ref{as:cond-instab-0} hold, and assume $\lambda_A>0$.  
Let $c^*_{R}$, $c^*_{L}$ be the right and left spreading speeds associated with \eqref{eq:syst-hom-rd}, respectively. 
Then $c^*_{R}=c^*_{L}=2\sqrt{\sigma\lambda_A}>0$. Furthermore, any nonnegative solution $(u(t,x), v(t,x))$ to the Cauchy problem with bounded nontrivial initial data converges as $t\to+\infty$ to the unique positive stationary solution $(u^*, v^*)$ of \eqref{eq:syst-hom-rd}, uniformly in the sense that for each $0<c<c^*_{R}$ we have:
	\begin{equation}\label{u(t,x)-to-u*}
		\lim_{t\to+\infty}\sup_{|x|\leq ct}\max\big(|u(t,x)-u^*|, |v(t,x)-v^*|\big)=0.
	\end{equation}
\end{thm}

\subsection{Homogenization}\label{ss:homogenization}
 
Here we extend the global stability result of the last section to the case of rapidly oscillating coefficients. Our method is based on a combination of dynamical systems theory and parabolic homogenization techniques. In the case of scalar equations with periodic coefficients, the homogenization limits of spreading speeds and traveling waves have been studied in particular by  El Smaily \cite{ElS-08, ElS-10} and El Smaily, Hamel and Roques \cite{ElS-Ham-Roq-09}.

In stating our results, we restrict ourselves to the case $L=1$, without loss of generality. For each $1$-periodic function $\sigma(x)$, $\sigma(x)$, $r_u(x)$, $r_v(x)$, 
$\mu_u(x)$, $\mu_v(x)$, we define:
\begin{equation}\label{rapid-osc-coeffs}
    \begin{aligned}
	r_u^\ep(x)&:=r_u\left(\frac{x}{\ep}\right), & \kappa_u^\ep(x)&:=\kappa_u\left(\frac{x}{\ep}\right) , & \mu_u^\ep(x)&:=\mu_u\left(\frac{x}{\ep}\right), \\
	r_v^\ep(x)&:=r_v\left(\frac{x}{\ep}\right), &\kappa_v^\ep(x)&:=\kappa_v\left(\frac{x}{\ep}\right), & \mu_v^\ep(x)&:=\mu_v\left(\frac{x}{\ep}\right), \\
	&&\sigma^\ep(x)&:=\sigma\left(\frac{x}{\ep}\right)
    \end{aligned}
\end{equation}
and
\begin{equation}\label{eq:mean-coeffs}
    \begin{aligned}
	\overline{r_u}&:=\int_0^1r_u(x)\dd x , & \overline{\kappa_u}&:=\int_0^1\kappa_u(x)\dd x , & \overline{\mu_u}&:=\int_0^1\mu_u(x)\dd x, \\
	\overline{r_v}&:=\int_0^1r_v(x)\dd x, &\overline{\kappa_v}&:=\int_0^1\kappa_v(x)\dd x, & \overline{\mu_v}&:=\int_0^1\mu_v(x)\dd x,  
    \end{aligned}
\end{equation}
\begin{align*}
	\overline{\sigma}^H&:=\left(\int_0^1\frac{1}{\sigma(x)}\dd x\right)^{-1} .
\end{align*}

We study the following system, whose coefficients oscillate rapidly when $\ep$ is small:
	\begin{equation}\label{eq:rapidosc}
		\begin{system}
			\relax &u_t=(\sigma^\ep(x) u_{x})_x+(r_u^\ep(x)-\kappa_u^\ep(x)(u+v))u+\mu_v^\ep(x)v-\mu_u^\ep(x)u, \\
			\relax &v_t=(\sigma^\ep(x) v_{x})_x+(r_v^\ep(x)-\kappa_v^\ep(x)(u+v))v+\mu_u^\ep(x)u-\mu_v^\ep(x)v
		\end{system}
		\ \ \hbox{on}\ \ \R.
	\end{equation}
	On a formal level, the homogenization limit of \eqref{eq:rapidosc} as $\ep\to 0$ is given by:
	\begin{equation}\label{eq:homogenized}
		\begin{system}
			\relax &u_t=\overline{\sigma}^H u_{xx}+(\overline{r_u}-\overline{\kappa_u}(u+v))u+\overline{\mu_v}v-\overline{\mu_u}u, \\
			\relax &v_t=\overline{\sigma}^H v_{xx}+(\overline{r_v}-\overline{\kappa_v}(u+v))v+\overline{\mu_u}u-\overline{\mu_v}v
		\end{system}
		\ \ \hbox{on}\ \ \R.
	\end{equation}
	
	The main result of this section is the following:

\begin{thm}[Homogenization]\label{thm:rapidosc}
    Let $\sigma(x)$, $r_u(x)$, $r_v(x)$, $\kappa_u(x)$, $\kappa_v(x)$, $\mu_u(x)$ and $\mu_v(x)$ be $1$-periodic functions that satisfy Assumption \ref{as:coop-comp} with $L=1$, and such that the matrix $A$ in \eqref{matrixA} with the entries $\overline{r_u}$, $\overline{r_v} $, $\overline{\mu_u}$ and $\overline{\mu_v}$ satisfies $\lambda_A>0$. 
	Then there is $\bar\ep>0$ such that for each $0<\ep<\bar \ep$, 
	\begin{enumerate}[label={\rm(\roman*)}]
		\item the system \label{item:rapid-osc-unique}
		\eqref{eq:rapidosc} possesses 
		a unique positive stationary solution $(u^*_\ep(x), v^*_\ep(x))$; furthermore, $(u^*_\ep(x), v^*_\ep(x))$ is $\ep$-periodic and converges to $(u^*,v^*)$ as $\ep\to 0$ uniformly on $\R$, where $(u^*,v^*)$ is the positive stationary solution of the homogenized system \eqref{eq:homogenized};
		\item \label{item:rapid-osc-speeds}
		 	let $c^*_{\ep,R}, c^*_{\ep,L}$ denote the right and left spreading speeds of the system \eqref{eq:rapidosc}, respectively, then $c^*_{\ep,R}>0$, $c^*_{\ep,L}>0$ and 
			\begin{equation}\label{convergence-speeds}
			\lim_{\ep\to 0}c^*_{\ep,R}=\lim_{\ep\to 0}c^*_{\ep,L}=\overline{c^*_R}\ (\,=\overline{c^*_L}\;) =2\sqrt{\overline{\sigma}^H\lambda_A},
			\end{equation}
			where $\overline{c^*_R}$ and $\overline{c^*_L}$ denote the right and left spreading speeds of the homogenized system \eqref{eq:homogenized};
		\item \label{item:rapid-osc-cv}
			any solution to the Cauchy problem \eqref{eq:rapidosc} starting from a nonnegative nontrivial bounded initial condition converges as $t\to+\infty$ to $(u^*_\ep(x), v^*_\ep(x))$, uniformly in the sense that for any $c_1, c_2$ with $0<c_1<c^*_{\ep,L}$, $0<c_2<c^*_{\ep,R}$, we have
			\begin{equation}\label{uep(t,x)-to-u*ep}
				\lim_{t\to+\infty}\,\sup_{-c_1t\leq x\leq c_2 t}\max\big(|u(t,x)-u^*_\ep(x)|, |v(t,x)-v^*_{\ep}(x)|\big)=0.
			\end{equation}
	\end{enumerate}
\end{thm}

\begin{figure}[H]
    \centering
    \includegraphics[bb = 0 0 118 131]{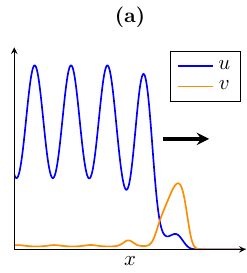} \hspace{1cm}
    \includegraphics[bb = 0 0 118 131]{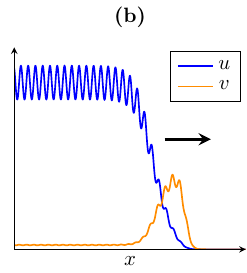}\hspace{1cm}
    \includegraphics[bb = 0 0 118 131]{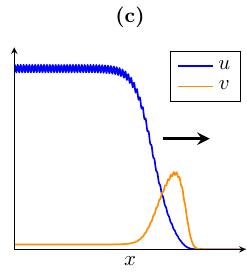}
\caption{Illustration of the homogenization process of 
{propagating} 
fronts. In each figure, the reproduction rates $r_u(x), r_v(x)$ are periodic functions of the form $r_i(x) = r_i+A_i\cos\big((x+\varphi_i)\times(2\pi)/L\big)$ for $i= u, v$, where $L$ is a varying parameter, and all other coefficients including $A_i$ are spatially homogeneous and fixed. $L$ is relatively large in \textbf{(a)}, smaller in \textbf{(b)} and very small in \textbf{(c)}. As $L$ decreases, the amplitude of oscillation of the front profiles becomes smaller and the shape ultimately converges to that of a traveling wave for the homogenized problem. }\label{fig:homogenization}
\end{figure}

\section{Proof of the mathematical results}
\label{s:proof}

\subsection{Principal eigenvalues of the linearized system}
\label{ss:principal}

In this section we focus on the principal eigenvalue problem for general cooperative elliptic systems with periodic coefficients. 
\begin{proof}[Proof of Proposition \ref{prop:k(lambda)}]
	Statement \ref{item:eigenpairk(lambda)} is a direct consequence of the Krein-Rutman Theorem. We concentrate on the remaining statements. \medskip

		{\bf Proof of Statement \ref{item:minimaxk(lambda)}.} 
		We prove the minimax formula \eqref{eq:minimax-k(lambda)}. 
		Let $(\varphi_\lambda,\psi_\lambda)$ denote a principal eigenfunction of \eqref{eq:lambda-periodic-principal-eigen}. Then
		\[
		\frac{L^1_\lambda[\varphi_\lambda, \psi_\lambda](x)}{\varphi_\lambda(x)}=\frac{L^1_\lambda[\varphi_\lambda, \psi_\lambda](x)}{\varphi_\lambda(x)}=k(\lambda). 
		\]
		Thus, using  $(\varphi_\lambda,\psi_\lambda)$ as a test function in \eqref{eq:minimax-k(lambda)}, we find that 
		\begin{equation*}
				k(\lambda)\geq k^*:=
				\underset{(\varphi, \psi)\in C^2_{per}(\R)^2}{\underset{{{\varphi}>{0}, \psi>0}}{\min}}\underset{x\in \R}{\sup}\, {\max}\left(\frac{L^1_\lambda[\varphi, \psi](x)}{\varphi(x)},\frac{L^2_\lambda[\varphi, \psi](x)}{\psi(x)}\right), 
		\end{equation*}
		
		Next let us show the converse inequality. Let $\ep>0$ be given, then by the definition of $k^*$ there exists  $(\varphi, \psi)$ such that 
		\begin{align*}
			\big(\sigma(x)\varphi_x\big)_x& -2\lambda \sigma(x)\varphi_x + \big(\lambda^2\sigma(x)-\lambda \sigma_x(x) +r_u(x)-\mu_u(x)\big)\varphi(x) + \mu_v(x)\psi(x)\\
			&\leq (k^*+\ep)\varphi(x), \\
			\big(\sigma(x)\psi_x\big)_x& -2\lambda \sigma(x)\psi_x + \big(\lambda^2\sigma(x)-\lambda \sigma_x(x) +r_v(x)-\mu_v(x)\big)\psi(x) + \mu_u(x)\varphi(x) \\
			&\leq (k^*+\ep)\psi(x). 
		\end{align*}
		Let $\mu^*>0$ be the largest constant such that $\mu^*\varphi_\lambda\leq \varphi$, $\mu^*\psi_\lambda\leq \psi$. Then there exists a point $x_0\in[0,L]$ such that either $\mu^*\varphi_\lambda$ is tangential to $\varphi$ from below at $x=x_0$ or $\mu^*\psi_\lambda$ is tangential to $\psi$ from below at $x=x_0$. In the former case we have $\mu^*(\varphi_\lambda)_{xx}(x_0)\leq \varphi_{xx}(x_0)$, $\mu^*(\varphi_\lambda)_x(x_0)=\varphi_x(x_0)$, $\mu^*\varphi_\lambda(x_0)=\varphi(x_0)>0$, while in the latter case, we have $\mu^*(\psi_\lambda)_{xx}(x_0)\leq \psi_{xx}(x_0)$, $\mu^*(\psi_\lambda)_x(x_0)=\psi_x(x_0)$, $\mu^*\psi_\lambda(x_0)=\psi(x_0)>0$. In either case, one can deduce from the above inequalities that ${k^*+\ep\geq k(\lambda)}$. Since $\ep>0$ is arbitrary, ${k^*\geq k(\lambda)}$.  Statement \ref{item:minimaxk(lambda)} is proved.\medskip

	{\bf Proof of Statement \ref{item:k(lambda)concave}.} 
	We first note that the analyticity of $k(\lambda)$ is classical. In the terminology of Kato \cite{Kat-95}, the family of unbounded operators in the left-hand side of \eqref{eq:lambda-periodic-principal-eigen} is a holomorphic family of type (A) \cite[Paragraph 2.1 on page 375]{Kat-95} and the principal eigenvalue is isolated in the spectrum by the Krein-Rutman Theorem; therefore the spectral projection and the principal eigenvalue are analytic (see  \cite[Remark 2.9 on page 379]{Kat-95}). The analyticity of $k(\lambda)$ with respect to $\lambda$  follows.

The convexity of $k(\lambda)$ can be established by following the proof of Nadin \cite[Proposition 2.10]{Nad-09} in the scalar case. For the sake of brevity, we omit the proof. Since $k(\lambda)$ is analytic and convex, it is either strictly convex everywhere, or is a linear function. The latter, however, is not possible by the inequality \eqref{k-quadratic}, which we will prove in the next paragraph. Hence $k(\lambda)$ is strictly convex.

Now we prove \eqref{k-quadratic}. 
Let $w(x):=\varphi(x)+\psi(x)$. Adding up the two equations in \eqref{eq:lambda-periodic-principal-eigen2}, we get
\begin{equation}\label{lambda-peri-w}
 \big(\sigma(x) w_{x}\big)_x-2\lambda \sigma(x) w_x+\big(\lambda^2\sigma(x)-\lambda\sigma_x(x)\big)w+r_u(x)\varphi+r_v(x)\psi=k(\lambda) w.
\end{equation}
Let us first prove the upper bound in \eqref{k-quadratic}. From the above equation we have
\[
k(\lambda)w^2 \leq \big(\sigma(x) w_{x}\big)_x w-2\lambda \sigma(x) w_x w+\big(\lambda^2\sigma(x)-\lambda\sigma_x(x) +r_{\max}\big)w^2.
\]
Integrating by parts and recalling the $L$-periodicity of the coefficients and $w$, we obtain
\[
k(\lambda)\int_0^L w^2  \dd x\leq -\int_0^L \sigma w_x^2 \dd x +\int_0^L \big(\lambda^2\sigma(x)+r_{\max}\big)w^2 \dd x
\leq \big(\sigma_{\max}\lambda^2+r_{\max}\big)\int_0^L w^2 \dd x.
\]
This proves the upper bound in \eqref{k-quadratic}. 
Next we prove the lower bound. From \eqref{lambda-peri-w} we get
\[
k(\lambda)\geq \frac{(\sigma(x) w_{x})_x}{w}-2\lambda\sigma(x)\frac{w_x}{w}+\lambda^2\sigma(x)-\lambda\sigma_x(x)+r_{\min}.
\]
Now we integrate the above inequality over $[0,L]$. First note that integration by parts gives
\[
\int_0^L \frac{(\sigma(x) w_{x})_x}{w}\dd x=\int_0^L \sigma(x)\left(\frac{w_x}{w}\right)^2 \dd x.
\]
Consequently,
\[
\begin{split}
k(\lambda) L &\geq \int_0^L \sigma(x)\left(\frac{w_x^2}{w^2}-2\lambda\frac{w_x}{w}+\lambda^2\right)\dd x +r_{\min} L=\int_0^L \sigma(x)\left(\frac{w_x}{w}-\lambda\right)^2 \dd x+r_{\min} L\\
&\geq \sigma_{\min}\int_0^L \left(\frac{w_x}{w}-\lambda\right)^2 \dd x+r_{\min} L
=\sigma_{\min}\int_0^L \left(\frac{w_x^2}{w^2}+\lambda^2\right) \dd x+r_{\min} L\\
&\geq \sigma_{\min}\lambda^2 L +r_{\min} L.
\end{split}
\]
This proves the lower bound of \eqref{k-quadratic}. The proof of Proposition~\ref{prop:k(lambda)} is complete.
\end{proof}

\begin{proof}[Proof of Theorem \ref{thm:comp-dir-per}]	
	 The following  proof is inspired by \cite[Theorem 2.11]{Nad-09} for the scalar case.

	We fix $\lambda\in\R$ and let $(\varphi, \psi)$ be the associated $\lambda$-periodic principal eigenvector. The functions $ \overline{\varphi}(x):=e^{-\lambda x}{\varphi}(x)$ and $ \overline{\psi}(x):=e^{-\lambda x}{\psi}(x)$  satisfy 
	\begin{align*}
	    \big(\sigma(x)\overline{\varphi}_x\big)_x + \big(r_u(x)-\mu_u(x)\big)\overline{\varphi}+\mu_v(x)\overline{\psi}&=k(\lambda) \overline{\varphi}(x), \text{ and }\\
	    \big(\sigma(x)\overline{\psi}_x\big)_x + \big(r_v(x)-\mu_v(x)\big)\overline{\psi}+\mu_u(x)\overline{\varphi}&=k(\lambda) \overline{\psi}.
	\end{align*}
    By comparing $\big(\overline{\varphi}, \overline{\psi}\big)$ to the Dirichlet principal eigenvector in $[-R, R]$ for $R>0$, we find that $k(\lambda)>\lambda_1^R$ for all $R>0$ and $\lambda\in\R $.

	Let us show that there exists $\lambda\in\R $ such that $\lim_{R\to+\infty}\lambda_1^R=k(\lambda)$. Let $(\varphi, \psi)$ be the locally uniform limit of a sequence of Dirichlet principal eigenpairs $(\varphi^R, \psi^R)$ with $R\to+\infty$, normalized with $\varphi^R(0)+\psi^R(0)=1$. Then $(\varphi, \psi)$ is positive and satisfies
	\[
	\begin{array}{l}
	    \big(\sigma(x){\varphi}_x\big)_x + \big(r_u(x)-\mu_u(x)\big){\varphi}+\mu_v(x){\psi}=\lambda_1^\infty \varphi , \\[4pt]
	    \big(\sigma(x){\psi}_x\big)_x + \big(r_v(x)-\mu_v(x)\big){\psi}+\mu_u(x){\varphi} =\lambda_1^\infty {\psi}
	\end{array}
	\ \ \  \hbox{on}\ \; \R ,
	\]
	where $\lambda_1^\infty:=\lim_{R\to+\infty} \lambda_1^R$.
	We let 
	\begin{equation*}
	    \tilde{\varphi}(x):=\frac{\varphi(x+L)}{\varphi(x)} \text{ and }\tilde{\psi}(x):=\frac{\psi(x+L)}{\psi(x)}.
	\end{equation*}
	Then, applying the Harnack inequality for fully coupled elliptic systems \cite[Theorem 8.2]{Bus-Sir-04}  to $(\tilde{\varphi}, \tilde{\psi})$, we see that the function $(\tilde{\varphi}, \tilde{\psi})(x)$ is uniformly bounded. We let 
	\begin{equation*}
	    m:=\sup_{x\in\R}\max\big(\tilde{\varphi}(x), \tilde{\psi}(x)\big)<+\infty.
	\end{equation*}
	Then there exists a sequence $(x_n)$ such that either $\tilde{\varphi}(x_n)\to m$ or $\tilde{\psi}(x_n)\to m$. 
	We define 
	\begin{align*}
	    {\varphi}^n(x)& :=\frac{1}{\varphi(x_n)}\varphi\left(x+L\left\lfloor\frac{x_n}{L}\right\rfloor\right), &  
	    {\psi}^n(x)& :=\frac{1}{\varphi(x_n)}\psi\left(x+L\left\lfloor\frac{x_n}{L}\right\rfloor\right), \\
	    \tilde{\varphi}^n(x)& :=\tilde{\varphi}\left(x+L\left\lfloor\frac{x_n}{L}\right\rfloor\right),  &
	    \tilde{\psi}^n(x)& :=\tilde{\psi}\left(x+L\left\lfloor\frac{x_n}{L}\right\rfloor\right), 
	\end{align*}
	where $\lfloor\cdot\rfloor$ denotes the integer part. 
	Applying the Harnack inequality for fully coupled elliptic systems, the sequences $\varphi^n$ and $\psi^n$ are locally bounded. Moreover, we have
	\begin{equation*}
	    \frac{\varphi^n(x+L)}{\varphi^n(x)} =  \frac{\varphi\left(x+L\left\lfloor\frac{x_n}{L}\right\rfloor+L\right)}{\varphi\left(x+L\left\lfloor\frac{x_n}{L}\right\rfloor\right)} = \tilde{\varphi}^n(x), 
	\end{equation*}
	and similarly
	\begin{equation*}
	    \frac{\psi^n(x+L)}{\psi^n(x)} =  \frac{\psi\left(x+L\left\lfloor\frac{x_n}{L}\right\rfloor+L\right)}{\psi\left(x+L\left\lfloor\frac{x_n}{L}\right\rfloor\right)} = \tilde{\psi}^n(x).
	\end{equation*}
Up to the extraction of subsequences, the sequences $\varphi^n$, $ \psi^n$, $\tilde{\varphi}^n$, and $\tilde{\psi}^n$,  converge locally uniformly to $\varphi^\infty$, $\psi^\infty$, $\tilde{\varphi}^\infty$ and $\tilde{\psi}^\infty$; and importantly, the supremum of $\max(\tilde{\varphi}^\infty, \tilde{\psi}^\infty)$ is attained on the interval $[0, L]$. We remark that 
	$\frac{\varphi^\infty(x+L)}{\varphi^\infty(x)} = \tilde{\varphi}^\infty(x)$, 
	$\frac{\psi^\infty(x+L)}{\psi^\infty(x)} = \tilde{\psi}^\infty(x)$, and that $(\tilde{\varphi}^\infty, \tilde{\psi}^\infty)$ solves 
	\begin{align*}
	    \big(\sigma(x)\tilde{\varphi}^\infty_x\big)_x+\sigma(x)\frac{\varphi^\infty_x}{\varphi^\infty}\tilde{\varphi}^\infty_x + \mu_v(x)\frac{\psi^\infty}{\varphi^\infty}\big(\tilde{\psi}^\infty-\tilde{\varphi}^\infty\big) &=0, \\
	    \big(\sigma(x)\tilde{\psi}^\infty_x\big)_x+\sigma(x)\frac{\psi^\infty_x}{\psi^\infty}\tilde{\psi}^\infty_x + \mu_u(x)\frac{\varphi^\infty}{\psi^\infty}\big(\tilde{\varphi}^\infty-\tilde{\psi}^\infty\big) &=0. 
	\end{align*}
	Consequently the functions $\Phi(x):=\tilde{\varphi}^\infty(x)-m$ and $\Psi(x):=\tilde{\psi}^\infty(x)-m$ satisfy
	\begin{align*}
	    \big(\sigma(x)\Phi_x\big)_x+\sigma(x)\frac{\varphi^\infty_x}{\varphi^\infty}\Phi_x + \mu_v(x)\frac{\psi^\infty}{\varphi^\infty}\big(\Psi-\Psi\big) &=0, \\
	    \big(\sigma(x)\Psi_x\big)_x+\sigma(x)\frac{\psi^\infty_x}{\psi^\infty}\Psi_x + \mu_u(x)\frac{\varphi^\infty}{\psi^\infty}\big(\Phi-\Psi\big) &=0. 
	\end{align*}
	This is a cooperative, fully coupled elliptic system. 
	Furthermore, $\Phi, \Psi\leq 0$ and either $\Phi$ or $\Psi$ attains its maximum somewhere. 
	Hence \cite[Proposition 12.1]{Bus-Sir-04} implies that 
	$\Phi(x)\equiv \Psi(x)\equiv 0$, that is, 
	$\tilde{\varphi}^\infty(x)\equiv m$ and $\tilde{\psi}^\infty(x)\equiv m$. Let 
	$\lambda:=\frac{1}{L}\ln(m)$,
	then we have
	\begin{equation*}
	    \varphi^\infty(x+L)=\tilde{\varphi}^\infty(x)\varphi^\infty(x)=m\varphi^\infty(x) = e^{\lambda L}\varphi^\infty(x) 
	    \text{ and } \psi^{\infty}(x+L) = e^{\lambda L} \psi^\infty(x), 
	\end{equation*}
	therefore the vector function $e^{-\lambda x}\big({\varphi}^\infty(x), {\psi}^\infty(x)\big)$ is $L$-periodic. Moreover it satisfies 
	\begin{align*}
	    \big(\sigma(x){\varphi}^\infty_x\big)_x& -2\lambda \sigma(x){\varphi}^\infty_x + \big(\lambda^2\sigma(x)-\lambda \sigma_x(x) +r_u(x)-\mu_u(x)\big){\varphi}^\infty(x) + \mu_v(x){\psi}^\infty(x) \\
		&=\lambda_1^\infty{\varphi}^\infty(x),  \\
	    \big(\sigma(x){\psi}^\infty_x\big)_x& -2\lambda \sigma(x){\psi}^\infty_x + \big(\lambda^2\sigma(x)-\lambda \sigma_x(x) +r_v(x)-\mu_v(x)\big){\psi}^\infty(x) + \mu_u(x){\varphi}^\infty(x)\\
		&=\lambda_1^\infty{\psi}^\infty(x). 
	\end{align*}
	Thus we have shown that $\lambda_1^\infty=k(\lambda)$ for some $\lambda\in\R $. This finishes the proof of Theorem \ref{thm:comp-dir-per}.
\end{proof}

\begin{proof}[Proof of Proposition \ref{prop:symm-case}]
	We first prove statement \ref{item:symmreflexion}. 
For each $\lambda\in\R $, let $\big(k(\lambda),(\varphi(x),\psi(x))\big)$ be a $\lambda$-periodic principal eigenpair of \eqref{eq:lambda-periodic-principal-eigen}; in other words, suppose that it satisfies \eqref{eq:lambda-periodic-principal-eigen2}. Then it is easily seen that $(\varphi(-x),\psi(-x))$ satisfies \eqref{eq:lambda-periodic-principal-eigen2} with $\lambda$ replaced by $-\lambda$ but with  the constant $k(\lambda)$ unchanged. This means that $\big(k(\lambda),(\varphi(-x),\psi(-x))\big)$ is a $\lambda$-periodic principal eigenpair of \eqref{eq:lambda-periodic-principal-eigen} for $-\lambda$, which shows that the equality $k(-\lambda)=k(\lambda)$ holds.

	Next we prove statement \ref{item:symmmatrix}. Since $\mu_u=\mu_v$, the joint operator $\big(L^1_{-\lambda}, L^2_{-\lambda}\big)^T$ is the formal adjoint of the operator $\big(L^1_{\lambda}, L^2_{\lambda}\big)^T$ for the canonical scalar product in $L^2_{per}(\R)^2$:
	\begin{equation*}
	    \big\langle \big(\varphi_1, \psi_1\big), \big(\varphi_2, \psi_2\big)\big\rangle=\int_0^L\varphi_1(x)\varphi_2(x)\dd x+\int_0^L\psi_1(x)\psi_2(x)\dd x.
	\end{equation*}
Hence, by the Fredholm alternative, $k(-\lambda)=k(\lambda)$. This completes the proof of Proposition~\ref{prop:symm-case}.
\end{proof}

\subsection{Analysis of propagation dynamics}
\label{ss:propagation-proof}

Here we first prove the basic boundedness estimate, Proposition~\ref{prop:uniform-bound}. 

\begin{proof}[Proof of Proposition~\ref{prop:uniform-bound}]
The nonnegativity of $u(x), v(x)$ follows easily from the maximum principle, since $\mu_u(x)\geq 0$, $\mu_v(x)\geq 0$. The details are omitted. 
Next let $w(t,x):=u(t,x)+v(t,x)$. Summing up the two equations in \eqref{eq:main-sys} yields
\[
\begin{split}
(u+v)_t&=\big(\sigma(x)(u+v)_x\big)_x+r_u(x)u+r_v(x)v-\big(\kappa_u(x)u+\kappa_v(x)v\big)\\
&\leq \big(\sigma(x)(u+v)_x\big)_x+r_{\max}(u+v)-\kappa_{\min}(u+v)^2.
\end{split}
\]
Therefore, $w$ satisfies
\begin{equation}\label{w}
w_t\leq \big(\sigma(x)w_x\big)_x+\big(r_{\max}-\kappa_{\min}w\big)w\quad (t>0,\ x\in\R).
\end{equation}
Let $W(t)$ be the solution of the following ODE problem:
\[
W_t=\big(r_{\max}-\kappa_{\min}W\big)W,\quad 
W(0)=\max\big(\overline{K},\, \sup_{x\in\R}(u_0(x)+v_0(x))\big),
\]
where $\overline{K}:=r_{\max}/\kappa_{\min}$. Then by the comparison principle we have $w(t,x)\leq W(t)$ for all $t\geq 0$, $x\in\R$. Furthermore, it is clear from the equation for $W$ that $W(t)$ is nonincreasing in $t$ and converges to $\overline{K}$ as $t\to+\infty$. The conclusion of the proposition then follows.
\end{proof}

Next we prove Theorems~\ref{thm:main-lindet} for the spreading speeds and Theorem~\ref{thm:hairtrigger} for the hair-trigger effect. 
In the proof of Theorem~\ref{thm:main-lindet}, we only consider the case where $(u_0,v_0)$ are right front-like and focus on the right spreading speed $c^*_{R}$, since the other case can be treated precisely the same way. 
Theorem~\ref{thm:main-lindet} follows as a direct consequence of  Lemma~\ref{lem:upper-spreading} and Lemma~\ref{lem:lower-spreading} below.

\begin{lem}[Upper spreading speed]\label{lem:upper-spreading}
    Let Assumption \ref{as:coop-comp} hold true. Let $\lambda>0$ be fixed and $\big(k(\lambda), (\varphi^\lambda(x), \psi^\lambda(x))  \big)$ be the associated $\lambda$-periodic principal eigenpair. Assume that, for some $\alpha>0$,
	\begin{equation}
		u_0(x)\leq \alpha e^{-\lambda x} \varphi^\lambda(x) \ \ \text{and}\ \  v_0(x)\leq \alpha e^{-\lambda x}\psi^{\lambda}(x) \ \ \text{for all}\ x\in\R . 
	\end{equation}
	Then if $c=\frac{k(\lambda)}{\lambda}$ we have
	\begin{equation}\label{eq:upper-spreading-comp}
		u(t, x)\leq \alpha e^{-\lambda (x-ct)} \varphi^\lambda(x) \ \ \text{and}\ \  v_0(x)\leq \alpha e^{-\lambda (x-ct)}\psi^{\lambda}(x) \ \ \text{for all}\ x\in\R \ \text{and}\  t>0. 
	\end{equation}
\end{lem}

\begin{proof}
The vector function $\alpha e^{-\lambda(x-ct)}\big(\varphi^\lambda(x), \psi^\lambda(x)\big)$ is an explicit solution to the linear system \eqref{eq:linearized},
as we mentioned in section~\ref{ss:principal} and also in \eqref{uv-pair}. 
Consequently, this vector function is a super solution to \eqref{eq:main-sys} since the nonlinearity of \eqref{eq:main-sys} is sublinear. 
This implies \eqref{eq:upper-spreading-comp}.
\end{proof}

Before stating the next result on the lower spreading speed, we introduce some key notations. Let 
\begin{equation}\label{K-beta}
    K:=\min\left(\inf_{x\in\R }\frac{\mu_v(x)}{\kappa_u(x)}, \inf_{x\in\R }\frac{\mu_u(x)}{\kappa_v(x)}\right) , \qquad 
    \beta:={2}\,\dfrac{\max\left(\sup_{x\in\R } r_u(x) , \sup_{x\in\R }r_v(x)\right)}{K}, 
\end{equation}
and let $\big(\tilde u(t, x), \tilde v(t, x)\big) $ denote a solution to the auxiliary system: 
	\begin{equation}\label{eq:auxiliary-below}
		\begin{system}
			\relax &\tilde u_t=\big(\sigma(x)\tilde u_{x}\big)_x+\big(r_u(x)-\kappa_u(x)(\tilde u+\tilde v)-\beta \tilde u\big)\tilde u+\mu_v(x)\tilde v-\mu_u(x)\tilde u, & t>0, x\in\R, \\
			\relax &\tilde v_t=\big(\sigma(x) \tilde v_{x}\big)_x+\big(r_v(x)-\kappa_v(x)(\tilde u+\tilde v)-\beta \tilde v\big)\tilde v+\mu_u(x)\tilde u-\mu_v(x)\tilde v, &t>0, x\in\R.
		\end{system}
	\end{equation}

\begin{lem}[Comparison with a lower barrier]\label{lem:comparison-below}                                                              Let Assumption \ref{as:coop-comp} hold true. Let $\tilde u_0(x) $ and $\tilde v_0(x)$ be continuous functions such that 
    \begin{equation}\label{eq:comparison-below-initcond}
		0\leq \tilde u_0(x)\leq\min\big(u_0(x), \frac{1}{2}K\big) \ \ \text{and}\ \  0\leq \tilde v_0(x)\leq \min\big(v_0(x), \frac{1}{2}K\big), 
	\end{equation}
	and let $\big(\tilde u(t, x), \tilde v(t, x)\big) $ be the solution of \eqref{eq:auxiliary-below} starting from $\tilde u(0, x)=\tilde u_0(x)$ and $\tilde v(0, x) = \tilde v_0(x)$. Then for all $t>0$ and $x\in\R $ we have
	\begin{equation}\label{tildeu<u}
		\tilde u(t, x) \leq u(t, x) \ \ \text{and}\ \ \tilde v (t, x) \leq v(t, x).
	\end{equation}
\end{lem}

\begin{proof} 
We first show that $\tilde u+\tilde v\leq K$. 
Summing up the two equations in \eqref{eq:auxiliary-below} yields
\[
(\tilde u+\tilde v)_t \leq \big(\sigma(x)(\tilde u+\tilde v)_x\big)_x + r_u(x)\tilde u+r_v(x)\tilde v -\beta \left(\tilde u^2+\tilde v^2\right).
\]
Therefore, the function $\tilde{w}:=\tilde u+\tilde v$ satisfies
\begin{equation}\label{w-tilde}
\tilde{w}_t \leq \big(\sigma(x)\tilde{w}_x\big)_x + \max(\sup r_u, \sup r_v)\tilde{w} - \frac{\beta}{2} \tilde{w}^2= \big(\sigma(x)\tilde{w}_x\big)_x + \frac{\beta}{2}\left(K-\tilde{w}\right)\tilde{w}.
\end{equation}
Since $\tilde w(0,x)=\tilde u_0(x)+\tilde v_0(x)\leq K$, by the comparison principle we have $\tilde w(t,x)\leq K$. Hence
\begin{equation}\label{u+v<K}
\tilde{u}(t,x)+\tilde{v}(t,x)\leq K \quad \hbox{for all} \ \ t\geq 0, \ x\in\R .
\end{equation}
In particular, $\tilde{u}(t,x)\leq K$ and $\tilde{v}(t,x)\leq K$, which imply 
\begin{equation}\label{muv-kappau}
\mu_v(x) - \kappa_u(x)\tilde u(t,x)\geq 0\ \ \hbox{and}\ \ \mu_u(x) - \kappa_v(x)\tilde v(t,x)\geq 0
\ \ \hbox{for all}\ \ t>0, x\in\R .
\end{equation}

Now, in order to prove \eqref{tildeu<u}, we define $U:=u-\tilde u$, $V:=v-\tilde v$. Then a direct calculation shows
\[
\begin{split}
U_t & = \left(\sigma U_x\right)_x +\big((r_u-\mu_u)-\kappa_u(u+\tilde u+v)\big)U+(\mu_u-\kappa_u\tilde u)V+\beta \tilde u^2,\\
V_t & = \left(\sigma V_x\right)_x +\big((r_v-\mu_v)-\kappa_v(u+v+ \tilde v)\big)V+(\mu_v-\kappa_v\tilde v)U+\beta \tilde v^2,
\end{split}
\]
hence
\begin{equation}\label{UV-ineq}
\begin{split}
U_t & \geq \left(\sigma U_x\right)_x +\big((r_u-\mu_u)-\kappa_u(u+\tilde u+v)\big)U+(\mu_u-\kappa_u\tilde u)V,\\
V_t & \geq \left(\sigma V_x\right)_x +\big((r_v-\mu_v)-\kappa_v(u+v+ \tilde v)\big)V+(\mu_v-\kappa_v\tilde v)U.
\end{split}
\end{equation}
By virtue of the inequalities \eqref{muv-kappau}, the right-hand side of \eqref{UV-ineq} is a cooperative system. In view of this, and the fact that $U(0,x)=u_0(x)-\tilde u_0(x)\geq 0$, $V(0,x)=v_0(x)-\tilde v_0(x)\geq 0$, we obtain $U(t,x)\geq 0$, $V(t,x)\geq 0$ for all $t\geq 0$ and $x\in\R $, which implies \eqref{tildeu<u}. The lemma is proved.
\end{proof}

Note that, by the inequality \eqref{muv-kappau}, $\left(\tilde u(t, x),\tilde v(t, x)\right)$ can be regarded as a solution of the following system so long as the initial data satisfies $\tilde{u}_0(x)+\tilde{v}_0(x)\leq K$.
\begin{equation}\label{eq:auxiliary-below-2}
	\left\{\begin{aligned}\relax
		&\tilde u_t = (\sigma(x)\tilde u_x)_x + \big(r_u(x)-\mu_u(x)-(\kappa_u+\beta)\tilde u\big)\tilde u + \tilde v \big(\mu_v(x)-\kappa_u(x) \tilde u\big)_+, \\ 
		&\tilde v_t = (\sigma(x)\tilde v_x)_x +\big(r_v(x)-\mu_v(x)-(\kappa_v+\beta)\tilde v\big) \tilde v + \tilde u \big(\mu_u(x)-\kappa_v(x) \tilde v\big)_+.
	\end{aligned}\right.
\end{equation}	
This is a cooperative system, therefore the comparison principle holds. 

\begin{lemma}[Spreading properties of \eqref{eq:auxiliary-below}]\label{lem:lower-spreading}
The system \eqref{eq:auxiliary-below} possesses an $L$-periodic positive stationary solution $\big(p(x),q(x)\big)$ satisfying $p(x)+q(x)\leq K$ 
with the following properties.
\begin{itemize}
\item[{\rm (i)}] For any solution $\big(\tilde u(t, x),\tilde v(t, x)\big)$ of \eqref{eq:auxiliary-below} whose initial data  $\big(\tilde u_0(x),\tilde v_0(x)\big)$ is $L$-periodic and satisfies $0<\tilde{u}_0(x)\leq p(x)$, $0<\tilde{v}_0(x)\leq q(x)$, it holds that
\begin{equation}\label{u-to-p}
\lim_{t\to\infty}\tilde u(t, x)=p(x),\quad \lim_{t\to\infty}\tilde v(t, x)=q(x)\quad\ 
\hbox{uniformly on}\ \ \R .
\end{equation}
\item[{\rm (ii)}] For any solution $\big(\tilde u(t, x),\tilde v(t, x)\big)$ of \eqref{eq:auxiliary-below} whose initial data  $\big(\tilde u_0(x),\tilde v_0(x)\big)$ is right front-like and satisfies $0\leq\tilde{u}_0(x)\leq p(x)$, $0\leq\tilde{v}_0(x)\leq q(x)$, it holds that
\begin{equation}\label{tilde-u-propagation}
		\lim_{t\to+\infty} \sup_{x\leq ct} \big(|\tilde u(t, x)-p(x)|+|\tilde v(t, x)-q(x)|\big)=0,\quad
		\hbox{for every}\ \ c<c^*_{R},
\end{equation}
where $c^*_{R}$ is the right spreading speed defined in \eqref{eq:speed}.
\end{itemize}
\end{lemma}

\begin{proof}
As we have shown in the proof of Lemma~\ref{lem:comparison-below}, there is no distinction between the solutions of \eqref{eq:auxiliary-below} and those of \eqref{eq:auxiliary-below-2} so long as the initial data satisfies $\tilde u_0(x)+\tilde v_0(x)\leq K$, thanks to the inequality \eqref{u+v<K}. 
Since \eqref{eq:auxiliary-below-2} is a cooperative system, the comparison principle holds for such solutions. Note that the linearized system for \eqref{eq:auxiliary-below} is \eqref{eq:linearized}, the same as that for \eqref{eq:main-sys}. Hence the principal eigenvalues $\lambda_1^{per}$, $\lambda_1^R$, $k(\lambda)$ associated with \eqref{eq:auxiliary-below} are identical to those associated with \eqref{eq:main-sys}.

Let us first prove the existence of the periodic stationary solution $(p,q)$ with the property (i). Thus, for the moment, we focus on solutions of \eqref{eq:auxiliary-below} whose initial data $(\tilde u_0(x),\tilde v_0(x))$ is $L$-periodic.   
Since $\lambda_1^{per}>0$, there exists $\ep_0>0$ such that, for any $\ep\in(0,\ep_0]$, the pair $\big(\ep\varphi^{per},\ep\psi^{per}\big)$ is a strict subsolution of \eqref{eq:auxiliary-below}, where $(\varphi^{per},\psi^{per})$ is the principal eigenvector of the problem \eqref{eq:periodic-principal-eigen}. 
We choose $\ep_0$ sufficiently small if necessary, so that $\ep_0\varphi^{per}(x)+\ep_0\psi^{per}(x)\leq K$.  
Let $\big(u^{\ep}(t,x),v^{\ep}(t,x)\big)$ denote the solution of \eqref{eq:auxiliary-below} whose initial data is $\big(\ep\varphi^{per},\ep\psi^{per}\big)$. Then, since $\big(u^{\ep}(t,x),v^{\ep}(t,x)\big)$ is also a solution of \eqref{eq:auxiliary-below-2}, which is a cooperative system, this solution is strictly monotone increasing in $t$. Moreover it is bounded from above by the inequality \eqref{u+v<K}. Hence it converges to an $L$-periodic stationary solution $\big(p(x),q(x)\big)$ as $t\to+\infty$. 
Note that we have 
\[
p(x)>\ep_0\varphi^{per}(x), \quad q(x)>\ep_0\psi^{per}(x),
\]
since otherwise $p$ (or $q$) would have to be tangential to $\ep_1\varphi^{per}$ (or $\ep_1\psi^{per}$) from above for some $0<\ep_1\leq \ep_0$, but this is impossible by the strong maximum principle and the fact that $\big(\ep_1\varphi^{per},\ep_1\psi^{per}\big)$ is a strict subsolution.
Consequently, the limit stationary solution $\big(p(x),q(x)\big)$ does not depend on the choice of $\ep\in(0,\ep_0]$. 

Now let $\big(\tilde u_0(x),\tilde v_0(x)\big)$ be any $L$-periodic initial data that satisfies $0<\tilde{u}_0(x)\leq p(x)$, $0<\tilde{v}_0(x)\leq q(x)$. Then there exists $\ep>0$ such that $\ep\varphi^{per}\leq \tilde u_0$, $\ep\psi^{per}\leq \tilde v_0$. Since the solution of \eqref{eq:auxiliary-below} with initial data $\big(\ep\varphi^{per},\ep\psi^{per}\big)$ converges to $(p,q)$ as $t\to+\infty$, we see, by the comparison principle, that the same holds for the solution with initial data $(\tilde u_0,\tilde v_0)$, which proves \eqref{u-to-p}.

Next we prove statement (ii). This is actually a direct consequence of the result of Weinberger \cite{Wei-02}, after adapting our problem to make it fit into the scalar framework used in the paper. 
The paper deals with propagation dynamics of a system defined by a rather abstract order-preserving real-valued operator $Q$ defined on a close set ${\mathcal H}\subset \R ^d$. To make this result applicable to our vector-valued system, 
we rewrite our system \eqref{eq:auxiliary-below} as a nonlocal scalar equation defined on 
$\mathcal H:=\R\times \{0, 1\}\subset \R^2$, which represents two parallel straight lines.  We remark that any continuous vector function $\big(\tilde u(t,x), \tilde v(t,x)\big)$ can be represented as a scalar function $w:\R\times\mathcal H\to \R$ by letting $w(t, x, 0)=\tilde u(t, x)$ and $w(t, x, 1)=\tilde v(t, x)$. 
Therefore our system can be regarded as a scalar system on the habitat $\mathcal H$. If we define the operator $Q$ as the time-1 map of the system \eqref{eq:auxiliary-below}:
\[
Q:\big(\tilde u_0(x),\tilde v_0(x)\big) \mapsto \big(\tilde u(1,x),\tilde v(1,x)\big),
\]
then it is not difficult to see that the assumptions of \cite[Theorem 2.1]{Wei-02} are all fulfilled, thanks, in particular, to the property \eqref{u-to-p}. 
The fact that the right spreading speed of $\big(\tilde u,\tilde v\big)$ coincides with the value $c^*_{R}$ in \eqref{eq:speed} follows from \cite[Corollary 2.1]{Wei-02} and the fact that the $\lambda$-principal eigenvalues $k(\lambda)$ for \eqref{eq:auxiliary-below} are identical to those for \eqref{eq:main-sys}. 
(We remark that the same conclusion also follows from the abstract results of Liang and Zhao \cite[Theorems 2.11, 2.15, 3.10]{Lia-Zha-07}). 
The Lemma is proved.
\end{proof}

\begin{proof}[Proof of Theorem \ref{thm:main-lindet}]
Since the assertions \eqref{right-spreading} and \eqref{left-spreading} can be shown precisely the same way by simply reversing the direction of $x$-axis, we only prove the former.

The second assertion of \eqref{right-spreading} is a consequence of Lemma~\ref{lem:upper-spreading}. 
The first assertion of \eqref{right-spreading} follows from the inequalities \eqref{tildeu<u} and Lemma~\ref{lem:lower-spreading} (ii). 
Finally, the ``$\inf$'' in \eqref{eq:speed} can be replaced by ``$\min$'', since $k(0)=\lambda^{per}>0$ and $k(\lambda)$ grows quadratically by virtue of \eqref{k-quadratic}.  
The Theorem is proved.
\end{proof}

\begin{proof}[Proof of Proposition~\ref{prop:c*-estimate}] 
We only prove the assertion for $c^*_R$, as the proof for $c^*_L$ is precisely the same. By \eqref{k-quadratic},
\[
c^*_R=\min_{\lambda>0}\frac{k(\lambda)}{\lambda}\leq \min_{\lambda>0}\left(\sigma_{\max}\lambda+\frac{r_{\max}}{\lambda}\right)=2\sqrt{\sigma_{\max}r_{\max}}.
\]
Next, assume $r_{\min}>0$. 
Let the above minimum of $k(\lambda)/\lambda$ is attained at $\lambda=\lambda_0>0$. Then
\[
c^*_R=\frac{k(\lambda_0)}{\lambda_0}\geq \sigma_{\min}\lambda_0+\frac{r_{\min}}{\lambda_0}\geq
\min_{\lambda>0}\left(\sigma_{\min}\lambda+\frac{r_{\min}}{\lambda}\right)=2\sqrt{\sigma_{\min}r_{\min}}.
\]
This completes the proof of the proposition.
\end{proof}

\begin{proof}[Proof of Theorem~\ref{thm:hairtrigger} (hair-trigger effect)]
Let us first prove that the conditions (a), (b), (c) are equivalent. The equivalence ${\rm (a)}\Leftrightarrow{\rm (b)}$ is already implied in \eqref{lambdaR-k}, since $\lambda_1^R$ is strictly increasing in $R$. The assertion ${\rm (b)}\Rightarrow{\rm (c)}$ is also clear since $k(\lambda)$ is convex. Now assume that (c) holds. Then by the formula \eqref{eq:speed}, we have $k(\lambda)>0$ for $\lambda>0$ and also for $\lambda<0$. It remains to show that $k(0)>0$. Assume by contradiction that $k(0)=0$. This means that $k(0)=k'(0)=0$. Then we have
\[
c^*_{R}=\min_{\lambda>0}\frac{k(\lambda)}{\lambda}=k'(0)=0,\quad c^*_{L}=\min_{\lambda<0}\frac{k(\lambda)}{-\lambda}=-k'(0)=0,
\]
contradicting the assumption (c). This contradiction proves that ${\rm (c)}\Rightarrow{\rm (b)}$ holds. The equivalence of (a), (b), (c) is proved.

Next we prove \eqref{eq:hairtrigger-below}. Actually this statement follows from \eqref{right-spreading} and \eqref{left-spreading}, but since the proof of the latter two statements relies on Theorem~\ref{thm:main-lindet}, we give a much simpler direct proof of \eqref{eq:hairtrigger-below}. By Lemma~\ref{lem:comparison-below}, it suffices to prove the claim for solutions of \eqref{eq:auxiliary-below}.

Choose a large enough $R>0$ such that $\lambda_1^R>0$ and that $R\geq L$, and consider the system \eqref{eq:auxiliary-below} on the interval $[-R,R]$ under the Dirichlet boundary conditions at $x=\pm R$, namely
\begin{equation}\label{eq:auxiliary-below-R}
		\begin{system}
			\relax &\tilde u_t=\big(\sigma\tilde u_{x}\big)_x+\big(r_u-\kappa_u(\tilde u+\tilde v)-\beta \tilde u\big)\tilde u+\mu_v\tilde v-\mu_u\tilde u, & t>0, x\in(-R,R), \\
			\relax &\tilde v_t=\big(\sigma \tilde v_{x}\big)_x+\big(r_v-\kappa_v(\tilde u+\tilde v)-\beta \tilde v\big)\tilde v+\mu_u\tilde u-\mu_v\tilde v, &t>0, x\in(-R,R),\\
			\relax &\tilde u(t,-R)=\tilde u(t,R)=0,\ \ \ \tilde v(t,-R)=\tilde v(t,R)=0, & t>0.
		\end{system}
\end{equation}
As in the case of \eqref{eq:auxiliary-below}, for any solution of \eqref{eq:auxiliary-below-R} whose initial data satisfies $\tilde u_0(x)+\tilde v_0(x)\leq K$, the inequality \eqref{u+v<K} holds on the interval $[-R,R]$, therefore the comparison principle holds among such solutions of \eqref{eq:auxiliary-below-R}. 
Since $\lambda_1^{R}>0$, there exists $\ep_0>0$ such that, for any $\ep\in(0,\ep_0]$, the pair $\big(\ep\varphi^R,\ep\psi^R\big)$ is a strict subsolution of \eqref{eq:auxiliary-below-R}, where $(\varphi^R,\psi^R)$ denotes the principal eigenvector of the problem \eqref{eq:Dirichlet-principal-eigen}. We choose $\ep_0$ small enough so that $\ep_0\varphi^{R}(x)+\ep_0\psi^{R}(x)\leq K$.  
Let $\big(u^{\ep}(t,x),v^{\ep}(t,x)\big)$ denote the solution of \eqref{eq:auxiliary-below-R} whose initial data is $\big(\ep\varphi^R,\ep\psi^R\big)$. Then, by the comparison principle, this solution is strictly monotone increasing in $t$ and is bounded from above by the inequality \eqref{u+v<K}. Hence it converges to a stationary solution $\big(P^R(x),Q^R(x)\big)$ as $t\to+\infty$. 
Note that we have 
\[
P^R(x)>\ep_0\varphi^{R}(x), \quad Q^R(x)>\ep_0\psi^{R}(x),
\]
since otherwise $P^R$ (or $Q^R$) has to be tangential to $\ep_1\varphi^{R}$ (or $\ep_1\psi^{R}$) from above for some $0<\ep_1\leq \ep_0$, but this is impossible by the strong maximum principle, Hopf boundary lemma and the fact that $\big(\ep_1\varphi^{R},\ep_1\psi^{R}\big)$ is a strict subsolution of the system \eqref{eq:auxiliary-below-R}. 
Consequently, the limit stationary solution $\big(P^R(x),Q^R(x)\big)$ does not depend on the choice of $\ep\in(0,\ep_0]$.  

Now let $\big(\tilde u(t,x),\tilde v(t,x)\big)$ be any solution of \eqref{eq:auxiliary-below} whose initial data is nonnegative, nontrivial and satisfies $\tilde u_0(x)+\tilde v_0(x)\leq K$. Fix $\tau>0$. Then $\tilde u(\tau,x)>0$, $\tilde v(\tau,x)>0$ for all $x\in\R $, hence $\tilde u(\tau,x)\geq \ep\varphi^R(x)$, $\tilde v(\tau,x)\geq \ep\psi^R(x)$ on $[-R,R]$ for some $\ep\in(0,\ep_0]$. By the comparison principle,
\[
\tilde u(t+\tau,x)\geq \tilde u^\ep(t,x),\ \  \tilde v(t+\tau,x)\geq \tilde v^\ep(t,x)\quad \hbox{for all}\ \ t>0.\ x\in[-R,R],
\] 
where $\big(u^{\ep}(t,x),v^{\ep}(t,x)\big)$ denote the solution of \eqref{eq:auxiliary-below-R} whose initial data is $\big(\ep\varphi^R,\ep\psi^R\big)$. Letting $t\to+\infty$, we obtain
\[
\liminf_{t\to+\infty}\tilde u(t,x)\geq P^R(x),\quad 
\liminf_{t\to+\infty}\tilde v(t,x)\geq Q^R(x),\quad \hbox{for}\ \ x\in[-R,R]. 
\]
Replacing the interval $[-R,R]$ by $[-R+kL, R+kL]$ ($k\in{\mathbb Z}$) and repeating the same argument, we  see that the following estimate holds for all $k\in{\mathbb Z}$:
\begin{equation}\label{u>P}
\liminf_{t\to+\infty}\tilde u(t,x)\geq P^R(x+kL),\ \ 
\liminf_{t\to+\infty}\tilde v(t,x)\geq Q^R(x+kL),\ \  \hbox{for}\ \ x\in[-R+kL,R+kL]. 
\end{equation}
Since $R\geq L$, the family of intervals $[-R+kL,R+kL]\,(k\in{\mathbb Z})$ covers the entire $x$-axis with much overlapping. Therefore, \eqref{u>P} gives a uniform positive lower bound. The assertion \eqref{eq:hairtrigger-below} is proved.

Next we prove the second part of the theorem. As mentioned before, we only prove \eqref{right-spreading2}, since \eqref{left-spreading2} can be shown precisely the same way by simply reversing the direction of the $x$-axis. By what we have just shown above, the following inequalities hold:
\[
\liminf_{t\to+\infty}\tilde u(t,0)\geq P^R(0),\quad 
\liminf_{t\to+\infty}\tilde v(t,0)\geq Q^R(0).
\]
Fix a constant $m$ satisfying $0<m<\min(P^R(0),Q^R(0))$. Then there exists $T>0$ such that
\begin{equation}\label{u>m}
\tilde u(t,0)>m,\ \  \tilde v(t,0)>m\quad \hbox{for all}\ \ t\geq T.
\end{equation} 
Fix such $T>0$. Note that, since $P^R+Q^R\leq K$, we have $m<K/2$.

Next we consider another auxiliary system of the form
	\begin{equation}\label{eq:auxiliary-below-3}
		\begin{system}
			\relax &{\widehat{u}}_t=\big(\sigma(x)\widehat{u}_{x}\big)_x+\big(r_u(x)-\kappa_u(x)(\widehat{u}+\widehat{v})-\beta' \widehat{u}\big)\widehat{u}+\mu_v(x)\widehat{v}-\mu_u(x)\widehat{u}, & t>0, x\in\R, \\
			\relax &\widehat{v}_t=\big(\sigma(x) \widehat{v}_{x}\big)_x+\big(r_v(x)-\kappa_v(x)(\widehat{u}+\widehat{v})-\beta' \widehat{v}\big)\bar v+\mu_u(x)\widehat{u}-\mu_v(x)\widehat{v}, &t>0, x\in\R,
		\end{system}
	\end{equation}
where the constant $\beta'$ is given by
\[
\beta':=\frac{K}{m}\beta,
\]
with $K$ and $\beta$ being the constants defined in \eqref{K-beta}. This system is obtained by replacing the constant $\beta$ in \eqref{eq:auxiliary-below} by $\beta'$. By using an argument similar to \eqref{w-tilde}, we see that $\widehat{w}(t,x):=\widehat{u}(t,x)+\widehat{v}(t,x)$ satisfies
\[
\widehat{w}_t \leq \big(\sigma(x)\widehat{w}_x\big)_x + \frac{\beta K}{2}\widehat{w} - \frac{\beta'}{2} \widehat{w}^2= \big(\sigma(x)\widehat{w}_x\big)_x + \frac{\beta K}{2m}\left(m-\widehat{w}\right)\widehat{w}.
\]
Therefore, if the initial data of the solution of \eqref{eq:auxiliary-below-3} satisfies 
\begin{equation}\label{w0<m}
\widehat{w}(0,x):=\widehat{u}_0(x)+\widehat{v}_0(x)\leq m,
\end{equation}
then $\widehat{w}(t,x):=\widehat{u}(t,x)+\widehat{v}(t,x)\leq m$ for all $t\geq 0$, $x\in\R $. In particular, we have
\begin{equation}\label{u-hat(0)}
\widehat{u}(t,0)\leq m,\ \ \widehat{v}(t,0)\leq m\quad \hbox{for all}\ \ t\geq 0, 
\end{equation}
provided that \eqref{w0<m} holds. Now we consider a solution $\big(\widehat{u}(t,x),\widehat{v}(t,x)\big)$
whose initial data $\big(\widehat{u}_0,\widehat{v}_0\big)$ satisfies \eqref{w0<m} and is left front-like in the sense that 
\[
\inf_{x\leq K_1} \min\big(\widehat{u}_0(x),\widehat{v}_0(x)\big)>0\ \ \hbox{for some}\ K_1<0,\quad \widehat{u}_0(x)=\widehat{v}_0(x)=0\ \ \hbox{for all}\ x\geq 0.
\]
Since $\beta'>\beta$, this is a subsolution of the system \eqref{eq:auxiliary-below}. We claim that 
\begin{equation}\label{u-hat<u-tilde}
\widehat{u}(t,x) < \tilde{u}(t+T,x),\ \ \widehat{v}(t,x)<\tilde{v}(t+T,x)\quad\hbox{for all}\ \ t\geq 0,\ x\geq 0.
\end{equation}
Indeed, at $t=0$, \eqref{u-hat<u-tilde} certainly holds for all $x\geq 0$ since $\widehat{u}_0$ and $\widehat{v}_0$ are $0$. At the boundary $x=0$, the above inequality holds by virtue of \eqref{u>m} and \eqref{u-hat(0)}. Thus the comparison principle implies \eqref{u-hat<u-tilde}. 

By Theorem~\ref{thm:main-lindet}, the front of $\big(\widehat{u}(t,x),\widehat{v}(t,x)\big)$ propagates at the speed $c^*_{R}$, since the linearized system for \eqref{eq:auxiliary-below-3} is the same as that for \eqref{eq:main-sys}. This and \eqref{u-hat<u-tilde} proves 
\eqref{right2a}. Note that this statement holds for any nonnegative nontrivial solution of \eqref{eq:main-sys}. 
For solutions with compactly supported initial data, the assertion \eqref{right2b} 
is a consequence of Lemma~\ref{lem:upper-spreading}. The Theorem is proved.
\end{proof}

\subsection{Proof of global asymptotic stability of the positive equilibrium}
\label{ss:global stability-proof}

In this section we focus on the case where the coefficients of \eqref{eq:main-sys} are spatially homogeneous. In Section \ref{s:statsol-ode} we study the corresponding ODE problem and prove local asymptotic stability and uniqueness of stationary solutions. 
Then, in Section \ref{sss:hom-rd}, we extend those results to the system \eqref{eq:main-sys} with spatially homogeneous coefficients and prove  Theorem~\ref{thm:ltb}.

\subsubsection{Global dynamics of the ODE problem}
\label{s:statsol-ode}

Here we prove Proposition~\ref{prop:longtime-ode} on the dynamics of the ODE system. We rewrite the system \eqref{eq:syst-ode}:
\begin{equation*}
	\begin{system}
		u_t=(r_u-\kappa_u(u+v))u+\mu_vv-\mu_uu&=:f^u(u,v), \\
		v_t=(r_v-\kappa_v(u+v))v+\mu_uu-\mu_vv&=:f^v(u,v).
	\end{system}
\end{equation*}
Here the coefficients $r_u, r_v$ need not be positive, but the other coefficients are all assumed to be positive. 
We first prove statement (ii) (for $\lambda_A\leq 0$), which can be done by simply comparing the solutions with those of the linearized system. The proof of statement (i) (for $\lambda_A>0$), on the other hand, requires much more involved arguments, and a large part of this section is devoted to the proof of statement (ii). To achieve this goal, two different methods are to be employed, depending on the sign of $r_u-\mu_u$ and $r_v-\mu_v$. If one is positive, the system admits a Lyapunov function which 
simplifies the convergence proof significantly; whereas in the case where both are nonpositive, the system is ultimately cooperative and the long-time behavior can be handled by monotonicity arguments   
(using super- and subsolutions). Note that both arguments are inspired by 
\cite{Can-Cos-Yu-18}. We still include the proofs for the sake of completeness. 

\begin{proof}[Proof of Proposition~\ref{prop:longtime-ode} (ii)] 
The linearized system of \eqref{eq:syst-ode} is given in the following form:
\begin{equation}\label{eq:syst-ode-linearized}
	\begin{system}
		\relax &u_t =(r_u-\mu_u)u + \mu_vv, \\
		\relax &v_t =(r_v-\mu_v)v + \mu_uu,
	\end{system}
	\quad t>0,\ x\in\R .
\end{equation}
This is a cooperative system, and since the nonlinearity of \eqref{eq:syst-ode} is sublinear, solutions of \eqref{eq:syst-ode} are subsolutions of the system \eqref{eq:syst-ode-linearized}. Consequently, if $\big(u,v\big)$ and $\big(\bar u,\bar v\big)$ denote the solutions of \eqref{eq:syst-ode} and \eqref{eq:syst-ode-linearized}, respectively, we have
\begin{equation}\label{comparison-ODE}
\big(u(0),v(0)\big)\leq \big(\bar u(0),\bar v(0)\big)\ \Rightarrow\ 
\big(u(t),v(t)\big)\leq \big(\bar u(t),\bar v(t)\big)\ \ \hbox{for}\ \ t\geq 0.
\end{equation}

We first consider the case where $\lambda_A<0$. Let $(\varphi_A^u, \varphi_A^v)^T$ denote the positive eigenvector of the matrix $A$ corresponding to $\lambda_A$. Then for all $M>0$, $(\bar u(t), \bar v(t)):=Me^{\lambda_At}(\varphi_A^u, \varphi_A^v)$ is a solution of \eqref{eq:syst-ode-linearized} that converges to $(0,0)$ as $t\to+\infty$. Consequently, by \eqref{comparison-ODE}, any nonnegative solution of \eqref{eq:syst-ode} converges to $(0,0)$. 

Next we consider the case where $\lambda_A=0$. In this case, the system possesses a one-dimensional family of equilibrium points $(M \varphi_A^u, M \varphi_A^v)\;(M\geq 0)$. For each $u,v\geq 0$, define
\[
M(u,v):=\min \left\{ M\geq 0\,:\, u\leq M \varphi_A^u,\; v\leq M \varphi_A^v \right\} .
\]
    Then, by applying \eqref{comparison-ODE} to the case when $(\bar u, \bar v)$ is an equilibrium, we see that $M(u(t),v(t))$ is non-increasing in $t$ for any nonnegative solutions $(u(t),v(t))$ of \eqref{eq:syst-ode}. Furthermore, it is easily seen that $M(u(t),v(t))$ is strictly decreasing in $t$ except when $M=u(0)=v(0)=0$. Therefore, $M(u,v)$ is a Lyapunov function for the system \eqref{comparison-ODE} whose unique local minimum is achieved at $(u,v)=(0,0)$. This proves that $(u(t),v(t))\to (0,0)$ as $t\to+\infty$. The proof of statement (ii) is complete.
\end{proof}

\begin{remark}
    The non-existence of a positive stationary solution when $\lambda_A=0$ was treated in \cite[Theorem 1.4 (ii)]{Gir-18}  by using a different method and it also follows from \cite[Theorem 13.1 (c)]{Bus-Sir-04}.
\end{remark}

Now we turn to the proof of statement (ii). We prepare several lemmas.

\begin{lem}[Existence and uniqueness of stationary state]\label{lem:stat-ode-uniqueness}
    Let $ r_u, r_v\in\R$, $\kappa_u>0$, $\kappa_v>0$, and $\mu_u, \mu_v>0$. 
    Suppose that $\lambda_A>0$.
    Then, there exists a unique nonnegative nontrivial equilibrium $(u^*, v^*)$ for \eqref{eq:syst-ode},
    that satisfies:
	\begin{enumerate}[label={\rm(\roman*)}]
		\item \label{item:lemestu^*-smallmu}
			if $r_u-\mu_u>0$ (resp. $r_v-\mu_v>0$), then 
			\begin{align*}
				0< \frac{\min\left(\mu_v, r_u-\mu_u\right)}{\kappa_u}&\leq u^*\leq \frac{\max\left(\mu_v,r_u-\mu_u\right)}{\kappa_u} \\
				\text{resp.}\ \ 0<\frac{\min\left(\mu_u, r_v-\mu_v\right)}{\kappa_v}&\leq v^*\leq \frac{\max\left( \mu_u,r_v-\mu_v\right)}{\kappa_v}. 
			\end{align*}
			Equality holds in the above inequalities if, and only if $\mu_v=r_u-\mu_u$ (resp. $r_v-\mu_v=\mu_u$).
		\item \label{item:lemestu^*-bigmu}
			if $r_u-\mu_u\leq 0$ (resp. $r_v-\mu_v\leq0$), then $0<u^*< \frac{\mu_v}{\kappa_u}$ (resp. $0<v^*<\frac{\mu_u}{\kappa_v}$).
			
			In particular, $\big(u^*,v^*\big)$ belongs to the interior of the cooperative zone defined in \eqref{cooperative-zone} if $r_u-\mu_u\leq 0$ and $r_v-\mu_v\leq0$.
	\end{enumerate}
\end{lem}

\begin{proof}
	Let $(u,v)$ be a nonnegative nontrivial stationary state for \eqref{eq:syst-ode}. Then $(u,v)$ satisfies
	\begin{equation*}
		\begin{system}
			\relax &u(r_u-\kappa_u(u+v))+\mu_vv-\mu_uu=0, \\
			\relax &v(r_v-\kappa_v(u+v))+\mu_uu-\mu_vv=0. 
		\end{system}
	\end{equation*}
	Since $(u,v)$ is nonnegative and nontrivial, and since $\mu_u>0,\mu_v>0$, we have in fact $u>0$ and $v>0$. 
	 We introduce the new variables $S=u+v$ and $Q=\frac{u}{v}$, 
	which 
	satisfy the system:
	\begin{align*}
		& \begin{system}
			\relax &Q(r_u-\kappa_uS)+\mu_v-\mu_uQ=0, \\
			\relax &r_v-\kappa_vS+\mu_uQ-\mu_v=0,
		\end{system} 
		\ \ \Longleftrightarrow\ \   
		\begin{system}
			\relax &Q(r_u-\kappa_uS)+\mu_v-\mu_uQ=0, \\
			\relax &S=\frac{r_v+\mu_uQ-\mu_v}{\kappa_v},
		\end{system}
		\\
		&\Longleftrightarrow \ \ 
		\begin{system}
			\relax &-\mu_u\frac{\kappa_u}{\kappa_v}Q^2 + \left(r_u-\mu_u\frac{\kappa_u}{\kappa_v}(r_v-\mu_v)\right)Q+\mu_v=0, \\
			\relax &S=\frac{r_v+\mu_uQ-\mu_v}{\kappa_v}.
		\end{system}
	\end{align*}
	
	The first line of the latter system has a unique positive solution: 
	\begin{equation*}
		Q=\frac{\kappa_v}{2\mu_u\kappa_u}\left(r_u-\mu_u-\frac{\kappa_u}{\kappa_v}(r_v-\mu_v)+\sqrt{\left(r_u-\mu_u-\frac{\kappa_u}{\kappa_v}(r_v-\mu_v)\right)^2+4\frac{\kappa_u}{\kappa_v}\mu_u\mu_v}\right).
	\end{equation*}
	Since the change of variables is reversible, we have proved the uniqueness of the solution. To prove the existence of an equilibrium for \eqref{eq:syst-ode}, we first observe that $S(t)=u(t)+v(t)$ satisfies
	\[
	S_t \leq \max(r_u,r_v)S-\min(\kappa_u,\kappa_v)S^2,
	\] 
	which follows by adding up the two equations in \eqref{eq:syst-ode}. Consequently, the interior of the triangle delimited by the axes and the line $\{u+v=\max(r_u.r_v)/\min(\kappa_u, \kappa_v)\}$ is positively invariant for the flow, and $(0,0)$ is an ejective equilibrium point whenever $\lambda_A>0$. By an extension of the  ejective fixed point theorem to  flows \cite[Theorem 19]{Horn-1970}, there exists a nonejective equilibrium for \eqref{eq:syst-ode}, which proves the existence. 

	Next we focus on the estimates on statement \ref{item:lemestu^*-smallmu}. Since the statement is symmetric with respect to the variables $u$ and $v$, we only prove the result for $u^*$. 
	Assume first that $r_u-\mu_u>\mu_v>0$. Then $u^*$ satisfies:
	\begin{equation}\label{eq:estonu^*}
		0=u^*(r_u-\mu_u-\kappa_uu^*)+v^*(\mu_v-\kappa_uu^*).
	\end{equation}
	If $u^*<\frac{\mu_v}{\kappa_u}$, then both terms in the right-hand side of \eqref{eq:estonu^*} are positive, which is a contradiction. Similarly, if $u^*>\frac{r_u-\mu_u}{\kappa_u}$, then both terms are negative, which is also a contradiction. We conclude that $\frac{\mu_v}{\kappa_u}\leq u^*\leq \frac{r_u-\mu_u}{\kappa_u}$. Finally, if equality is achieved in the latter inequality, then one of the terms in \eqref{eq:estonu^*} is 0 and the other is positive, which is a contradiction. Thus
	\begin{equation*}
		\frac{\mu_v}{\kappa_u}<u^*<\frac{r_u-\mu_u}{\kappa_u}.
	\end{equation*}
	In the case $0<r_u-\mu_u<\mu_v$, a similar argument shows that 
	\begin{equation*}
		\frac{r_u-\mu_u}{\kappa_u}<u^*<\frac{\mu_v}{\kappa_u}.
	\end{equation*}
	Finally, if $r_u-\mu_u=\mu_v$, then both terms in the right-hand side of \eqref{eq:estonu^*} have the same sign independently of $u^*$, hence the only possibility is
	\begin{equation*}
		u^*=\frac{r_u-\mu_u}{\kappa_u}=\frac{\mu_v}{\kappa_u}.
	\end{equation*}
	Statement \ref{item:lemestu^*-smallmu} is proved. To show Statement \ref{item:lemestu^*-bigmu}, since $r_u-\mu_u\leq 0$, we simply rewrite \eqref{eq:estonu^*} as:
	\begin{equation*}
		u^*=\frac{\mu_v}{\kappa_u}+\frac{u^*}{\kappa_uv^*}(r_u-\mu_u-\kappa_uu^*)<\frac{\mu_v}{\kappa_u}.
	\end{equation*}
	This proves Statement (ii) and the proof of Lemma \ref{lem:stat-ode-uniqueness} is complete.
\end{proof}

We have seen above that the unique nontrivial nonnegative equilibrium point $(u^*,v^*)$ of \eqref{eq:syst-ode} is automatically positive. Now we discuss its linear stability. The Jacobian matrix of the nonlinearity $f:=(f^u,f^v)$ at $(u^*,v^*)$ is given in the form
	\begin{equation}\label{eq:abcd}
		D_{(u^*, v^*)}f=\left(\begin{matrix} r_u-\mu_u-\kappa_u(2u+v) & \mu_v - \kappa_uu \\
		\mu_u-\kappa_vv & r_v-\mu_v-\kappa_v(u+2v) \end{matrix}\right)=:
		\begin{pmatrix}
			a & b \\ c & d
		\end{pmatrix}.
	\end{equation}
	The eigenvalues of this matrix determines the linear stability of $(u^*,v^*)$.

\begin{lem}[Linear stability of the positive equilibrium]\label{lem:stat-ode-stability}
	Let $ r_u, r_v\in\R$, $\kappa_u>0$, $\kappa_v>0$, and $\mu_u>0$, $\mu_v>0$. Assume that $\lambda_A>0$ and let $(u^*, v^*)$ be the positive equilibrium point of \eqref{eq:syst-ode}. Then $(u^*, v^*)$ is linearly stable. More precisely, the constants $a, b, c, d$ in \eqref{eq:abcd} satisfy:
\[
a=-\left(\kappa_u u^*+\mu_v\frac{v^*}{u^*}\right)<0, \quad d=-\left(\kappa_v v^*+\mu_u\frac{u^*}{v^*}\right)<0,
\] 
    as well as:
    \begin{equation*}
		\text{\rm tr} (D_{(u^*, v^*)} f) =a+d < 0,\quad\ 
		\det (D_{(u^*, v^*)} f) =ad-bc >0.
	\end{equation*}
\end{lem}
\begin{proof}
Let us first remark that the equation satisfied by the equilibrium $(u^*, v^*)$ of \eqref{eq:syst-ode} can be written as 
	\begin{equation*}
		\begin{system}
			r_u-\mu_u-\kappa_u(u^*+v^*)=-\mu_v\frac{v^*}{u^*},\\
			r_v-\mu_v-\kappa_v(u^*+v^*)=-\mu_u\frac{u^*}{v^*}.
		\end{system}
	\end{equation*}
	Using the above relation, we have
	\begin{equation*}
		a+d = r_u-\mu_u-\kappa_u(2u^*+v^*) +  r_v-\mu_v-\kappa_v(u^*+2v^*) = -\mu_v \frac{v^*}{u^*}-\mu_u\frac{u^*}{v^*} -\kappa_u u^*-\kappa_vv^*<0.
	\end{equation*}
Computing further, we obtain	
	\begin{align*}
		ad-bc 
		 &= (r_u-\mu_u-\kappa_u(2u^*+v^*))(r_v-\mu_v-\kappa_v(u^*+2v^*)) - (\mu_v - \kappa_uu^*)(\mu_u-\kappa_vv^*) \\
		&= \left(\mu_v\frac{v^*}{u^*}+\kappa_uu^*\right)\left( \mu_u\frac{u^*}{v^*}+\kappa_vv^*\right) - (\mu_v - \kappa_uu^*)(\mu_u-\kappa_vv^*) \\
		&=\mu_v\kappa_v\frac{(v^*)^2}{u^*} + \mu_u\kappa_u\frac{(u^*)^2}{v^*}+\mu_v\kappa_uu^*+\mu_u\kappa_vv^*>0.
	\end{align*}
	The lemma is proved.
\end{proof}

\begin{rem}[Stability of $(0,0)$]\label{rem:stab-0-ode}
    The principal eigenvalue $\lambda_A$ can be computed explicitly as 
    \begin{equation*}
	\lambda_A=\frac{r_u-\mu_u+r_v-\mu_v+\sqrt{\big(r_u-\mu_u-(r_v-\mu_v)\big)^2+4\mu_u\mu_v}}{2}.
    \end{equation*}
    By a direct computation, one sees the following:
    \begin{itemize}\setlength{\itemsep}{0pt}
    \item[(i)] $\lambda_A$ is monotone increasing in both $r_u, r_v$ and $\lambda_A=0$ if $r_u=r_v=0$;
    \item[(ii)] If we fix the ratio between $\mu_u$ and $\mu_v$ as $\mu_u=\mu$, $\mu_v=\alpha\mu$, then $\lambda_A$ is monotone decreasing in $\mu$ and $\lambda_A\to\max(r_u,r_v)$ as $\mu\to 0$, while $\lambda_A\to\frac{\mu_v}{\mu_u+\mu_v}r_u+\frac{\mu_u}{\mu_u+\mu_v}r_v$ as $\mu\to\infty$.
    \end{itemize}
    From (i) above, we see that $\lambda_A>0$ if $r_u$, $r_v$ are both positive (hence $(0,0)$ is unstable), and $\lambda_A<0$ if $r_u$, $r_v$ are both negative (hence $(0,0)$ is stable). When $\max(r_u,r_v)>0$ but $r_u<0$ or $r_v<0$, then from (ii) above, we see that 
$(0,0)$ is always unstable if $\frac{\mu_v}{\mu_u+\mu_v}r_u+\frac{\mu_u}{\mu_u+\mu_v}r_v\geq0$, whereas if $\frac{\mu_v}{\mu_u+\mu_v}r_u+\frac{\mu_u}{\mu_u+\mu_v}r_v<0$, the stability of $(0,0)$ depends on the size of the mutation rate; roughly speaking, $(0,0)$ is unstable if $\mu_u, \mu_v$ are sufficiently small, and stable if $\mu_u, \mu_v$ are sufficiently large.
\end{rem}

We are now in a position to give our key arguments for the long-time behavior of the ODE problem. We begin with the case where a Lyapunov function exists for the system. We define:
\begin{equation}\label{eq:Lyapunov-func}
	\mathcal F_u(u):=u-u^*-u^*\ln\left(\frac{u}{u^*}\right), \qquad \mathcal F_v(v):=v-v^*-v^*\ln\left(\frac{v}{v^*}\right).
\end{equation}
Note that this Lyapunov function is rather classical and has been used for instance by \cite{Hsu-78} for competitive Lotka-Volterra systems. The present argument was inspired by  
\cite{Can-Cos-Yu-18}.

\begin{lem}[Lyapunov function]\label{lem:Lyapunov-ode}
	Let Assumption \ref{as:cond-instab-0} hold true, and assume that $\lambda_A>0$ and 
	that $\max(r_u-\mu_u, r_v-\mu_v)>0$. Then there is $K>0$ such that the function $\mathcal F^K(u, v):=\mathcal F_u(u)+K\mathcal F_v(v)$ is a Lyapunov function for \eqref{eq:syst-ode}, that is, for any positive solution $(u(t), v(t))$ of \eqref{eq:syst-ode}, 
	\begin{equation*}
		\frac{\dd}{\dd t}\mathcal F^K(u(t),v(t))\leq 0 \quad \hbox{for}\ \ t\geq 0. 
	\end{equation*}
	Moreover the inequality is strict unless $(u(t), v(t))=(u^*, v^*)$. 
\end{lem}

\begin{proof}
	Since it is clear that $\mathcal F(u^*, v^*)=0$, we will focus on the case of a solution orbit $\big(u(t),v(t)\big)$ starting from a positive initial condition $(u_0, v_0)$. We first compute: 
	\begin{align*}
		\frac{\dd}{\dd t}\mathcal F_u(u(t))&= u_t\left(1-\frac{u^*}{u}\right)= (u-u^*)\frac{u_t}{u}\\
		&=(u-u^*)\left(r_u-\mu_u-\kappa_u(u+v)+\mu_v\frac{v}{u}\right) \\
		&=(u-u^*)\left(\kappa_u(u^*+v^*)-\mu_v\frac{v^*}{u^*}-\kappa_u(u+v)+\mu_v\frac{v}{u}\right) \\ 
		&=-\kappa_u(u-u^*)^2 -\kappa_u(u-u^*)(v-v^*)+\mu_v(u-u^*)\left(\frac{u^*v-uv^*}{uu^*}\right) \\
		&=-\left(\kappa_u+\mu_v\frac{v^*}{uu^*}\right)(u-u^*)^2 -\left(\kappa_u-\frac{\mu_v}{u^*}\right)(u-u^*)(v-v^*) \\
		&\leq-\kappa_u(u-u^*)^2-\left(\kappa_u-\frac{\mu_v}{u^*}\right)(u-u^*)(v-v^*),
	\end{align*}
	and the  inequality is strict unless $u=u^*$. Similarly, 
	\begin{equation*}
		\frac{\dd}{\dd t}\mathcal F_v(y)\leq-\kappa_v(v-v^*)^2-\left(\kappa_v-\frac{\mu_u}{v^*}\right)(u-u^*)(v-v^*),
	\end{equation*}
	and the inequality is strict unless $v=v^*$. Since $(u, v)\neq (u^*, v^*)$, we have for all $K>0$: 
	\begin{multline*}
		\frac{\dd}{\dd t}\mathcal F^K(u,v)<-\kappa_u(u-u^*)^2 -\left(\kappa_u-\frac{\mu_v}{u^*} + K\left(\kappa_v-\frac{\mu_u}{u^*}\right)\right)(u-u^*)(v-v^*)-K\kappa_v(v-v^*)^2.
	\end{multline*}
	Next we prove that the right-hand side can be made nonpositive for all $(u,v)>(0,0)$ for a well-chosen value of $K$. We remark that the right-hand side is a quadratic form in $(U:=u-u^*, V:=v-v^*)$, which can be written as $-Q(U,V)$ where:
	\begin{equation}\label{eq:QUV}
		Q(U,V):=AU^2+(B+KC)UV+KDV^2, 
	\end{equation}
	and $U=u-u^*$, $V=v-v^*$, $A=\kappa_u$, $B=\kappa_u-\frac{\mu_v}{u^*}$, $C=\kappa_v-\frac{\mu_u}{v^*}$ and $D=\kappa_v$. We claim that $Q(U,V)$ can be made positive definite by a proper choice of $K>0$. Indeed, algebraic computations lead to 
	\begin{align*}
		Q(U,V)&=A\left(U+\frac{B+KC}{2A}V\right)^2+\left(KD-\frac{(B+KC)^2}{4A}\right)V^2, 
	\end{align*}
	and therefore it suffices to find $K>0$ such that 
	\begin{equation*}
		0<KD-\frac{(B+KC)^2}{4A}=\dfrac{-C^2K^2+(4AD-2BC)K-B^2}{4A}=:\frac{P(K)}{4A}.
	\end{equation*}
	Here $P(K)$ is a quadratic polynomial and the number of its real roots is determined by the sign of the quantity
	\begin{equation*}
		\Delta=(4AD-2BC)^2-4B^2C^2=16AD(AD-BC)>0. 
	\end{equation*}
	If $BC<AD$, the polynomial $P$ has two real roots, and those roots have to be nonnegative since $P(K)<0$ for all $K<0$. This implies that there exists $K>0$ with $P(K)>0$, which will prove our claim and consequently will complete the proof of Lemma \ref{lem:Lyapunov-ode}.

	Our last task is therefore to check that $BC<AD$. Assume first that $r_u-\mu_u>0$ and $r_v-\mu_v>0$, then $B=\kappa_u-\frac{\mu_v}{u^*}>0$ and $C=\kappa_v-\frac{\mu_u}{v^*}>0$ are both positive by Lemma \ref{lem:stat-ode-uniqueness}. Thus,
	\begin{equation*}
		BC=\left(\kappa_u-\frac{\mu_v}{u^*}\right)\left(\kappa_v-\frac{\mu_u}{v^*}\right)\leq \kappa_u\kappa_v=AD.
	\end{equation*}
	Next assume that $r_u-\mu_u\leq 0$ and $r_v-\mu_v>0$ (the case $r_v-\mu_v\leq 0$ and $r_u-\mu_u>0$ can be treated similarly). In this case, $\kappa_u-\frac{\mu_v}{u^*}\leq 0$ and $\kappa_v-\frac{\mu_u}{v^*}>0$ and thus
	\begin{equation*}
		BC=\left(\kappa_u-\frac{\mu_v}{u^*}\right)\left(\kappa_v-\frac{\mu_u}{v^*}\right)\leq0< \kappa_u\kappa_v=AD.
	\end{equation*}
	Hence $BC<AD $ always holds under our hypotheses. Lemma \ref{lem:Lyapunov-ode} is proved.
\end{proof}

Notice in particular that Proposition \ref{prop:longtime-ode}  (i) follows directly from Lemma \ref{lem:Lyapunov-ode} in the case $\max(r_u-\mu_u, r_v-\mu_v)>0$. 
Next we consider the case $\max(r_u-\mu_u, r_v-\mu_v)\leq 0$. In this case, we show that the dynamics is eventually cooperative and we deduce the conclusion by comparison arguments.

\begin{lem}[Ultimately cooperative dynamics]\label{lem:coop}
	Let Assumption \ref{as:cond-instab-0} hold, and assume that $\lambda_A>0$ and that $\max(r_u-\mu_u, r_v-\mu_v)\leq 0$. 
	Then any positive solution $(u(t), v(t))$ of \eqref{eq:syst-ode} satisfies
\begin{equation}\label{u-to-u*}
		\lim_{t\to+\infty}(u(t), v(t))=(u^*, v^*).
	\end{equation}
\end{lem}
\begin{proof}
	Let $(u(t), v(t))$ be a positive solution to \eqref{eq:syst-ode}. Then $(u(t), v(t))$ is a subsolution to the cooperative system:
	\begin{equation}\label{upper-ODE}
		\left\{\begin{aligned}
			\bar u_t&=\bar u(r_u-\mu_u-\kappa_u\bar u)+\bar v\max(\mu_v-\kappa_u\bar u, 0),\\
			\bar v_t&=\bar v(r_v-\mu_v-\kappa_v\bar v)+\bar u\max(\mu_u-\kappa_v\bar v, 0).
		\end{aligned}\right.
	\end{equation}
	Let $(\bar u(t), \bar v(t))$ be a solution of \eqref{upper-ODE} with the following initial data:
	\[
	\bar u(0) = \max\Big(u(0), \frac{\mu_v}{\kappa_u}\Big),\quad 
	\bar v(0) = \max\Big(v(0), \frac{\mu_u}{\kappa_v}\Big). 
	\]
	Since $\max(r_u-\mu_u, r_v-\mu_v)\leq 0$, we have $\bar u_t(0)<0$, $\bar v_t(0)<0$. And since \eqref{upper-ODE} is a cooperative system, the comparison principle implies that $(\bar u(t), \bar v(t))$ is monotone decreasing in $t\geq 0$. 
	Therefore $(\bar u(t), \bar v(t))$ eventually enters the cooperative zone defined in \eqref{cooperative-zone}, {\it i.e.}, $0<\bar u< \frac{\mu_v}{\kappa_u}$, $0<\bar v< \frac{\mu_u}{\kappa_v}$, and converges to an equilibrium point $(\bar u^*, \bar v^*)$ of \eqref{upper-ODE} as $t\to+\infty$. 
	Since the systems \eqref{upper-ODE} and \eqref{eq:syst-ode} are identical in the cooperative zone, $(\bar u^*, \bar v^*)$ is also an equilibrium point of \eqref{eq:syst-ode}. 
	In view of this and the fact that $\bar u(0)\geq \frac{\mu_v}{\kappa_u}>u^*$, $\bar v(0)\geq \frac{\mu_u}{\kappa_v}>v^*$, and recalling the uniqueness of the positive equilibrium point, we see that $(\bar u^*, \bar v^*)=(u^*,v^*)$. Since $(u(t), v(t))$ is a subsolution of \eqref{upper-ODE} such that $u(0)\leq \bar u(0)$ and $v(0)\leq \bar v(0)$, we have $u(t)\leq \bar u(t)$, $v(t)\leq \bar v(t)$ for all $t\geq 0$. Hence
	\begin{equation}\label{u<u*}
	\limsup_{t\to\infty}u(t)\leq \lim_{t\to\infty}\bar u(t)=u^*,\quad 
	\limsup_{t\to\infty}v(t)\leq \lim_{t\to\infty}\bar v(t)=v^*.
	\end{equation}
	
Next let $(\varphi_A^u, \varphi_A^v)^T$ denote the positive eigenvector of the matrix $A$ corresponding to $\lambda_A$, and let $(u^\ep(t),v^\ep(t))$ be the solution of \eqref{eq:syst-ode} with initial data $(u^\ep(0),v^\ep(0))=\ep (\varphi_A^u, \varphi_A^v)$. Since $\lambda_A>0$, the following inequalities hold if $\ep>0$ is chosen sufficiently small:
\[
u^\ep_t(0)=\lambda_A \ep\varphi_A^u+o(\ep)>0,\quad 
v^\ep_t(0)=\lambda_A \ep\varphi_A^v+o(\ep)>0.
\]
Consequently, $(u^\ep(t),v^\ep(t))$ is monotone increasing $t$ so long as it stays in the cooperative zone. 

By \eqref{u<u*}, there exists $t_1\geq 0$ such that $(u(t),v(t))$ lies in the cooperative zone for all $t\geq t_1$. Replacing $\ep>0$ by a smaller constant if necessary, we may assume that $0<u^\ep(0) < u(t_1)$, $0<v^\ep(0) < v(t_1)$. 
Then $(u^\ep(t),v^\ep(t))$ remains in the cooperative zone and satisfies 
\begin{equation}\label{uep<u}
u^\ep(t)<u(t+t_1), \ \ v^\ep(t)<v(t+t_1)\quad \hbox{for all}\ \  t\geq 0.
\end{equation}
Hence $(u^\ep(t),v^\ep(t))$ converges monotonically to an equilibrium point of \eqref{eq:syst-ode}, which coincides with $(u^*,v^*)$ by the uniqueness of the positive equilibrium point. This and \eqref{uep<u} imply
\[
\liminf_{t\to\infty}u(t)\geq \lim_{t\to\infty}u^\ep(t)=u^*,\quad 
	\liminf_{t\to\infty}v(t)\geq \lim_{t\to\infty}v^\ep(t)=v^*.
\]
Combining this with \eqref{u<u*}, we obtain \eqref{u-to-u*}. The lemma is proved.
\end{proof}

We are now in a position to prove Proposition \ref{prop:longtime-ode} \textit{(i)} and conclude this section:

\begin{proof}[Proof of Proposition \ref{prop:longtime-ode} \textit{(i)}]
    If $\lambda_A>0$, the existence and uniqueness of a stationary solution $(u^*, v^*)$ has been shown in Lemma \ref{lem:stat-ode-uniqueness}. The convergence of $(u(t), v(t))$ when $t\to+\infty$ has been shown in Lemma \ref{lem:Lyapunov-ode} for the case $\max(r_u-\mu_u, r_v-\mu_v)>0$ using a Lyapunov function,  
and in Lemma \ref{lem:coop} for the case 
$\max(r_u-\mu_u, r_v-\mu_v)\leq 0$ by means of comparison arguments. 
This covers all the cases therefore completes the proof of the statement \textit{(i)} of Proposition \ref{prop:longtime-ode}.
\end{proof}


\subsubsection{Asymptotic behavior of the homogeneous RD problem}
\label{sss:hom-rd}

In this section we prove Theorem~\ref{thm:ltb} on the convergence of solutions of \eqref{eq:syst-hom-rd} to the positive equilibrium point $(u^*,v^*)$. For that purpose, we first prove the following Liouville type result which states that any entire solution of \eqref{eq:syst-hom-rd} that is uniformly positive is identically equal to $(u^*,v^*)$. 

\begin{thm}[Liouville type resut]\label{thm:entire-sol-hom}
	Let Assumption \ref{as:cond-instab-0} hold and assume that $\lambda_A>0$. Let $(u(t, x), v(t,x))$ be a nonnegative bounded entire solution to \eqref{eq:syst-hom-rd}. 
	Assume that $(u,v)$ is uniformly positive, that is, 
	there exists $\delta>0$ such that 
	\begin{equation*}
		u(t, x)\geq \delta, \ \  v(t,x)\geq \delta\quad \hbox{for all}\ \ t\in\R ,\ x\in\R . 
	\end{equation*}
	Then $(u, v)\equiv (u^*, v^*)$.
\end{thm}

\begin{proof}
	We divide the proof in three steps.

\vskip 3pt
	\begin{stepping}
		\step \underbar{The ultimately cooperative case: $\max(r_u-\mu_u, r_v-\mu_v)\leq 0$.}

In this case, our argument is partly similar to the proof of Lemma~\ref{lem:coop}. Define
\[
	M:=\sup_{(t,x)\in\R ^2} \max\big(u(t,x),v(t,x)\big)
\]
and let $(\bar u(t),\bar v(t))$ be the solution of the ODE system \eqref{upper-ODE} for the following initial data:
\[
	\bar u(0)=\max\Big(M,\frac{\mu_v}{\kappa_u}\Big), 
	\quad
	\bar v(0)=\max\Big(M,\frac{\mu_u}{\kappa_v}\Big).
\]
As we have seen in the proof of Lemma \ref{lem:coop}, $(\bar u(t), \bar v(t))\to (u^*,v^*)$ as $t\to+\infty$. Note also that $(\bar u(t),\bar v(t))$ can be identified with a spatially uniform solution of the following reaction-diffusion system:
	\begin{equation}\label{upper-RD}
	\begin{system}
		\relax &u_t-\sigma u_{xx}=(r_u-\mu_u-\kappa_u\bar u)\bar u+\bar v\max(\mu_v-\kappa_u\bar u, 0), \\
		\relax &v_t-\sigma v_{xx}=(r_v-\mu_v-\kappa_v\bar v)\bar v+\bar u\max(\mu_u-\kappa_v\bar v, 0).
	\end{system}
	\end{equation}
		
Let $t_0\in\R $ be an arbitrary real number. What we have to show is that 
\begin{equation}\label{u(t0,x)=u^*}
	\big(u(t_0,x), v(t_0,x)\big)=\big(u^*,v^*\big)\quad \hbox{for all}\ \ x\in\R .
\end{equation}		
Choose $T>0$ arbitrarily and define $U(t,x)=u(t+t_0-T,x)$, $V(t,x)=v(t+t_0-T,x)$. Then $(U(t,x), V(t,x))$ is a solution of \eqref{eq:syst-hom-rd} and satisfies $U(0,x)\leq M\leq \bar u(0)$, $V(0,x)\leq M\leq \bar v(0)$ for all $x\in\R $. Since any solution of \eqref{eq:syst-hom-rd} is a subsolution to the cooperative system \eqref{upper-RD}, we have $U(t,x)\leq \bar u(t)$, $V(t,x)\leq \bar v(t)$ for all $t\geq 0$, $x\in\R $. Setting $t=T$, we obtain
\[
		u(t_0,x)=U(T,x)\leq \bar u(T,x),\ \ v(t_0,x)=V(T,x)\leq \bar v(T,x)\quad \hbox{for all}\ \ x\in\R .
\]
Now we let $T\to\infty$. Then $(\bar u(T),\bar v(T))\to (u^*,v^*)$, hence
\begin{equation}\label{u(t,x)<u*}
	\sup_{x\in\R }u(t_0,x)\leq u^*,\quad 
	\sup_{x\in\R }v(t_0,x)\leq v^*.
\end{equation}
	
In order to obtain a lower estimate, let $(u^\ep(t),v^\ep(t))$ be the solution of \eqref{eq:syst-ode} with initial data $(u^\ep(0),v^\ep(0))=\ep (\varphi_A^u, \varphi_A^v)$, as in the proof of Lemma~\ref{lem:coop}, and let $t_1\geq 0$ be such that $(U(t,x), V(t,x))$ lies in the cooperative zone for all $t\geq t_1$. Then, choosing $\ep>0$ sufficiently small and arguing as in the proof of Lemma~\ref{lem:coop}, we obtain $u^\ep(t)<U(t+t_1,x)$ and $v^\ep(t)<V(t+t_1,x)$ for all $t\geq 0$, $x\in\R $. Now we set $t=T-t_1$. Then we have
\[
u^\ep(T-t_1)<U(T,x)=u(t_0,x), \ \ v^\ep(T-t_1)<V(T,x)=v(t_0,x)\quad \hbox{for all}\ \  x\in\R .	
\]
Now we let $T\to \infty$. Then $(u^\ep(T-t_1),v^\ep(T-t_1))\to (u^*,v^*)$, hence
\[
	\inf_{x\in\R }u(t_0,x)\geq u^*,\quad 
	\inf_{x\in\R }v(t_0,x)\geq v^*.
\]
Combining this with \eqref{u(t,x)<u*}, we obtain \eqref{u(t0,x)=u^*}, as desired.

\medskip
		\step \underbar{The Lyapunov case: $\max(r_u-\mu_u, r_v-\mu_v)\geq 0$.}

In this case we use a generalisation of the Lyapunov argument used in the proof of Lemma~\ref{lem:Lyapunov-ode}. Let $\mathcal F_u$, $\mathcal F_v$ be the functions defined in \eqref{eq:Lyapunov-func} and $K$ be the constant given by Lemma \ref{lem:Lyapunov-ode}, so that $\mathcal F^K(u, v):=\mathcal F_u(u)+K\mathcal F_v(v)$ is a Lyapunov function for the flow of the ODE \eqref{eq:syst-ode}. Define $w(t, x)=\mathcal F^K(u(t,x), v(t,x))$. 
Then, since $\mathcal F_u''(u)\geq 0$, $\mathcal F_v''(v)\geq 0$ for all $u,v$, $w$ satisfies:
\begin{align*}
	w_t-\sigma w_{xx}&=(u_t-\sigma u_{xx})\mathcal F_u'(u)+K(v_t-\sigma v_{xx})\mathcal F_v'(v)-\sigma(u_x^2\mathcal F_u''(u) + Kv_x^2\mathcal F_v''(v))\\
	&\leq -\kappa_u(u-u^*)^2-\left(\kappa_u-\frac{\mu_v}{u^*}+K\left(\kappa_v-\frac{\mu_u}{v^*}\right)\right)-K\kappa_v(v-v^*)^2\\
	&=- Q(u-u^*,v-v^*),
\end{align*}
where $Q$ is the quadratic form defined in \eqref{eq:QUV}. 
As we have shown in the proof of Lemma~\ref{lem:Lyapunov-ode}, $Q(u-u^*, v-v^*)>0$ whenever $(u, v)\neq (u^*, v^*)$. Since $-Q(u,v)\leq 0$, $w$ is a bounded entire subsolution to the heat equation, therefore it has to be a constant. 
And since $-Q(u,v)<0$ whenever $(u, v)\neq (u^*, v^*)$, the only possibility is $w(t, x)\equiv 0$, which implies $(u(t, x), v(t,x))\equiv (u^*, v^*)$. 
This establishes the claim of the theorem for the case $\max(r_u-\mu_u, r_v-\mu_v)\geq 0$.
	\end{stepping}

Combining Step 1 and Step 2 completes the proof of Theorem \ref{thm:entire-sol-hom}.
\end{proof}

\begin{proof}[Proof of Theorem \ref{thm:ltb}] 
As the formula $c^*_{R}=c^*_{L}=2\sqrt{\sigma\lambda_A}$ is already given in \eqref{c*-homo}, in what follows we focus on the assertion \eqref{u(t,x)-to-u*}.

We argue by contradiction. Suppose that \eqref{u(t,x)-to-u*} does not hold for  $(u(t,x), v(t,x))$. Then there exists $\ep>0$ and  sequences $t_n\to+\infty$ and $x_n\in\R$ such that $|x_n|\leq ct_n$ and that 
	\begin{equation*}
		\max\big(|u(t_n, x_n)-u^*|,|v(t_n, x_n)-v^*|\big)\geq\ep>0 \quad\hbox{for}\ \ n=1,2,3,\ldots.
	\end{equation*}
By the classical parabolic estimates, the sequence $(u(t+t_n, x),v(t+t_n, x)$ has a subsequence that converges locally uniformly to an entire solution $(u^\infty(t, x), v^\infty(t, x))$ which satisfies 
\begin{equation}\label{u-infty(0)-u*}
\max\big(|u^\infty(0, 0)-u^*|, |v^\infty(0, 0)-v^*|\big)\geq \ep.
\end{equation}
Now choose $c'$ such that $0<c<c'<c^*_R$. Then by Theorem~\ref{thm:hairtrigger}, there exists $\eta>0$ such that
\[
\liminf_{t\to\infty} \left[\inf_{|x|\leq c't}\min\big(u(t,x),v(t,x)\big)\right]\geq \eta.
\]
Since $|x_n|\leq ct_n$ and $c'>c$, we see from the above inequality that $u^\infty(t, x)\geq \eta$, $v^\infty(t, x)\geq \eta$ for all $t\in\R ,\,x\in\R $. Hence, by Theorem \ref{thm:entire-sol-hom}, we have $(u^\infty(t, x), v^\infty(t, x))\equiv (u^*, v^*)$, but this contradicts \eqref{u-infty(0)-u*}. This contradiction proves \eqref{u(t,x)-to-u*}. The proof of Theorem~\ref{thm:ltb} is complete.
\end{proof}

\subsection{Proof of homogenization}
\label{ss:homogenization-proof}

In this section we prove Theorem~\ref{thm:rapidosc} on the homogenization of the system \eqref{eq:rapidosc}. We start with a lemma concerning the homogenization of associated eigenproblems.

\begin{lemma}[Homogenization of the eigenproblems]\label{lem:homogenization-eigenproblem}
    For each $\lambda\in\R$, let $\big(k^\ep(\lambda), (\varphi^\ep,\psi^\ep)\big)$ be the principal eigenpair of \eqref{eq:lambda-periodic-principal-eigen} for the coefficients \eqref{rapid-osc-coeffs} under the following normalization condition:
    \begin{equation}\label{eq:hom-norm}
	\big(\Vert \big(\varphi^\ep, \psi^\ep)\Vert_{L^2(0, \ep)^2}\big)^2 = \int_0^\ep \big(\varphi^\ep(x)\big)^2+\big(\psi^\ep(x)\big)^2 \dd x = \ep \quad\ \ (0<\ep<1).
    \end{equation}
Then $k^\ep(\lambda)\to k^0(\lambda)$ as $\ep\to 0$ and the vector function $\big(\varphi^\ep, \psi^\ep\big)$ converges to the constant function $(\varphi^0, \psi^0)$ uniformly on $\R$, where $\big(k^0(\lambda), (\varphi^0, \psi^0)\big) $ denotes the principal eigenpair of \eqref{eq:lambda-periodic-eigen-homo} 
with the homogenized coefficients given in \eqref{eq:mean-coeffs} and under the normalization condition $(\varphi^0)^2+(\psi^0)^2=1$:
\begin{equation}\label{lambda-periodic-eigen-homogenized}
\begin{system}
	\relax & \big(\lambda^2\overline{\sigma}^H+\overline{r_u}\big)\varphi^0+\overline{\mu_v}\psi^0-\overline{\mu_u}\varphi^0=k^0(\lambda) \varphi^0, \\
	\relax & \big(\lambda^2\overline{\sigma}^H+\overline{r_v}\big)\psi^0+\overline{\mu_u}\varphi^0-\overline{\mu_v}\psi^0=k^0(\lambda) \psi^0.
\end{system}
\end{equation}
\end{lemma}

\begin{proof}
We first note that, since the coefficients in \eqref{rapid-osc-coeffs} are $\ep$-periodic, the uniqueness of the principal eigenpair of \eqref{eq:lambda-periodic-principal-eigen} implies that $\big(\varphi^\ep, \psi^\ep\big)$ are also $\ep$-periodic.

Fix $\lambda\in\R$ arbitrarily. Then $k^\ep(\lambda) $ is uniformly bounded as $\ep$ varies. This follows from the inequalities \eqref{k-quadratic}, since the maxima and minima of $\sigma^\ep(x)$, $r_u^\ep(x)$, $r_v^\ep(x)$ do not depend on $\ep$. 
For notational simplicity, we denote these bounds (for a fixed $\lambda$) by $C_1, C_2$, that is, 
    \begin{equation}\label{eq:bound-keps}
	C_1\leq k^\ep(\lambda)\leq C_2.
    \end{equation}

    Next we show that, given any $R>0$, the family $(\varphi^\ep, \psi^\ep)$ is uniformly bounded in $H^1(-R, R)^2$ as $\ep$ varies.  Indeed, multiplying the first line of \eqref{eq:lambda-periodic-principal-eigen2} by $\varphi^\ep$ and integrating by parts, we have
	\begin{multline*}
    \int_0^\ep \sigma^\ep(x)\left(\varphi^\ep_x\right)^2 \dd x
    =\int_0^\ep\left(\lambda^2\sigma^\ep(x) + r_u^\ep(x)-\mu_u^\ep(x)-k^\ep(\lambda)\right)\left(\varphi^\ep\right)^2 \dd x
	    +\int_0^\ep \mu_v^\ep(x)\varphi^\ep\psi^\ep \dd x.
	\end{multline*}
    Since the coefficients $\sigma^\ep, r_u^\ep, \mu_u^\ep, \mu_v^\ep$, as well as $k^\ep(\lambda)$, are uniformly bounded, and since $\sigma^\ep(x)\geq \sigma_{\min}:=\min_{y\in[0,1]}\sigma(y)$, the above identity and the normalization condition \eqref{eq:hom-norm} imply 
    \[
    \int_0^\ep \left(\varphi^\ep_x\right)^2 \dd x={\mathcal O}(\ep).
    \]
    The same estimate holds for $\psi^\ep$. Thus, recalling again the normalization condition \eqref{eq:hom-norm}, we have
    \[
   \left(\Vert \big(\varphi^\ep, \psi^\ep\big)\Vert_{H^1(0, \ep)^2}\right)^2={\mathcal O}(\ep).
    \]
    In view of this and the $\ep$-periodicity of $(\varphi^\ep, \psi^\ep)$, we see that there exists a constant $C_3>0$ that is independent of $\ep$ such that 
    \begin{equation}\label{eq:bound-phi-psi-eps}
	\left(\Vert \big(\varphi^\ep, \psi^\ep\big)\Vert_{H^1(-R, R)^2}\right)^2 \leq 
	C_3 R \quad (0<\ep<1).
    \end{equation}

    Next we prove that the auxiliary functions $\xi^\ep(x):=\sigma^\ep(x)\big(\varphi^\ep_x(x)-\lambda \varphi^\ep(x)\big)$ and $\zeta^\ep(x):= \sigma^\ep(x)\big(\psi^\ep_x(x)-\lambda \psi^\ep(x)\big)$ are also uniformly bounded in $H^1(-R, R)$. Indeed \eqref{eq:lambda-periodic-principal-eigen2} {can} be rewritten as:
    \begin{equation}\label{eq:syst-rewritten}
	\left\{\begin{aligned}\relax
		-\xi^\ep_x=-\big(\sigma^\ep\big(\varphi^\ep_x-\lambda \varphi^\ep)\big)_x=-\lambda\sigma^\ep (\varphi^\ep_x-\lambda\varphi^\ep) &+\big(r_u^\ep(x)-\mu_u^\ep(x)\big)\varphi^\ep(x)+\mu_v^\ep\psi^\ep(x)\\
		&-k^\ep(\lambda)\varphi^\ep,\\
		-\zeta^\ep_x=-\big(\sigma^\ep\big(\psi^\ep_x-\lambda \psi^\ep)\big)_x=-\lambda\sigma^\ep (\psi^\ep_x-\lambda\psi^\ep) &+\big(r_v^\ep(x)-\mu_v^\ep(x)\big)\psi^\ep(x)+\mu_u^\ep\varphi^\ep(x)\\
		&-k^\ep(\lambda)\psi^\ep.
	\end{aligned}\right.
    \end{equation}
    This and \eqref{eq:bound-phi-psi-eps} imply that $\Vert\xi^\ep_x\Vert_{L^2(-R, R)}$ and $\Vert \zeta^\ep_x\Vert_{L^2(-R,R)}$ are uniformly bounded, hence
    \begin{equation}\label{eq:bound-xi-zeta-eps}
	\left(\Vert \big(\xi^\ep, \zeta^\ep\big)\Vert_{H^1(-R, R)^2}\right)^2 \leq C_4 R \quad (0<\ep<1)
    \end{equation}
for some constant $C_4>0$ that is independent of $\ep$.

Let $(\ep_n)$ be any sequence with $\ep_n\to 0$. By \eqref{eq:bound-keps}, \eqref{eq:bound-phi-psi-eps} and \eqref{eq:bound-xi-zeta-eps}, 
we can extract a subsequence, again denoted by $(\ep_n)$, such that $k^{\ep_n}(\lambda)\to k^0(\lambda)$ and that
\begin{equation}\label{conv-H1-weak}
	(\varphi^{\ep_n}, \psi^{\ep_n})\rightharpoonup(\varphi^0, \psi^0), \ \  
	(\xi^{\ep_n}, \zeta^{\ep_n})\rightharpoonup(\xi^0, \zeta^0) \ \ \ \hbox{as}\ \ n\to\infty\quad \hbox{weakly in } H^1_{loc}(\R)^2
\end{equation}
for some real number $k^0(\lambda)$ and $(\varphi^0, \psi^0), (\xi^0, \zeta^0)\in H^1_{loc}(\R)^2$. 

Since $H^1(-R, R)$ is compactly embedded in $C([-R, R])$, the convergence in \eqref{conv-H1-weak} is uniform on any interval $[-R,R]$. Furthermore, by 
the embedding $H^1\hookrightarrow C^{1/2}$, the functions $\varphi^{\ep_n}$ are uniformly $1/2$-H$\ddot{\rm o}$lder continuous. In view of this and the $\ep$-periodicity of $\varphi^\ep$, we see that $\max \varphi^\ep(x)-\min \varphi^\ep(x)={\mathcal O}(\ep^{1/2})$, hence $\varphi_0$ is a constant function. The same holds for $\psi^0$, $\xi^0$, $\zeta^0$. 

Since $\varphi^{\ep_n}\to\varphi^0$ uniformly and since $r_u^\ep(x):=r_u(\frac{x}{\ep})$ is bounded and $\ep$-periodic,
\[
r_u^{\ep_n}(x)\varphi^{\ep_n}-\overline{r_u}\varphi^0=
r_u^{\ep_n}(x)\left(\varphi^{\ep_n}-\varphi^0\right)+\left(r_u^{\ep_n}(x)-\overline{r_u}\right)\varphi^0\rightharpoonup 0 \quad \hbox{as}\ \ n\to\infty
\]
weakly in $L^2_{loc}(\R)$. 
Repeating the same argument and recalling that $(\xi^{\ep_n}, \zeta^{\ep_n})\rightharpoonup(\xi^0, \zeta^0)$ weakly in $H^1_{loc}(\R)^2$, we see that \eqref{eq:syst-rewritten} converges to the following system as $n\to\infty$ weakly in $L^2_{loc}(\R)^2$:
\begin{equation}\label{eq:limit-lambda-eigen}
\begin{cases}
0=-\xi^0_x= -\lambda \xi^0 + \big(\overline{r_u}-\overline{\mu_u}\big)\varphi^0  + \overline{\mu_v}\psi^0 - k^0(\lambda) \varphi^0,\\
0=-\zeta^0_x= -\lambda \zeta^0 + \big(\overline{r_v}-\overline{\mu_v}\big)\psi^0  + \overline{\mu_u}\varphi^0 - k^0(\lambda) \psi^0.
\end{cases}
\end{equation}
Here we used that fact that $\xi^0, \zeta^0$ are constant functions. Observe also that
    \[
    \varphi^{\ep_n}_x-\lambda \varphi^{\ep_n}=\frac{1}{\sigma^{\ep_n}(x)}\xi^\ep\rightharpoonup \int_0^1 \frac{\dd y}{\sigma(y)}\, \xi^0
    \]
weakly in $L^2_{loc}(\R)$. 
Thus $\xi^0=\overline{\sigma}^H \big(\varphi^0_x-\lambda \varphi^0\big)$, but since $\varphi^0_x=0$, we have $\xi^0=-\overline{\sigma}^H \lambda \varphi^0$, and similarly $\zeta^0 = -\overline{\sigma}^H \lambda \psi^0$. 
Substituting these into \eqref{eq:limit-lambda-eigen}, we obtain \eqref{lambda-periodic-eigen-homogenized}. The condition $(\varphi^0)^2+(\psi^0)^2=1$ follows from \eqref{eq:hom-norm}. 
Since the principal eigenpair of \eqref{lambda-periodic-eigen-homogenized} is unique, the limit $\big(k^0(\lambda),(\varphi^0,\psi^0)\big)$ does not depend on the choice of the sequence $(\ep_n)$. Hence $\big(k^\ep(\lambda),(\varphi^\ep,\psi^\ep)\big)$ converges to $\big(k^0(\lambda),(\varphi^0,\psi^0)\big)$ as $\ep\to 0$. This completes the proof of Lemma~\ref{lem:homogenization-eigenproblem}.
\end{proof}

\begin{lem}[Convergence of the spreading speeds]\label{lem:convergence-speeds}
Let the assumptions of Theorem~\ref{thm:rapidosc} hold and $k^\ep(\lambda)$ be as in Lemma~\ref{lem:homogenization-eigenproblem}. Then $\min_{\lambda\in\R}k^\ep(\lambda)>0$ for all sufficiently small $\ep>0$. Furthermore, \eqref{convergence-speeds} holds. 
\end{lem}

\begin{proof}
Let $k^0(\lambda)$ be as in Lemma~\ref{lem:homogenization-eigenproblem}. Then, by \eqref{k(lambda)-homo}, $k^0(\lambda)=\overline{\sigma}^H\lambda^2+\lambda_A$. Since we are assuming $\lambda_A>0$, we have $k^0(\lambda)>0$ for all $\lambda\in\R$. 
Next we note that the convergence $k^\ep(\lambda)\to k^0(\lambda)$ in Lemma~\ref{lem:homogenization-eigenproblem} is locally uniform in $\lambda\in\R$. This is because pointwise convergence of a sequence of convex functions is uniform on bounded sets (see \cite[Theorem 10.8]{Rockafellar-1970}). 
Moreover, by \eqref{k-quadratic}, there exists $M>0$ such that $k^\ep(\lambda)>0$ for $|\lambda|>M$ for any $\ep>0$. Thus it suffices to show $k^\ep(\lambda)>0$ for $|\lambda|\leq M$. The uniform convergence $k^\ep(\lambda)\to k^0(\lambda)$ on $-M\leq \lambda \leq M$ proves this claim.

Next we prove \eqref{convergence-speeds}. Let $\ep$ be small enough so that $\min_{\lambda\in\R}k^\ep(\lambda)>0$. By \eqref{k-quadratic}, there exists $M_1>0$ independent of $\ep$ such that $\inf_{\lambda>0} k^\ep(\lambda)/\lambda$ and $\inf_{\lambda<0} k^\ep(\lambda)/|\lambda|$ are both attained on the interval $|\lambda|\leq M_1$. Since $k^\ep(\lambda)\to k^0(\lambda)$ uniformly on $|\lambda|\leq M_1$, the claim \eqref{convergence-speeds} follows. 
\end{proof}

\begin{lem}[Homogenization limit of entire solutions]\label{lem:rapid-osc-lim}
	Let the assumptions of Theorem~\ref{thm:rapidosc} hold. 
	For each small $\ep>0$, let 
$(u^\ep(t, x), v^\ep(t, x)) $ be an entire solution to \eqref{eq:rapidosc} that is bounded from above and from below by positive constants, that is, $m^\ep\leq u^\ep(t,x)+v^\ep(t,x)\leq M^\ep$ for all $t\in\R $, $x\in\R $ and for some constants $M^\ep, m^\ep>0$. 
Then, as $\ep\to 0$, $(u^\ep(t, x), v^\ep(t, x))$ converges locally uniformly to the unique positive stationary state $(u^*, v^*)$ of the homogenized problem \eqref{eq:syst-hom-rd} with $\sigma$, $r_u$, $r_v$, $\kappa_u$, $\kappa_v$, $\mu_u$, $\mu_v$ replaced by $\overline{\sigma}^H$, $\overline{r_u}$, $\overline{r_v}$, $\overline{\kappa_u}$, $\overline{\kappa_v}$, $\overline{\mu_u}$, $\overline{\mu_v}$.
\end{lem}

\begin{proof}
	We divide the proof in three steps.
	
\vskip 3pt
	\begin{stepping}
		\step  \underbar{Uniform upper bound.}

\vskip 3pt
We first derive a uniform upper bound for $(u^\ep(t, x), v^\ep(t, x)) $ that is independent of $\ep>0$. 
This is done by slightly modifying the proof of Proposition~\ref{prop:uniform-bound}. Let $r_{\max}$, $\kappa_{\min}$ and $\overline{K}:=r_{\max}/\kappa_{\min}$ be the constants defined in \eqref{rmax-kappamin} for the coefficients of the system \eqref{eq:rapidosc}. Then these constants do not depend on $\ep$. As in \eqref{w}, 
$w^\ep(t,x):=u^\ep(t,x)+v^\ep(t,x)$ 
satisfies the inequality
\begin{equation}\label{wep}
w^\ep_t \leq \big(\sigma^\ep(x) w^\ep_{x}\big)_x+\big(r_{\max}-\kappa_{\min} w^\ep\big)w^\ep.
\end{equation}

Next let $W(t)$ be a solution of the following ODE problem:	
\[
W_t=\big(r_{\max}-\kappa_{\min} W\big)W, \quad \ W(0)=M^\ep.
\]
Fix $t_0\in\R $ arbitrarily, and let $T>0$. Then, since $w^\ep(t_0-T,x)\leq M^\ep=W(0)$, the comparison principle implies $w^\ep(t+t_0-T,x)\leq W(t)$ for all $t\geq 0$ and $x\in\R$. Setting $t=T$, we obtain
\[
w^\ep(t_0,x)\leq W(T) \quad\hbox{for all}\ \ x\in\R.
\]
The right-hand side of the above inequality converges to $r_{\max}/\kappa_{\min}$. Since $t_0$ is arbitrary, we get
\begin{equation}\label{upper-bound}
u^\ep(t,x)+v^\ep(t,x)=w^\ep(t,x)\leq \frac{r_{\max}}{\kappa_{\min}}:=\overline{K}\quad \hbox{for all}\ \ t\in\R,\ x\in\R.
\end{equation}

\vskip 3pt
\step \underbar{Uniform lower bound.}

\vskip 3pt
Here we derive a uniform lower bound for $(u^\ep(t, x), v^\ep(t, x)) $ that is independent of $\ep>0$. 
Let $\big(\lambda_1^\ep, (\varphi^\ep(x)>0,\psi^\ep(x)>0)\big) $ be the principal eigenpair associated with the eigenproblem:
		\begin{equation}\label{eigenproblem-lower-bound}
			\begin{system}
				\relax &(\sigma^\ep(x)\varphi^\ep_{x})_x+(r_u^\ep(x)-\mu_u^\ep(x))\varphi^\ep(x)+\mu_v^\ep(x)\psi^\ep(x) =\lambda_1^\ep\varphi^\ep(x)\\
				\relax &(\sigma^\ep(x)\psi^\ep_{x})_x+ \mu_u^\ep(x)\varphi^\ep(x)+(r^\ep_v(x)-\mu_v^\ep(x))\psi^\ep(x) =\lambda_1^\ep\psi^\ep(x),
			\end{system}
		\end{equation}
		under the $\ep$-periodic boundary conditions, and normalized  
		as $\sup_{x\in\R}(\varphi^\ep(x)+\psi^\ep(x))=1$. 
		
		Next let ${\mathcal F}^\ep(u,v)$, ${\mathcal G}^\ep(u,v)$ denote the right-hand side of the system \eqref{eq:rapidosc}, namely
		\[
		\begin{system}
			\relax & {\mathcal F}^\ep(u,v):=(\sigma^\ep(x) u_{x})_x+(r_u^\ep(x)-\kappa_u^\ep(x)(u+v))u+\mu_v^\ep(x)v-\mu_u^\ep(x)u, \\
			\relax & {\mathcal G}^\ep(u,v):=(\sigma^\ep(x) v_{x})_x+(r_v^\ep(x)-\kappa_v^\ep(x)(u+v))v+\mu_u^\ep(x)u-\mu_v^\ep(x)v.
		\end{system}		
		\]
		Then it is easily seen that, for any constant $\alpha>0$, 
		\begin{equation}\label{FG(alpha)}
		\begin{system}
		\relax & {\mathcal F}^\ep(\alpha\varphi^\ep,\alpha\psi^\ep)=\big(\lambda_1^\ep -\alpha\kappa_u^\ep(x)(\varphi^\ep+\psi^\ep)\big)\varphi^\ep\geq \big(\lambda_1^\ep -\alpha\kappa_u^\ep(x)\big)\varphi^\ep ,\\
		\relax & {\mathcal G}^\ep(\alpha\varphi^\ep,\alpha\psi^\ep)=\big(\lambda_1^\ep -\alpha\kappa_v^\ep(x))(\varphi^\ep+\psi^\ep)\big)\psi^\ep\geq\big(\lambda_1^\ep -\alpha\kappa_v^\ep(x)\big)\psi^\ep.
		\end{system}
		\end{equation}
		
		Now we claim that the following inequalities hold:
		\begin{equation}\label{lower-bound}
		u^\ep(t,x)\geq \underline{\alpha}^\ep \varphi^\ep(x),\ \ v^\ep(t,x)\geq \underline{\alpha}^\ep\psi^\ep(x)
		\quad \hbox{for all}\ \ t\in\R,\; x\in\R,
		\end{equation}
		where 
		$\underline{\alpha}^\ep=\min(\lambda_1^\ep K_1,K_2)$ 
		and 
		\[
		K_1=\frac{{\c1}}{\max\big(\max_{y\in[0,1]}\kappa_u(y),\max_{y\in[0,1]}\kappa_v(y)\big)},
		\ \  
		K_2=\min\left(\min_{y\in[0,1]}\frac{\mu_v(y)}{\kappa_u(y)}, \min_{y\in[0,1]}\frac{\mu_u(y)}{\kappa_v(y)}\right).
		\]
		In order to prove \eqref{lower-bound}, we define
		\begin{equation*}
			\alpha^*:=\sup\left\{\alpha>0 \,:\, \alpha\varphi^\ep(x)\leq u^\ep(t, x),\,\alpha\psi^\ep(x)\leq v^\ep(t, x)\ \,(\forall (t,x)\in \R\times\R)\right\}.
		\end{equation*}		
		Since $u^\ep(t, x), v^\ep(t, x)$ are bounded from below, $\alpha\varphi^\ep(x)\leq u^\ep(t, x)$, $\alpha\psi^\ep(x)\leq v^\ep(t, x)$ if $\alpha>0$ is sufficiently small, therefore the quantity $\alpha^*$ is well-defined. All we have to show is that $\alpha^*\geq \underline{\alpha}^\ep$. 

We argue by contradiction. Suppose that $\alpha^*< \underline{\alpha}^\ep$. By the definition of $\alpha^*$, we have $\alpha^*\varphi^\ep(x)\leq u^\ep(t, x),$ and $\alpha^*\psi^\ep(x)\leq v^\ep(t, x)$ for all $(t,x)\in\R\times\R$, and there exists a sequence $(t_n,x_n)\in\R\times\R$ such that either $u^\ep(t_n, x_n)-\alpha^*\varphi^\ep(x_n)\to 0\;(n\to\infty)$ or $v^\ep(t_n, x_n)-\alpha^*\psi^\ep(x_n)\to 0\;(n\to\infty)$. 
Let $m_n\,(n\in{\mathbb N})$ be the integers such that $m_n\ep\leq x_n<(m_n+1)\ep\ (n=1,2,3,\ldots)$. Replacing by a subsequence if necessary, we may assume that $x_n-m_n\ep\to x^*$ as $n\to\infty$ for some $x^*\in[0,\ep]$ and that $\big(u^\ep(t+t_n,x+m_n \ep),v^\ep(t+t_n, x+m_n \ep)\big)$ converges locally uniformly to an entire solution $\big(U^\ep(t,x),V^\ep(t,x)\big)$ of \eqref{eq:rapidosc} as $n\to\infty$. By the construction, $\big(U^\ep(t,x),V^\ep(t,x)\big)$ satisfies
\[
\begin{split}
& \alpha^*\varphi^\ep(x)\leq U^\ep(t, x), \ \ \alpha^*\psi^\ep(x)\leq V^\ep(t, x)\ \ \hbox{for all}\ \  (t,x)\in\R\times\R,\\ 
& \quad \hbox{and}\ \ \alpha^*\varphi^\ep(x^*)= U^\ep(0, x^*)\ \ \hbox{or}\ \  \alpha^*\psi^\ep(x^*)\leq V^\ep(0, x^*).
\end{split}
\]		
This, however, contradicts the strong maximum principle, since $\alpha^*\leq 
\lambda_1^\ep K_1$ and \eqref{FG(alpha)} imply that $(\alpha^*\varphi^\ep, \alpha^*\psi^\ep)$ is a subsolution of \eqref{eq:rapidosc}, and that $\alpha^*<K_2$ implies that $(\alpha^*\varphi^\ep, \alpha^*\psi^\ep)$ lies in the interior of the cooperative zone defined in \eqref{cooperative-zone}. This contradiction proves \eqref{lower-bound}. 

It remains to derive from \eqref{lower-bound} a uniform lower bounded for $u^\ep(t,x), v^\ep(t,x)$ that is independent of $\ep>0$. 
First we remark that \eqref{eigenproblem-lower-bound} is a special case of the eigenproblem treated in Lemma~\ref{lem:homogenization-eigenproblem} for $\lambda=0$, so the above eigenvalue $\lambda_1^\ep$ coincides with $k^\ep(0)$ in Lemma~\ref{lem:homogenization-eigenproblem}. Therefore $\lambda_1^\ep=k^\ep(0)\to k^0(0)$ as $\ep\to 0$, and, by \eqref{lambda_A=lambda-per}, $k^0(0)=\lambda_A$, where $\lambda_A$ denotes the largest eigenvalue of the matrix $A$ in \eqref{matrixA} with the entries $\overline{r_u}$, $\overline{r_v} $, $\overline{\mu_u}$ and $\overline{\mu_v}$. Consequently 
\[
\underline{\alpha}^\ep\to \min\left(\lambda_A K_1, K_2\right)\quad\hbox{as}\ \ \ep\to 0.
\]
Note also that, since $\varphi^\ep, \psi^\ep$ that appear in \eqref{lower-bound} 
are normalized by the condition $\max(\varphi^\ep(x)+\psi^\ep(x)) = 1$,
it is clear that
\begin{equation*}
    \varphi^\ep(x) \to  \dfrac{\varphi^0}{\varphi^0+\psi^0} \ \ \hbox{and}\ \ 
    \psi^\ep(x) \to  \dfrac{\psi^0}{\varphi^0+\psi^0} \ \ \hbox{as} \ \ \ep\to 0, \ \ \hbox{uniformly on}\ \ \R.
\end{equation*}
Combining these, together with \eqref{lower-bound}, we see that $(u^\ep, v^\ep)$ is uniformly bounded below by a positive constant that is independent of $\ep$, for all sufficiently small $\ep>0$.

\vskip 3pt
		\step \underbar{Convergence of $(u^\ep, v^\ep)$.}
\vskip 3pt
		We first remark that, since $(u^\ep, v^\ep)$ is uniformly bounded, the classical estimates for parabolic equations in divergence form with discontinuous coefficients (see {\it e.g.} \cite[Chapter III Theorem 10.1]{Lad-Sol-Ura-68}) imply that $(u^\ep, v^\ep)$ is locally uniformly bounded in $C^\alpha (\mathbb R\times\mathbb R)$, {\it i.e.} for any  $T>0$ an $R>0$ there exists $C>0$ (independent of $\ep$) such that 
		\begin{equation*}
			\max\left(\Vert u^\ep\Vert_{C^\alpha([-T, T]\times [-R, R])}, \Vert v^\ep\Vert_{C^\alpha([-T, T]\times [-R, R])}\right) \leq C.
		\end{equation*}
		Then the diagonal argument allows us to extract a sequence $\ep_n\to 0$ along which $(u^{\ep}, v^\ep)$ converges locally uniformly in $C^{\alpha/2}(\mathbb R^2)$ to a limit $(u, v)$. It is then classical that $(u, v)$ satisfies weakly: 
		\begin{equation}\label{eq:homogenized-parabolic}
			\begin{system}
				\relax &u_t=\overline{\sigma}^H u_{xx}+(\overline{r_u}-\overline{\kappa_u}(u+v))u+\overline{\mu_v}v-\overline{\mu_u}u \\
				\relax &v_t=\overline{\sigma}^H v_{xx}+(\overline{r_v}-\overline{\kappa_v}(u+v))v+\overline{\mu_u}u-\overline{\mu_v}v.
			\end{system}
		\end{equation} 
		Let us explain briefly how to obtain \eqref{eq:homogenized-parabolic} rigorously. Since $u^\ep$ and $v^\ep$ converge locally uniformly to their limit, it is also the case for  $u^\ep(u^\ep+v^\ep)$ and $v^\ep(u^\ep+v^\ep)$; therefore, except for $(\sigma^\ep u^\ep_x)_x$ and $(\sigma^\ep v^\ep_x)_x$, all the terms on the right-hand side of \eqref{eq:rapidosc} converge weakly to the corresponding term in the homogenized equation \eqref{eq:homogenized-parabolic}. 
		To show the convergence of $(\sigma^\ep u^\ep_x)_x$  to $\overline{\sigma}^H u_{xx}$, 
		let us fix $\phi\in C^\infty_0(\R)$ and define $u^\ep(\phi):=\int_{\R} u^\ep(t,x)\phi(t) \dd t$. Then
		\begin{equation*}
		    -\big(\sigma^\ep\big(u^\ep (\phi)\big)_x\big)_x = (r_u^\ep - \mu_u^\ep)u^\ep(\phi)\kappa_u^\ep u(\phi) - \kappa_u^\ep u^\ep\big(\phi(u^\ep+v^\ep)\big) + \mu_v v^\ep(\phi) + u^\ep(\phi_t).
		\end{equation*} 
		In particular, $\xi^\ep(\phi)=\sigma^\ep \big(u^\ep(\phi)\big)_x$ is bounded in $H^1_{loc}$ independently of $\ep$; hence up to the extraction of a subsequence,  $\xi^\ep(\phi)$ converges strongly in $L^2_{loc}$ (and weakly in $H^1_{loc}$) to some $\xi^0(\phi)$, and $u^\ep(\phi)_x =\frac{1}{\sigma^\ep}\xi^\ep(\phi) \rightharpoonup (\overline{\sigma}^H)^{-1} \xi^0(\phi)$ which shows that $\xi^0(\phi)=\overline{\sigma}^H u^0(\phi)_x$. This establishes the first line of \eqref{eq:homogenized-parabolic} rigorously. 
The second line can be treated similarly. 
Hence $(u, v)$ satisfies \eqref{eq:homogenized-parabolic} in a weak sense.
Parabolic regularity then implies that $(u(t,x), v(t,x))$ is in fact a classical entire solution to \eqref{eq:syst-hom-rd}. Since $(u(t,x), v(t,x))$ is bounded from below (by Step 1), Theorem \ref{thm:entire-sol-hom}  shows that $u(t,x)\equiv u^*$ and $v(t,x)\equiv v^*$. 
Finally, since the limit $(u(x), v(x))\equiv (u^*,v^*)$ does not depend on the choice of the sequence $\ep_n\to 0$, we have $(u^\ep,v^\ep)\to (u^*,v^*)$ as $\ep\to 0$. The lemma is proved.
	\end{stepping}
\end{proof}

Next we discuss the linearized stability of the equilibrium point $(u^*,v^*)$ of the homogeneous system. The linearized equation of \eqref{eq:syst-hom-rd} around $(u^*,v^*)$ is given in the following form: 
	\begin{equation}\label{eq:syst-hom-linearized}
		\begin{system}
		    \relax &\varphi_t-{\sigma}\varphi_{xx} = ({r_u}-{\mu_u})\varphi + {\mu_v}\psi  - {\kappa_u}(2u^*+v^*)\varphi - {\kappa_u}u^*\psi , & t>0, x\in\R, \\ 
			\relax &\psi_t-{\sigma} \psi_{xx}={\mu_u}\varphi + ({r_v}-{\mu_v})\psi - {\kappa_v}v^*\varphi - {\kappa_v}(u^*+2v^*)\psi, & t>0, x\in\R,\\
			&\varphi(t=0, x)=\varphi_0(x), ~\psi(t=0, x)=\psi_0(x), & x\in\R.
		\end{system}
    \end{equation}

\begin{lemma}[Linear stability of the equilibrium]\label{lem:linear-stability}
Let the assumptions of Theorem~\ref{thm:rapidosc} hold. Then any solution $(\varphi, \psi)$ of the linear parabolic system \eqref{eq:syst-hom-linearized} with $\big(\varphi_0(x), \psi_0(x)\big)\in BUC(\R)^2$, converges uniformly to zero as $t\to+\infty$:
	\begin{equation*}
	    \lim_{t\to+\infty} \max\big(\Vert \varphi(t, \cdot)\Vert_{\infty}, \Vert \psi(t, \cdot)\Vert_{\infty}\big)=0.
	\end{equation*}
\end{lemma}

\begin{proof}
		We show that the spectrum of the linearized operator is included in the negative complex plane. 
		The linearized operator around $(u^*,v^*)$ is given in the following form:
		\begin{equation*}
			\mathcal{A}
			\begin{pmatrix} 
				\varphi \\ \psi 
			\end{pmatrix}
			:=
			\begin{pmatrix}
				\sigma \varphi_{xx} \\ \sigma \psi_{xx} 
			\end{pmatrix}
			+
			\begin{pmatrix}
				(r_u-\mu_u-2\kappa_u u^* - \kappa_u v^*)\varphi + (\mu_v-\kappa_u v^*)\psi \\
				(\mu_u-\kappa_vu^*)\varphi+(r_v-\mu_v-\kappa_vu^* - 2\kappa_vv^*)\psi
			\end{pmatrix}.
		\end{equation*}
		We regard $\mathcal A$ as an unbounded operator acting on $(\varphi, \psi)\in BUC(\mathbb R)^2$. The operator $\mathcal A$ is sectorial and generates an analytic semigroup on $BUC(\mathbb R)^2$, as  a bounded perturbation of the unbounded operator $(\sigma \partial_{xx}, \sigma\partial_{xx})^T$ (acting on $D(\mathcal{A})=C^2_{BUC}(\mathbb R)^2$), which is sectorial and  generates an analytic semigroup on $BUC(\mathbb R)^2$ \cite[Corollary 3.1.9 p. 81]{Lun-95}.

		Let $\lambda\in \mathbb C$ and  $(g, h)\in BUC(\mathbb R)^2$ be given, and consider the resolvent equation
		\begin{equation}\label{eq:resolvent}
			(\lambda I-\mathcal A)\parmatrix{\varphi \\ \psi} = \parmatrix{g \\ h}. 
		\end{equation}
		The set of solutions of the equation \eqref{eq:resolvent} can be computed explicitly by the variation of constants formula. More precisely, we let $Y(x)=(\varphi, \varphi_x, \psi, \psi_x)^T$  and rewrite \eqref{eq:resolvent} as an ODE on $\mathbb R^4$:
		\begin{equation*}
			\frac{\dd}{\dd x}Y(x) = 
			\begin{pmatrix}
				0 & 1 & 0 & 0 \\ 
				\sigma^{-1}(\lambda-a) & 0 & -\sigma^{-1}b & 0\\
				0 & 0 & 0 & 1 \\
				-\sigma^{-1}c & 0 & \sigma^{-1}(\lambda-d) & 0
			\end{pmatrix}Y - \parmatrix{0 \\ g \\ 0 \\ h} 
			=:B_\lambda Y(x)+Z(x),
		\end{equation*}
		where $a$, $b$, $c$, $d$ are the constants introduced in Lemma \ref{lem:stat-ode-stability} to denote the coefficients of the Jacobian matrix of the nonlinearity at the equilibrium point: 
		\begin{align*}
			a&:=r_u-\mu_u-2\kappa_uu^*-\kappa_u v^*=-\left(\kappa_u u^*+\mu_v\frac{v^*}{u^*}\right)<0, & b &:=\mu_v-\kappa_u u^*, \\
			d&:=r_v-\mu_v-\kappa_v u^*-2\kappa_vv^*=-\left(\kappa_vv^*+\mu_u\frac{u^*}{v^*}\right)<0, & c &:=\mu_u-\kappa_v v^*.
		\end{align*}
		We first investigate  the bounded eigenvectors  of the ordinary differential equation  $Y'=B_\lambda Y$, which constitute the point spectrum of $\mathcal{A}$, $\sigma_P(\mathcal{A})$. These correspond to the imaginary eigenvalues of the matrix $B_\lambda$, {\it i.e.} the imaginary roots of the polynomial 
		\begin{equation*}
			\chi_\lambda(X):=X^4+\sigma^{-1}\big(a +d-2\lambda\big)X^2+\sigma^{-2}\big(\lambda^2-(a+d)\lambda +ad - bc\big).
		\end{equation*}
		We show that  $\sigma_P(\mathcal{A})$  is contained in the half-plane $\Re(z)\leq -\omega$ for $z\in\mathbb C$, where
		\begin{equation}\label{eq:omega}
			\omega:= -\frac{a+d}{2}>0.
		\end{equation}  
		Indeed, investigating the values taken by $\chi_\lambda(iX)$ for real values of $X$, we find that 
		\begin{equation*}
			\chi_\lambda(iX)=X^4-\sigma^{-1}\big((a + d)-2\lambda\big)X^2+\sigma^{-2}\big(\lambda^2-(a+d)\lambda +ad - bc\big).
		\end{equation*}
		Since $a<0$, $d<0$ and $ad-bc>0$ (see Lemma \ref{lem:stat-ode-stability}), we immediately see that $\chi_\lambda(iX)>0$ if $\lambda $ is real and $\lambda \geq \frac{a+d}{2}$. If $\Im(\lambda)\neq 0$, we remark that 
		\begin{equation*}
			\Im(\chi_\lambda (iX))=\Im(\lambda)\left[2\sigma^{-1}X^2+\sigma^{-2}(2\Re(\lambda)-(a+d))\right], 
		\end{equation*}
		therefore if $\Re(\lambda)>\frac{a+d}{2}$ the right-hand side is positive and the polynomial $\chi_\lambda(iX)$ cannot have a real root. 

		When $\lambda \in\mathbb C\backslash \sigma_P(\mathcal{A})$ then $Y$ is uniquely determined and depends continuously on $Z$. Indeed,  the set of solutions to the equation $Y'=B_\lambda Y+Z$ can be determined by the {variation of constants formula} 
		\begin{equation}\label{eq:hom-res}
			Y(x)=e^{xB_\lambda} Y_0+\int_0^xe^{(x-s)B_\lambda}Z(s)\dd z, 
		\end{equation}
		for arbitrary $Y_0\in \mathbb R^4$. Then there exists a unique choice of $Y_0$ such that $Y(x)$ remains bounded on $\mathbb R$. Indeed, the discriminant of $\chi_\lambda(X)$ considered as a second-order polynomial in $X^2$ is
		\begin{equation*}
		    D(\lambda):=\sigma^{-2}\big(a + d -2 \lambda \big)^2 - 4\sigma^{-2}\big(\lambda^2-(a+d)\lambda + ad-bc\big)=\sigma^{-2}\big((a+d)^2-ad + bc\big)>0,
		\end{equation*}
		thus is independent of $\lambda$ and positive. Hence $\chi_\lambda(X)$ has four distinct roots; by reducing $B_\lambda$ to a diagonal matrix,  $Y_0$ can be computed explicitly and the resulting $Y(x)=\big(\lambda-\mathcal{A}\big)^{-1}(Z)$ depends continuously on $Z(x)$. Since the computations are relatively long and classical, we omit them for the sake of brevity. 
		In particular, the spectrum is equal to the point spectrum $\sigma(\mathcal{A})=\sigma_P(\mathcal{A})$ and the spectral bound of $\mathcal{A}$ satisfies $s(\mathcal{A})\leq -\omega$.

		To complete our first Step we remark that  $e^{t\mathcal A}$ can be computed by the Dunford-Taylor integral
		\begin{equation*}
			e^{t\mathcal A}=\frac{1}{2i\pi}\int_{\Gamma}e^{\lambda t}(\lambda I-\mathcal A)^{-1}\dd \lambda ,
		\end{equation*}
		where $\Gamma$ is a curve joining a straight line $\{\rho e^{-i\theta}, \rho>0\}$ for some $\theta\in\left[\frac{\pi}{2}, \pi\right)$ to the straight line $\{\rho e^{-i\theta}:\rho>0\}$,  oriented so that $\Im (\lambda) $ increases on $\Gamma$, and such that $\sigma(\mathcal{A})$ is included in the left connected component of $\mathbb{C}\backslash \Gamma$. From the above computations it appears that $\Gamma $ can be chosen such that $\Re(\lambda)\leq -\frac{\omega}{2}$ (where $\omega$ is given by \eqref{eq:omega}) for all $\lambda \in \Gamma$, in which case
		\begin{equation*}
			e^{t\mathcal A} = e^{-\frac{\omega}{2} t}\cdot\frac{1}{2i\pi}\int_{\Gamma} e^{\left(\lambda + \frac{\omega}{2}\right) t}(\lambda-\mathcal A)^{-1}\dd \lambda, 
		\end{equation*}
		therefore
		\begin{align*}
			\Vert e^{t\mathcal A}\Vert_{BUC(\mathbb R)^2}&\leq e^{-\frac{\omega}{2}t}\cdot\frac{1}{2\pi} \int_{\Gamma} e^{-(\Re(\lambda)+\frac{\omega}{2})t}\Vert (\lambda-\mathcal A)^{-1}\Vert_{\mathcal L(BUC(\mathbb R)^2)}\dd\lambda\\ 
			&\leq Ce^{-\frac{\omega}{2}t},
		\end{align*}
		for all $t>0$, where $C$ depends only on $\mathcal A$ and $\omega$.
	This completes the proof of Lemma \ref{lem:linear-stability}.
\end{proof}

\begin{lem}[Uniqueness of rapidly oscillating entire solution]\label{lem:rapid-osc-unique}
	Let the assumptions of Theorem \ref{thm:rapidosc} hold. Then there exists $\overline{\ep}$ such that for any $0<\ep\leq\overline{\ep}$, the system \eqref{eq:rapidosc} possesses an entire solution $(u^\ep(t,x), v^\ep(t,x))$ that is bounded from above and from below by positive constants. Moreover, such an entire solution is unique. Hence it is a stationary solution and is $\ep$-periodic. 
\end{lem}

\begin{proof}
    We first show the existence. Let $(u(t,x),v(t,x))$ be a solution of \eqref{eq:rapidosc} with nonnegative nontrivial initial data $\big(u_0(x),v_0(x)\big)$ and choose a constant $M>0$ such that $u_0(x)+v_0(x)\leq M$ for all $x\in\R$.  Then from the inequality \eqref{wep}, we see that
\[
u(t,x)+v(t,x)\leq \max\left(\frac{r_{\max}}{\kappa_{\min}}, M\right):=M'\quad\hbox{for all}\ \ t\geq 0,\ x\in\R.
\]
Therefore $(u(t,x),v(t,x))$ is bounded from above. Next, by Theorem~\ref{thm:hairtrigger}, there exists $\eta>0$ such that \eqref{eq:hairtrigger-below} holds. Consequently, by choosing a sequence $t_n\to+\infty$ appropriately, the family of functions $(u(t+t_n,x),v(t+t_n,x))$ converges to an entire solution $\big(u_\infty(t,x),v_\infty(t,x)\big)$ of \eqref{eq:rapidosc} satisfying $\eta\leq u_\infty(t,x), v_\infty(t,x) \leq M'$ for all $t\in\R, x\in\R$. 

Next we prove the uniqueness. Assume by contradiction that there exists a sequence $\ep_n>0$ and  two sequences of bounded nonnegative nontrivial entire solutions $(u_1^{\ep_n}(t,x), v_1^{\ep_n}(t,x))\not\equiv(u_2^{\ep_n}(t, x), v_2^{\ep_n}(t, x))$ of \eqref{eq:rapidosc}. 
	Define 
	\begin{align*}
	    \delta_n&:=\max\left(\Vert u_2^{\ep_n}(t,x)-u_1^{\ep_n}(t,x)\Vert_{BUC(\mathbb R)^2},\Vert v_2^{\ep_n}(t,x)-v_1^{\ep_n}(t,x)\Vert_{BUC(\mathbb R)^2}\right),\\
		\varphi^{\ep_n}(t,x)&:=\frac{1}{\delta_n}(u_2^{\ep_n}(t,x)-u_1^{\ep_n}(t,x)), \\
		\psi^{\ep_n}(t,x)&:=\frac{1}{\delta_n}(v_2^{\ep_n}(t,x)-v_1^{\ep_n}(t,x)).
	\end{align*}
	With an appropriate shift in time and space, we may assume without loss of generality that  
	\begin{equation}\label{eq:rapid-osc-unique-norm}
		\frac{\delta_n}{2}\leq \sup_{x\in (0,L)}\big(\max(| u_2^{\ep_n}(0,x)-u_1^{\ep_n}(0,x)|, | v_2^{\ep_n}(0,x)-v_1^{\ep_n}(0,x)|)\big) \leq \delta_n.
	\end{equation}
	Then $(\varphi^{\ep_n}(t,x), \psi^{\ep_n}(t,x)) $ satisfy:
	\begin{equation*}
		\begin{system}
		    \relax &\varphi^{\ep_n}_t-\big(\sigma^{\ep_n}\varphi^{\ep_n}_x\big)_x = (r_u^{\ep_n}-\mu_u^{\ep_n})\varphi^{\ep_n} + \mu_v^{\ep_n}\psi^{\ep_n}  - \kappa_u^{\ep_n}(2u^{\ep_n}_2+v^{\ep_n}_2)\varphi^{\ep_n} - \kappa_u^{\ep_n}u^{\ep_n}_2\psi^{\ep_n}+o(1),  \\ 
			\relax &\psi^{\ep_n}_t-\big(\sigma^{\ep_n}\psi^{\ep_n}_{x}\big)_x=\mu_u^{\ep_n}\varphi^{\ep_n} + (r_v^{\ep_n}-\mu_v^{\ep_n})\psi^{\ep_n} - \kappa_v^{\ep_n}v^{\ep_n}_2\varphi^{\ep_n} - \kappa_v^{\ep_n}(u^{\ep_n}_2+2v^{\ep_n}_2)\psi^{\ep_n}+o(1), 
		\end{system}
	\end{equation*}
	where $o(1)$ denotes a remainder term that tends to $0$ as $n\to\infty$ locally uniformly with respect to $(t, x)\in\R\times\R$.
	Indeed, by virtue of Lemma \ref{lem:rapid-osc-lim}, there holds 
	\begin{equation*}
	    (u^{\ep_n}_1, v^{\ep_n}_1) \to (u^*, v^*) \ \text{and}\ (u^{\ep_n}_2, v^{\ep_n}_2) \to (u^*, v^*) \ \ \text{locally uniformly as}\ n\to+\infty.
	\end{equation*}
	Since $\varphi^{\ep_n}(t, x)$ and $\psi^{\ep_n}(t, x) $ are bounded, by following the arguments presented in Step 3 of the proof of Lemma \ref{lem:rapid-osc-lim}, the classical homogenization theory then  leads to the convergence (up to an extraction of a subsequence) of $(\varphi^{\ep_n}(t, x), \psi^{\ep_n}(t, x))$ to  $(\varphi(t, x), \psi(t, x))$ solving 
	\begin{equation*}
		\begin{system}
			\relax &\varphi_t-\overline{\sigma}^H\varphi_{xx} = (\overline{r_u}-\overline{\mu_u})\varphi + \overline{\mu_v}\psi  - \overline{\kappa_u}(2u^*+v^*)\varphi - \overline{\kappa_u}u^*\psi  \\ 
			\relax &\psi_t-\overline{\sigma}^H \psi_{xx}=\overline{\mu_u}\varphi + (\overline{r_v}-\overline{\mu_v})\psi - \overline{\kappa_v}v^*\varphi - \overline{\kappa_v}(u^*+2v^*)\psi,
		\end{system}
	\end{equation*}
	and the convergence holds locally uniformly in $(t, x)\in\R\times \R$. 
	Because of our normalization \eqref{eq:rapid-osc-unique-norm}, the limit function $\big(\varphi(t, x), \psi(t, x)\big)$ is nontrivial and bounded on $\R\times \R$, which is in contradiction with Lemma~\ref{lem:linear-stability}.  This establishes the uniqueness. 
	
	In order to prove the last claim, note that, for any $\tau\in\R$, $(u^\ep(t+\tau,x), v^\ep(t+\tau,x))$ is also an entire solution of \eqref{eq:rapidosc} with the same upper and lower bounds. The uniqueness then implies $(u^\ep(t+\tau,x), v^\ep(t+\tau,x))\equiv (u^\ep(t,x), v^\ep(t,x))$, which means that $(u^\ep(t,x), v^\ep(t,x))$ is a stationary solution. Similarly, $(u^\ep(t,x+\ep), v^\ep(t,x+\ep))$ is an entire solution since the coefficients are $\ep$-periodic. Hence $(u^\ep(t,x+\ep), v^\ep(t,x+\ep))\equiv (u^\ep(t,x), v^\ep(t,x))$. 
The proof of Lemma~\ref{lem:rapid-osc-unique} is complete. 
\end{proof}

\begin{proof}[Proof of Theorem \ref{thm:rapidosc}]
As regards Statement \ref{item:rapid-osc-unique}, the existence and uniqueness of a positive stationary solution $(u^*_\ep(x),v^*_\ep(x))$, as well as its $\ep$-periodicity, are a direct consequence of Lemma~\ref{lem:rapid-osc-unique}. The convergence $(u^*_\ep(x),v^*_\ep(x))\to (u^*,v^*)$ follows from Lemma~\ref{lem:rapid-osc-lim}. Statement \ref{item:rapid-osc-speeds} is already proved in Lemma~\ref{lem:convergence-speeds}.

    Let us prove Statement \ref{item:rapid-osc-cv}. Let $\ep>0$ be sufficiently small, so that $k^\ep(\lambda)>0$ for all $\lambda\in\R$. Then $c^*_{\ep, R}$ and $c^*_{\ep, L}$ are both positive, and \eqref{right2a} and \eqref{left2b} of Theorem~\ref{thm:hairtrigger} hold for any solution of \eqref{eq:rapidosc} with nonnegative nontrivial bounded initial data. We argue by contradiction. 
    Suppose that \eqref{uep(t,x)-to-u*ep} does not hold for some $c_1, c_2$ with $0<c_1<c^*_{\ep,L}$, $0<c_2<c^*_{\ep,R}$. Then there exist a constant $\delta>0$ and sequences $t_n\to+\infty$ and $x_n$ with $-c_1t_n\leq x_n\leq c_2 t_n$ such that
    \[
    \max\big(|u(t_n,x_n)-u^*_\ep(x_n)|, |v(t_n,x_n)-v^*_\ep(x_n)|\big)\geq \delta\quad \hbox{for}\ \ n=1,2,3,\ldots.
    \]
    Without loss of generality, we may assume that $|u(t_n,x_n)-u^*_\ep(x_n)|\geq \delta$ for $n=1,2,3,\ldots$ 

    Let $m_n\,(n\in{\mathbb N})$ be the integers such that $m_n\ep\leq x_n<(m_n+1)\ep\ (n=1,2,3,\ldots)$. Replacing by a subsequence if necessary, we may assume that $x_n-m_n\ep\to x^*$ as $n\to\infty$ for some $x^*\in[0,\ep]$ and that $\big(u(t+t_n,x+m_n\ep),v(t+t_n, x+m_n\ep)\big)$ converges locally uniformly to an entire solution $(U(t,x),V(t,x))$ of \eqref{eq:rapidosc} as $n\to\infty$. By the construction, we have
    \begin{equation}\label{U-u*ep(x*)}
    |U(0,x^*)-u^*_\ep(x^*)|\geq \delta.
    \end{equation}
    Now we choose constant $\tilde{c_1}, \tilde{c_2}$ satisfying $0<c_1<\tilde{c_1}<c^*_{\ep,L}$, $0<c_2<\tilde{c_2}<c^*_{\ep,R}$. Then, by Theorem~\ref{thm:hairtrigger}, we have
    \[
    \liminf_{t\to\infty}\min_{-\tilde{c_1}t\leq x\leq \tilde{c_2}t}\min\big(u(t,x),v(t,x)\big)\geq \eta
    \]
    for some constant $\eta>0$. In view of this and the fact that $-c_1t\leq x_n\leq c_2t\,(n\in{\mathbb N})$, we see that 
    \[
    U(t,x)\geq \eta, \ \ V(t,x)\geq \eta\quad\hbox{for all}\ \ (t,x)\in\R\times\R.
    \]
    Then, by Lemma~\ref{lem:rapid-osc-unique}, $(U(t,x),V(t,x))$ coincides with the stationary solution $(u^*_\ep(x),v^*_\ep(x))$, but this contradicts \eqref{U-u*ep(x*)}. This proves Statement \ref{item:rapid-osc-cv}. The proof of Theorem~\ref{thm:rapidosc} is complete.
\end{proof}

\printbibliography


\end{document}